\documentclass[a4paper,11pt]{article}
\usepackage{stmaryrd}
\usepackage{graphicx}
\usepackage{amsfonts}
\usepackage{amsmath,amsthm}
\usepackage{amssymb}
\usepackage{mathrsfs}
\usepackage{indentfirst}
\usepackage{titlesec}
\usepackage{caption2}
\usepackage{color}
\usepackage[normalem]{ulem}
\usepackage{algorithm}
\usepackage{algorithmic}
\usepackage{lineno}

\usepackage[english]{babel}
\usepackage{cite}
\usepackage{bbm}

\usepackage[colorlinks, linkcolor=blue]{hyperref}

\newtheorem{theorem}{Theorem}
\newtheorem{definition}{Definition}

\newtheorem{assumption}{Assumption}
\newtheorem{proposition}{Proposition}[section]
\newtheorem{remark}{Remark}[section]
\newtheorem{lemma}{Lemma}[section]
\newtheorem{corollary}{Corollary}[section]

\numberwithin{equation}{section}

\usepackage{fancyhdr}
\begin{document}


\title{Wave propagation for reaction-diffusion equations on infinite random trees}

\author{
Wai-Tong (Louis) Fan
\thanks{Mathematics Department, Indiana University. Email: \texttt{waifan@iu.edu}} \ ,
Wenqing Hu
\thanks{Department of Mathematics and Statistics, Missouri University of Science and Technology
(formerly University of Missouri, Rolla). Email: \texttt{huwen@mst.edu} Email Corresponence to: \texttt{huwenqing.pku@gmail.com}} \ ,
Grigory Terlov
\thanks{Department of Mathematics, University of Illinois Urbana-Champaign. Email: \texttt{gterlov2@illinois.edu}}
}

\date{}

\maketitle

\begin{abstract}
The asymptotic wave speed for FKPP type reaction-diffusion equations on a class of
infinite random metric trees  are considered.
We show that a travelling wavefront emerges, provided that the reaction rate is large enough.
The wavefront travels at a speed that can be quantified via a variational formula
involving the random branching degrees $\vec{d}$ and the random branch lengths $\vec{\ell}$ of
the tree $\mathbb{T}_{\vec{d}, \vec{\ell}}$. This speed is
\textit{slower than} that of the same equation on the real line $\mathbb{R}$, and we
estimate this slow down in terms of $\vec{d}$ and $\vec{\ell}$. 
The key idea is to project the Brownian motion on the tree onto a one-dimensional
axis along the direction of the wave propagation. The projected process is
a multi-skewed Brownian motion, introduced by Ramirez \cite{Ram11},
with skewness and interface sets that encode the metric structure
$(\vec{d}, \vec{\ell})$ of the tree.
Combined with analytic arguments based on the Feynman-Kac formula, this idea
connects our analysis of the wavefront propagation to the large deviations principle (LDP)
of the multi-skewed Brownian motion with random skewness and random interface set.
Our LDP analysis  involves delicate estimates
 for an infinite product of $2\times 2$ random matrices
parametrized by $\vec{d}$ and $\vec{\ell}$ and
for hitting times of a random walk in random environment. 
\end{abstract}

\textit{Keywords}: reaction-diffusion equations, wavefront propagation, multi-skewed Brownian motion,
large deviations principle, product of random matrices, random walk in random environment.

\textit{2010 Mathematics Subject Classification Numbers}: 35K57, 35A18, 60J60, 60K37, 60F10.

\newpage

\tableofcontents\newpage

\section{Introduction}\label{Sec:Intro}

The FKPP equation, named after Fisher \cite{Fisher} and Kolmogorov, Petrovski, and Piskunov \cite{KPP}, is one of the simplest reaction-diffusion equation which can exhibit traveling wave solutions. This equation arises  in ecology, population biology, chemical reactions, plasma physics and other disciplines. It describes the dynamics of a certain quantity $u(t,x)$ at time $t$ and location $x$, written as
\begin{equation}\label{Eq:FisherKPP}
\left\{\begin{array}{l}
\dfrac{\partial u}{\partial t}(t,x)=\dfrac{1}{2}\dfrac{\partial^2 u}{\partial x^2}(t,x)+f\big(u(t,x)\big) \ ,
\\
u(0,x)=u_0(x) \ ,
\end{array}\right.
\end{equation}
where the reaction function $f(u)=\beta u(1-u)$
for some constant $\beta>0$ which will be called the
\textit{reaction rate} throughout this paper.

The \textit{asymptotic speed} of the wavefront formed by \eqref{Eq:FisherKPP} can be defined as
a positive real number $\alpha^*>0$ such that for any $h>0$,
\begin{equation}\label{Eq:Intro:WaveSpeedRealLine}
\lim_{t \to \infty}\sup_{x>(\alpha^*+h)t}u(t,x)=0 \text{ and } \lim_{t \to \infty}\inf_{x<(\alpha^*-h)t}u(t,x)=1 \ .
\end{equation}
It is well known from \cite{Fisher} and \cite{KPP} that on the real line  $\mathbb{R}$, for step-like initial data including the Heaviside function $u_0(x)=\mathbf{1}_{x\leq 0}$, the solution to \eqref{Eq:FisherKPP} forms a wavefront that propagates through the real line $\mathbb{R}$ with asymptotic speed $\sqrt{2\beta}$ \footnote{If the diffusion term $\dfrac{1}{2}\dfrac{\partial^2 u}{\partial x^2}$
in the equation \eqref{Eq:FisherKPP} becomes
$\dfrac{D}{2}
\dfrac{\partial^2 u}{\partial x^2}$ for a general diffusion constant $D>0$, then it is easy to see from
a spacial rescaling $x\rightarrow \dfrac{x}{\sqrt{D}}$ that the wave speed is $\sqrt{2D\beta}$. Thus throughout
the paper we stick to the case $D=1$.}.
Freidlin in \cite{FreidlinGreenBook} presents an elegant argument to prove this statement that uses the Feynman-Kac formula to connect the asymptotic speed with the large deviations principle (LDP) of the Brownian motion on the real line.

For simplicity and to make the arguments more intuitive, we focus on this classical case  with diffusion coefficient $D=1$ throughout this paper, but on trees rather than on $\mathbb{R}$. By the same arguments with simple modifications, our results can readily be extended to  the general FKPP-type case, in which
$f$ is a continuous function on $[0,1]$,
$f(0)=f(1)=0$, $f'(0)=\sup_{u\in (0,1)}f(u)/u$ and  $f(u)>0$ for $u\in (0,1)$.
We expect the results will be the same when  $\beta$ is replaced by $f'(0)$, and the asymptotic speed will be multiplied by $\sqrt{D}$.

While asymptotic speed of FKPP wavefront on the real line  is well studied,  much less is known about the formation and the speed of wave propagations in different environments such as a network. These are challenging problems because the topological and metric structure of the underlying space interacts with the diffusion-reaction mechanism.

Nonetheless, it is of both practical and theoretical interests to consider equation \eqref{Eq:FisherKPP}
on geometric structures other than a line. Such equations arise as
scaling limits of reaction diffusion-equations in two-dimensional domains
(see \cite{FreidlinHu13}, \cite{FreidlinHu13Motor}, \cite{CF17}, \cite{CF18})
and of interacting particle systems (see \cite{MR3582808}, \cite{Fan17}), which provide descriptions of the
effective dynamics of much more complex systems. For example,
in \cite{FreidlinHu13} the authors considered a reaction-diffusion equation
on a narrow random channel. The channel consists of
a main track and random ``wings" added to it.
As the channel width becomes thin, it converges to a tree-like structure with many short branching
edges added to the real line. Such a tree is in a sense a ``noisy" real line because
it is the real line randomly adding short edges to it. Making use of LDP for diffusion processes in random environment, the authors of
\cite{FreidlinHu13} derived a formula for the wave speed in this case.

\medskip

\noindent
{\bf Our results. }
In this work, we consider the propagation of waves given by the FKPP  equation
\eqref{Eq:FisherKPP} on an infinite random tree $\mathbb{T}_{\vec{d}, \vec{\ell}}$
that is called \emph{symmetric $\vec{d}$ regular with branch lengths $\vec{\ell}$}
(the precise definitions are given
in Section \ref{Sec:TreeStructure:Subsection:Tree}).
Here the random branching degrees $\vec{d}=(d_i)_{i\in \mathbb{Z}_+}$ (we set $\mathbb{Z}_+=\{i\in \mathbb{Z}, i\geq 0\}$) is such that \footnote{For simplicity of presentation we assume $d_0=2$.
However, our arguments work for all cases when $d_0\geq 2$ is an arbitrary fixed integer without affecting the wave speed.}
$d_0=2$  and $(d_i)_{i\geq 1}$
is an i.i.d sequence
of bounded positive integers greater or equal than $2$, such that
all vertices of $\mathbb{T}_{\vec{d}, \vec{\ell}}$ at generation $i$ have degrees equal to $d_i$ (the root $\rho$
is the node at generation $0$). 
The random branch lengths $\vec{\ell}=(\ell_i)_{i\in \mathbb{Z}_+}$ is an i.i.d.
sequence of positive real numbers that are positively bounded from above and from below, such that the edges of $\mathbb{T}_{\vec{d}, \vec{\ell}}$ between generations $i$ and $i+1$ are all of length equal to $\ell_i$.
A typical example of $\mathbb{T}_{\vec{d}, \vec{\ell}}$ is shown in the upper part of
Figure \ref{Fig:Trees}.

\begin{figure}
\centering
\includegraphics[scale=0.5]{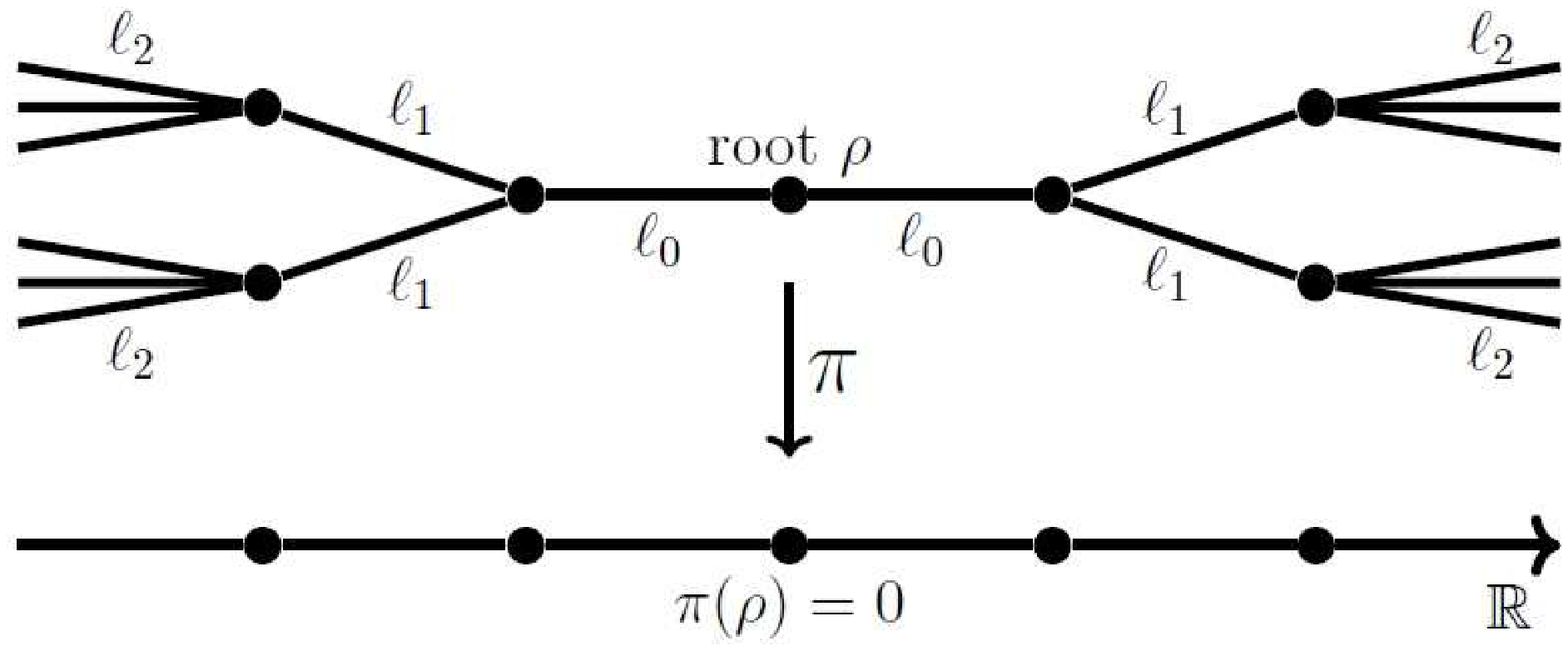}
\caption{Symmetric $\vec{d}$-regular tree with branch lengths $\vec{\ell}$, denoted by $\mathbb{T}_{\vec{d},\vec{\ell}}$, with $d_0=2$, $d_1=3$ and $d_2=4$. The projection  $\pi:\mathbb{T}_{\vec{d},\vec{\ell}}\to \mathbb{R}$ maps $x$ to the signed distance on $\mathbb{T}$ from $x$ to the root $\rho$.
Let $(B_t)_{t\geq 0}$ be a Brownian motion on $\mathbb{T}_{\vec{d},\vec{\ell}}$ and define $Y_t=\pi(B_t)$. Then the projected process $Y=(Y_t)_{t\geq 0}$ is a multi-skewed Brownian motion with skewness (see Definition \ref{Def:GSBM}) $2/3$ at $\ell_0$ and skewness $1/3$ at $-\ell_0$.}
\label{Fig:Trees}
\end{figure}

This class of random trees includes many random trees of interest, and in particular
the deterministic $d$-regular tree for $d>2$ with branch length
$\ell\in(0,\infty)$. The latter, called the \textit{constant-$(d,\ell)$ tree} in this paper, is an illuminating special case in which
all $\{\ell_i\}_{i\geq 0}$ are the same constant $\ell$ and all
$\{d_i\}_{i\geq 1}$ are equal to the same constant $d>2$. For this particular case, if $d=2$ we further obtain the \textit{degenerate case} $\mathbb{T}=\mathbb{R}$. 

Unlike the real line $\mathbb{R}$, the random tree $\mathbb{T}_{\vec{d}, \vec{\ell}}$ is in general a one-dimensional metric space
with singularities at its vertices (nodes at different generations).
Thus equation \eqref{Eq:FisherKPP}, when
considered on the tree $\mathbb{T}_{\vec{d}, \vec{\ell}}$\ , should also
be equipped with boundary conditions  at the vertices.
Here we put symmetric gluing conditions at each of the
vertices  of the tree $\mathbb{T}_{\vec{d}, \vec{\ell}}$\ , so that
the sum of the outward derivatives of the solution $u$ at each vertex of the tree is equal to $0$.
This specifies that the flow-in equals flow-out of mass at each vertex.

We also impose a step-like initial condition $u(0,x)=u_0(x)=\mathbf{1}_{U_0}(x)$ such that
$U_0$ is a symmetric subset of the set of all of the $d_0$ edges attached to
the root $\rho$.
Intuitively, these symmetric initial and boundary conditions will guarantee
that as time $t$ evolves, the solution $u(t,x)$ will also be symmetric
with respect to all edges of the tree $x\in \mathbb{T}_{\vec{d}, \vec{\ell}}$
that are between the same two consecutive generations.
In this way, following \eqref{Eq:Intro:WaveSpeedRealLine}, we say that a quantity $c^*>0$ is the asymptotic speed of the wavefront
formed by \eqref{Eq:FisherKPP} on the tree $\mathbb{T}_{\vec{d}, \vec{\ell}}$
if for any $h_1>0$ and $c^*>h_2>0$,
\begin{equation}\label{Eq:Intro:WaveSpeedTree}
\lim_{t \to \infty}\sup_{d_{\mathbb{T}}(x,\rho)>(c^*+h_1)t}u(t,x)=0
\text{ and }
\lim_{t \to \infty}\inf_{d_{\mathbb{T}}(x,\rho)<(c^*-h_2)t}u(t,x)=1 \ .
\end{equation}
Here $d_{\mathbb{T}}(x,\rho)$ denotes the geodesic distance  of the point $x\in \mathbb{T}_{\vec{d}, \vec{\ell}}$
to the root $\rho$, i.e., it is the length of the shortest path from $x$ to $\rho$ along the tree
$\mathbb{T}_{\vec{d}, \vec{\ell}}$. 

The main result in this paper can be stated roughly as below; the full statement is encapsulated in Theorems \ref{Theorem:WaveSpeed} and \ref{Theorem:VariationalFormulaWaveSpeed}.

\

\noindent
\textbf{Main Result.}
\textit{Let $\mathbb{T}_{\vec{d}, \vec{\ell}}$ be the random tree equipped with the aforementioned initial and boundary conditions.
There exists $\beta_c\in (0,\infty)$ such that for all $\beta\in(\beta_c,\infty)$,
as $t\rightarrow \infty$, the solution $\{u(t,x):\;t\in[0,\infty), x\in \mathbb{T}_{\vec{d}, \vec{\ell}}\}$ of equation
\eqref{Eq:FisherKPP} on $\mathbb{T}_{\vec{d}, \vec{\ell}}$
forms a wavefront on the tree.
The wavefront travels with an asymptotic speed that is less than or equal to $\sqrt{2\beta}$, with equality holds if and only if the tree degenerates to the real line $\mathbb{R}$.}

\

The above result is a direct consequence of Lemmas \ref{Lemma:proj}, \ref{Lemma:WaveSpeedProjectedEqualTree} and Theorems \ref{Theorem:WaveSpeed},
\ref{Theorem:VariationalFormulaWaveSpeed}.

It is not clear a-priori whether a wavefront exists for all $\beta>0$,
because intuitively branchings of the tree can destroy pattern formation
by spreading things out. This is in contrast with FKPP on $\mathbb{R}$. Our result guarantees that the wavefront sustains, provided that $\beta$ is large enough relative to
the topological and the metric structure of the tree. 
The quantity
$\beta_c$ will be given by the right hand side of
\eqref{Assumption:ReactionRateBetaLarge:Eq:BetaLarge} in Section \ref{Sec:WavePropagation}.
For the constant-$(d,\ell)$ tree mentioned above, $\beta_c=\dfrac{d-2}{\ell\,d}\ln (d-1)$ increases to infinity at the asymptotic order $\sim\mathcal{O}\left(\dfrac{\ln d}{\ell}\right)$ as $d$ increases to infinity. 
See Corollary \ref{Cor:ConstantCase} and Figure \ref{Fig:ConstantCase}. 
Note that this $\beta_c$ vanishes  when $d=2$ (i.e. the tree is $\mathbb{R}$) or when $\ell\to\infty$.
Technically speaking, the lower bound of $\beta$ is due to two reasons:
to ensure that we can use the LDP 
and that there is a unique wavefront; see Remark
\ref{Remark:PurposesOfConditions-BetaGreaterEtac-BetaGreaterMuStuff}.

The slow down of the wave speed due to branching can also be heuristically explained:
the density of the mass concentration described by $u$ spreads out to $d_i-1$
many edges as it go pass a vertex of degree $d_i$ (Figure \ref{Fig:Trees}).
In Remark \ref{Remark:HeuristicReasonSlowDownWave}, we provided a further intuitive explanation of this slow down effect, which, roughly speaking, can be attributed to the interaction between the ``drift effect" caused by branching and the large deviations principle.
For the constant-$(d,\ell)$ tree, the   asymptotic speed is given by
\begin{equation}\label{SpeedConstantCase}
    c^*
    =\inf\limits_{\lambda\geq 0}\,\dfrac{\lambda+\beta}{\sqrt{2\lambda}+\dfrac{1}{ \ell}\ln\left(\dfrac{4p}{1+\gamma^2-\sqrt{(\gamma^2-1)^2+4(2p-1)^2\gamma^2}}\right)}\,\in(0,\sqrt{2\beta}\,]\ ,
\end{equation}
where $p=(d-1)/d$ and $\gamma:=e^{\ell\sqrt{2\lambda}}$. The upper bound $\sqrt{2\beta}$ is attained if and only if $d=2$ (i.e. the tree degenerates to $\mathbb{R}$). See Corollary \ref{Cor:ConstantCase} and  Figure \ref{Fig:ConstantCase}.

\medskip

To prove our main result, we start from the classical idea (like the one presented in Freidlin \cite{FreidlinGreenBook}) which
connects the solution $u(t,x)$ of the FKPP equation \eqref{Eq:FisherKPP} with the functional integration over the
trajectories of an underlying stochastic process, and we then make use of the large deviations principle (LDP) of that process.
Indeed, for the classical FKPP
case when $x\in \mathbb{R}$, the solution $u(t,x)$ to \eqref{Eq:FisherKPP} can be represented
via the well-known Feynman-Kac formula as the solution of an integral equation
over the trajectories of a standard Brownian motion on $\mathbb{R}$. The result that the asymptotic wave speed
is given by $\alpha^*=\sqrt{2\beta}$ then follows from LDP for the Brownian motion on $\mathbb{R}$.
Similarly in our case, when
$u(t,x)$ to \eqref{Eq:FisherKPP} is considered
on the tree (i.e. $x\in \mathbb{T}_{\vec{d}, \vec{\ell}}$), the underlying stochastic process in the Feynman-Kac
formula is replaced by a Brownian motion $B_t$
on the tree $\mathbb{T}_{\vec{d}, \vec{\ell}}$. The Brownian motion $B_t$ on the tree behaves
as a standard $1$-dimensional Brownian motion in the interior of the edges, and at each vertex of the tree,
it chooses randomly and with equal probability to enter one of the edges adjacent to that vertex.

Since $d_0=2$, we can associate any point $x\in\mathbb{T}_{\vec{d},\vec{\ell}}$ with a unique
\textit{horizontal coordinate} $y\in \mathbb{R}$ which is the signed distance on $\mathbb{T}$ from $x$ to the root $\rho$ (i.e., $y=\pm d_{\mathbb{T}}(x,\rho)$ with $+$ sign when $x$ belongs to the right branch and $-$ sign when $x$ belongs to the left branch), as illustrated in Figure \ref{Fig:Trees}.
We denote by $\pi:\mathbb{T}_{\vec{d},\vec{\ell}}\to \mathbb{R}$ to be the
projection map sending $x$ to its horizontal coordinate $y$.
Due to the symmetric behavior of the Brownian motion $B_t$ at each vertex of the tree, one can show
(see Section \ref{Sec:TreeStructure:Subsection:ProjectionIdea} and in particular Lemma \ref{Lemma:proj}) that
the solution $u(t,x)=v(t, \pi(x))$, $x\in \mathbb{T}_{\vec{d}, \vec{\ell}}$\ . Here
$v(t,y)$, $y\in \mathbb{R}$ is the solution of an integral equation, given by the Feyman-Kac formula, to which the
underlying stochastic process is given by the projection $Y_t$ of $B_t$ onto $\mathbb{R}$: $Y_t=\pi(B_t)$.
Notice that, when a Brownian motion on $\mathbb{T}_{\vec{d},\vec{\ell}}$ is at a vertex with degree
$d_i$ that is on the right of the root,
the probability that it moves further away from the root (i.e. move to the right) in the next instance is
$\dfrac{d_i-1}{d_i}$. Thus the process $Y_t$ behaves like a Brownian motion except at its interface points (barriers),
i.e., those points on $\mathbb{R}$ that
are the projections under $\pi$ of the vertices of the tree $\mathbb{T}_{\vec{d}, \vec{\ell}}$\ .
At these interface points, it moves
to the right or left with respective probabilities $p_i=\dfrac{d_i-1}{d_i}$ and $1-p_i$ (see the lower part of Figure \ref{Fig:Trees}).
Such a real-valued process $Y_t$, introduced in \cite{Ram11}, is called a \emph{multi--skewed Brownian motion}; precise definitions are in
Section \ref{Sec:MultiSkewBMAssumptions:Subsection:Assumptions}.


Our LDP of the multi-skewed Brownian motion $Y_t$ (see Theorems \ref{Theorem:LDPHittingTime}, \ref{Theorem:LDPSkewBM} in Section \ref{Sec:LDP}) in general
follows  the method of LDP for random processes in random environment in
\cite{CometsGantertZeitouni2000}, \cite{Taleb2001}, \cite[Chapter 7]{FreidlinFunctionalBook}, \cite{Nolen-XinCMP2007}, \cite{Nolen2009}. However, these existing results do not apply directly to $Y$ or the embedded random walk at the interface points.

In fact, such LDP analysis for $Y$ turns out to be remarkably delicate
and interesting. It first involves a calculation of the Lyapunov exponent
given by the Laplace transform of certain hitting time of the multi-skewed Brownian motion $Y_t$
(see Theorem \ref{Theorem:LyapunovExponentPositiveDirection} and Section \ref{Sec:AuxiliaryFunctions}).
Interestingly, such a quantity is calculated by making use of some existence results of an infinite
product of $2\times 2$ random matrices parameterized by $\vec{d}, \vec{\ell}$
(see Proposition \ref{Prop:ExplicitCalculationw-AuxiliaryDeterministic} and Theorem \ref{Theorem:Existence_xi}). This allows us to obtain a variational formula for the wave speed in terms of $\vec{d}, \vec{\ell}$
(see Theorem \ref{Theorem:VariationalFormulaWaveSpeed} in Section \ref{Sec:Algorithm-Examples}).
This variational formula enables us to show that the speed of the wavefront on $\mathbb{T}_{\vec{d}, \vec{\ell}}$ is
slower than the speed of the standard FKPP equation on $\mathbb{R}$, and we can
estimate this slow down in terms of $\vec{d}$ and $\vec{\ell}$.

Due to the random tree structure of $\mathbb{T}_{\vec{d}, \vec{\ell}}$,
the multi-skewed process $Y_t$ behaves as a biased random walk at its interface points.
The biasedness of $Y_t$ at the interface points are away from the root, because $d_i>2$.
When $d_0=2$, one can think of the effects of such biasedness as adding positive and negative drifts to a standard Brownian motion on $\mathbb{R}$.
These ``drift-like skewnesses" result in some interesting behaviors of the hitting time of $Y_t$ (see Section \ref{Sec:MultiSkewBMAssumptions:Subsection:HittingTimeEstimates}), and
they make the LDP of $Y_t$ substantially different from the one for the standard Brownian motion.
Intuitively, such drift effects make $Y_t$ harder to come back to a neighborhood of the origin, so that the LDP
has a lower rate function (action functional). 
Moreover, there is a non-negative finite quantity at which the Laplace transform of the hitting time of the multi-skewed Brownian motion $Y_t$ jumps to infinity (see Theorem \ref{Theorem:etac_Y-w}). 
A more careful analysis will demonstrate that the LDP will
only hold in a particular regime of the parameters (see Theorems \ref{Theorem:LDPHittingTime}, \ref{Theorem:LDPSkewBM}).
Except for these features, the exact shape of the LDP rate function may exhaust various different possibilities (see Figure \ref{Fig:I-a-Graph}). Correspondingly, the analysis of the wavefront propagation
only works in the regime when the reaction rate $\beta$ is larger than some value $\beta_c$.
To the best of the authors' knowledge, except for a short remark
in \cite[Section 7.6, Remark 4, pp.524-525]{FreidlinFunctionalBook} that mentions the case when there is a drift, this is the
first work that carefully addresses such random drift phenomenon for the wavefront propagation of
FKPP equations in random environments via probabilistic method. The particular intricacy in our work is that
we are not working with a simple random drift that can be offset by a moving frame, but a more
complicated ``drift effect" caused by the multi-skewness.

\medskip

\noindent
{\bf Discussion. }In contrast with FKPP on $\mathbb{R}$, it is not completely clear what happens to FKPP on trees 
when the reaction rate $\beta>0$
is smaller than the critical value $\beta_c$ mentioned in the main result above. 

Our approach is based on the LDP analysis for processes in random environments, which only works in a
certain regime of the parameters (see Theorems \ref{Theorem:LDPHittingTime}, \ref{Theorem:LDPSkewBM}).
The LDP analysis that works for processes in random environments can only be applied to the case when the reaction rate $\beta$ is larger than $\beta_c$. Moreover, with the LDP rate function at hand this Assumption also guarantees the uniqueness of wavespeed in equation \eqref{Eq:WaveSpeedFormula}.
However, such approach does not exclude the possibility 
that there are other methods that may work when $\beta$ is small. We leave this issue for future investigation.

On the other hand, if the tree $\mathbb{T}_{\vec{d}, \vec{\ell}}$ is not random but has constant branching lengths and branch degrees, we can employ a more straightforward method (based on the eigenfunction of an elliptic operator) to obtain the LDP (see \cite[Chapter 7, Section 7.3]{FreidlinFunctionalBook}) rather than relying on the hitting time analysis for the multi-skewed BM in a random environment (like what we have in Theorem \ref{Theorem:LDPHittingTime}), so that we may be able to analyze the behavior of \eqref{Eq:FisherKPP} on trees for small values of $\beta$. This issue will be left to the theme of another paper.

It is also worth noticing that our multi-skewed process $Y_t$ here is different from the
process $Y_t$ introduced in \cite{FreidlinHu13Motor} in that the latter process is
ergodic with respect to both positive and negative shifts. In our case,
the behavior of our multi-skewed Brownian motion $Y_t$
is symmetric with respect to the origin. This leads to the fact that the wave speed is the same along positive and
negative axes (see Theorem \ref{Theorem:WaveSpeed}) as well as a few technical differences in the proof of the LDP and the
wave propagation (see Sections \ref{Sec:LDP} and \ref{Sec:WavePropagation}).



Reaction-diffusion systems on geometric structures that have branching and singularities have long been attracting interest in the scientific community. For example, a lot of physics literature discuss reaction-diffusion equation on fractals such as the Sierpinski gasket and the Koch curve. In \cite{campos2004description}, \cite{mendez2004dynamical}, \cite{campos2004propagation}, approximate expressions for the wavespeeds on fractal media have been obtained by physical intuition. See the Campos-M\'endez-Fort formula mentioned in the numerical work \cite[equation (4)]{suwannasen2016speed}. The wave equation is also considered on fractal tree in the simulation work \cite{joly2019wave} as a model of sound propagation in the human lung. However, besides all these efforts, there have been very few works that discuss these problems at the level of absolute mathematically rigor (except \cite{FreidlinHu13} that discusses reaction-diffusion equation on a particular type of infinite tree). Our work puts forward one more step in this direction and our wavespeed formulas are new.

\medskip

\noindent
{\bf Paper outline. }
Section \ref{Sec:TreeStructure} is dedicated to preliminaries, including
the definitions and assumptions of the random tree  and the precise statement of the FKPP equation and wavefront
speed on the tree, as well as the basic idea
of projecting the Brownian motion on the tree to a multi--skewed process $Y$ on $\mathbb{R}$.
Section \ref{Sec:MultiSkewBMAssumptions} contains some hitting time estimates for $Y$ that will be useful in later sections.
Section \ref{Sec:AuxiliaryFunctions} provides a calculation and an analysis for the auxiliary functions used in proving the LDP,
that are based on existence and properties of the limit of
an infinite product of $2\times 2$ random matrices parametrized by $\vec{d}$ and $\vec{\ell}$.
In Sections \ref{Sec:LDP} and \ref{Sec:WavePropagation} we analyze the LDP of $Y_t$
and the corresponding wave propagation respectively. Finally in Section \ref{Sec:Algorithm-Examples} we provide
 a variational formula for computing the wave speed that shows the slow down of the wave on $\mathbb{T}_{\vec{d}, \vec{\ell}}$
 with some concrete calculations.

\begin{flushleft}
{\bf\large Acknowledgements.} The authors thank Partha Dey, David Fisher, Xiaoqin Guo, Yong Liu,  Russ  Lyons, Johnathan Peterson, Lihu Xu and  Xiaoqian Xu for enlightening discussions. Financial support from NSF grant DMS--1804492 is gratefully acknowledged.
\end{flushleft}

\section{Preliminaries}
\label{Sec:TreeStructure}

\subsection{The structure of the random tree}
\label{Sec:TreeStructure:Subsection:Tree}

The class of infinite metric trees is described in the following
and in Figure \ref{Fig:Trees}.

\begin{definition}[symmetric $\vec{d}$-regular tree]\label{Def:VecdRegularTree}
Let $\vec{d}:=(d_n)_{n\in \mathbb{Z}_+}$ be a sequence of positive integers with $d_0=2$.
A \emph{symmetric $\vec{d}$-regular tree} $\mathbb{T}_{\vec{d}}$
is a rooted tree such that all vertices at generation $n$ have the same degree $d_n$
(the root is the node at generation $0$).
\end{definition}

In the above, the assumption that $d_0=2$ is
only introduced for the sake of simplifying the proof and to visualize the geometry, and the arguments in
this paper can easily be extended to the case when $d_0>2$, without affecting
the asymptotic speed of the wavefront (see Theorem \ref{Theorem:VariationalFormulaWaveSpeed}).

As an example, suppose there is a positive integer $d$
such that $d_n=d$ for all $n\geq 1$.
Then we have two identical $d$-regular trees attaching to the root.

We put the following assumption on $\vec{d}$:

\begin{assumption}[bounded branching degrees]\label{Assumption:d}
We assume that there exist some positive integer $\overline{d}\geq 2$ such that
\begin{equation}\label{Assumption:d:Eq:UpperboundBranching}
2\leq d_n\leq \overline{d} \ ,
\end{equation}
for all $n\in \mathbb{Z}_+$, and $d_0=2$.
\end{assumption}

\begin{definition}[symmetric $\vec{d}$-regular tree with branch lengths $\vec{\ell}$]\label{Def:VecdRegularTreeBranchLengths}
Denote by $\mathbb{T}:=\mathbb{T}_{\vec{d},\vec{\ell}}$ the
\emph{symmetric $\vec{d}$-regular tree with branch lengths $\vec{\ell}$},
that is, the $\vec{d}$-regular tree whose
 edges between generations $n$ and $n+1$ are all of length equal to $\ell_n$.
 See the tree in the upper part of Figure \ref{Fig:Trees}. We denote by $V$ the vertex set of $\mathbb{T}$ and $\mathring{\mathbb{T}}:=\mathbb{T}\setminus V$ to be
its interior.
\end{definition}

We put the following assumption on $\vec{\ell}$:

\begin{assumption}[bounded branch lengths]\label{Assumption:ell}
We assume that there exist some $0<\underline{\ell}<\overline{\ell}<\infty$ such that
\begin{equation}\label{Assumption:ell:Eq:PositiveBelowell}
0<\underline{\ell}\leq \ell_n\leq \overline{\ell}<\infty \ ,
\end{equation}
for all $n\in \mathbb{Z}_+$.
\end{assumption}

\begin{definition}[distance on $\mathbb{T}$]\label{Def:DistanceTree}
The tree $\mathbb{T}=\mathbb{T}_{\vec{d},\vec{\ell}}$ is made into a metric space
equipped with the metric $d_{\mathbb{T}}$: for any two points $x_1$ and $x_2$ on $\mathbb{T}$
belonging to the same edge of $\mathbb{T}$ we define
their distance $d_\mathbb{T}(x_1, x_2)$ to be the length of the interval between them; for $x_1$ and $x_2$ belonging to
different edges of $\mathbb{T}$ it is defined as the geodesic distance
$d_\mathbb{T}(x_1, x_2)=\min(d_\mathbb{T}(x_1, O_{j_1})+d_\mathbb{T}(O_{j_1}, O_{j_2})+...+d_\mathbb{T}(O_{j_l}, x_2))$, where the minimum
is taken over all chains of vertices $O_{j_i}\in V$ connecting the points $x_1$ and $x_2$.
\end{definition}

We think of $\mathbb{T}$ as a continuous object, where each edge is a line segment. As mentioned in the introduction and illustrated in Figure \ref{Fig:Trees},
each point $x\in\mathbb{T}$ has a unique
\textit{horizontal coordinate} $\pi(x)\in \mathbb{R}$ which is the signed distance on $\mathbb{T}$ from $x$ to the root.

Our probability space $(\Omega,\,\mathfrak{S},\,\mathbf{P})$ for the randomness in the tree $\mathbb{T}_{\vec{d}, \vec{\ell}}$ is defined as follows.
The sample space $\Omega:= \mathbb{N}^{\mathbb{Z}_+}\times (0,\infty)^{\mathbb{Z}_+}$ has generic sample point $(\vec{d},\vec{\ell})$ and is equipped with its Borel $\sigma$-algebra $\mathfrak{S}$. We then assume the following

\begin{assumption}[i.i.d and mutually independent branching degrees and branch lengths sequences]\label{Assumption:treeP}
Under $\mathbf{P}$,  $\{d_i\}_{i\geq 1}$ and $\{\ell_i\}_{i\geq 0}$ are
two mutually independent sequences of i.i.d. random variables
such that almost surely Assumptions \ref{Assumption:d} and \ref{Assumption:ell} hold.
\end{assumption}

Since we will be considering Brownian motion on the tree $\mathbb{T}_{\vec{d}, \vec{\ell}}$,
so that the pair $(\vec{d}, \vec{\ell})$ determines the environment under which the Brownian motion moves,
we will also refer to the measure $\mathbf{P}$ as the one that governs the \textit{random environment}.

Assumption \ref{Assumption:treeP} includes many random trees of interest. For example, $\{d_i\}_{i\geq 1}$ can be i.i.d. uniform on a finite integer set such as $\{2, 3, 2019\}$. The deterministic $d$-regular tree  with branch length
$\ell$, called constant-$(d,\ell)$ tree in this paper, is the case when
all $\{\ell_i\}_{i\geq 0}$ are equal to a constant $\ell$ and all $\{d_i\}_{i\geq 1}$ are equal to a constant $d$.

\subsection{FKPP equation and its wavefront propagation}
\label{Sec:TreeStructure:Subsection:ProblemStatement}

Our main results are about the speed of wave propagation, as $t\rightarrow\infty$,
for the FKPP equation on the random tree $\mathbb{T}_{\vec{d},\vec{\ell}}$ under $\mathbf{P}$.
Explicitly, we consider the FKPP equation

\begin{equation}\label{Eq:FisherKPPInitialBoundaryConditionVertices}
\left\{\begin{array}{lll}
\dfrac{\partial u}{\partial t}(t,x)=\dfrac{1}{2}\dfrac{\partial^2 u}{\partial x^2}+\beta u(1-u) & , &
(t,x)\in (0,\infty)\times \mathring{\mathbb{T}}_{\vec{d},\vec{\ell}} \ , \\
\nabla u(t,v)=0 & , & (t,v)\in (0,\infty)\times V \ , \\
u(0, x)=u_0(x) & , & x\in \mathring{\mathbb{T}}_{\vec{d},\vec{\ell}} \ ,
\end{array}\right.
\end{equation}
where $V$ is the vertex set of $\mathbb{T}_{\vec{d},\vec{\ell}}$
and $\mathring{\mathbb{T}}_{\vec{d},\vec{\ell}}:=\mathbb{T}_{\vec{d},\vec{\ell}}\setminus V$ is the interior of the tree.
The condition $\nabla u(t,v)=0$ is called the symmetric \textit{gluing condition},
which specifies that the flow-in equals flow-out of mass at each vertex. Specifically,
$\nabla f(v)$ is the sum of the outward derivatives of function $f$ at
vertex $v$, i.e.,
$\nabla f(v)=\sum_i \partial_i f(v)$ in which $\partial_i$ is the outward derivative
along the $i$-th edge attached to the vertex $v$.
The initial condition $u_0(x)=\mathbf{1}_{(-\delta, \delta)}(x)$ for some small $0<\delta<\underline{\ell}$, so it is 1 on part of the two edges connecting to the root and is 0 elsewhere.

Equation \eqref{Eq:FisherKPPInitialBoundaryConditionVertices} first appeared explicitly as scaling
limits of interacting particle systems in \cite{Fan17}.
Following \cite{FreidlinFunctionalBook}, \cite{FreidlinHu13}, we define
a \textit{generalized solution} of \eqref{Eq:FisherKPPInitialBoundaryConditionVertices}
with initial condition $u_0$ to be a
measurable function $u$ that solves the integral equation
\begin{equation} \label{Eq:GeneralizedSolutionFeynmanKacTree}
u(t,x)=E^{(\vec{d}, \vec{\ell})}_x\Big[u_0(B_s)\exp\Big\{\beta\int^t_0 \Big(1-u(t-s,B_s)\Big)ds\Big\} \Big] \ ,
\end{equation}
where $(B_t)_{t\geq 0}$ is the Brownian motion on the tree $\mathbb{T}_{\vec{d},\vec{\ell}}$ \ , and
$E^{(\vec{d}, \vec{\ell})}_x$ is the mathematical expectation with respect to $(B_t)_{t\geq 0}$ starting at $x$, under a
fixed tree $\mathbb{T}_{\vec{d}, \vec{\ell}}$. Notice that since $\mathbb{T}_{\vec{d},\vec{\ell}}$ is random
under $\mathbf{P}$, the process $B_t$ is indeed moving in a random environment distributed as $\mathbf{P}$.

The process $(B_t)_{t\geq 0}$ is the Markov process on $\mathbb{T}_{\vec{d},\vec{\ell}}$
associated with an infinitesimal generator $A$, that is given by the Laplace operator with
gluing boundary conditions. Within each edge of the tree $\mathbb{T}_{\vec{d},\vec{\ell}}$
the infinitesimal generator $A$ of the process $B_t$ is given by $\dfrac{1}{2}\dfrac{d^2}{dx^2}$,
in which $\dfrac{d}{dx}$ is the derivative along that edge. The domain of
definition $D(A)$ of the operator $A$ is given by functions $f$ that are
twice continuously differentiable inside each edge of the tree
$\mathbb{T}_{\vec{d},\vec{\ell}}$, and satisfy the gluing condition
$\nabla f(v)=0$ at each vertex $v$ of the tree $\mathbb{T}_{\vec{d},\vec{\ell}}$.
This notion of solution \eqref{Eq:GeneralizedSolutionFeynmanKacTree} is motivated by the Feynman-Kac formula. The process $B_t$ considered here is a typical example of Markov processes on manifolds with singularity (such as graphs, see \cite{[FHW-Vanishing]}, \cite{[FHLL]}, \cite{[FHNearlyElastic]}, \cite{[FH-CPDE]}, \cite{[H-Metastability]},
\cite{[H-Degenerate]},
\cite{[HPhDThesis]}).

 Let us denote by $\mathcal{B}(S;\,[0,1])$
 (respectively $C(S;\,[0,1])$) the space of bounded Borel measurable
 (respectively continuous) functions on any metric space $S$ taking values
 in $[0,1]$, equipped with the uniform norm $ \|\bullet\|_{\infty}$.
 Based on the contraction mapping principle, as detailed in
 \cite[Section 3, Chapter 5]{FreidlinFunctionalBook} and
  \cite[Theorem 3]{FreidlinHu13}, one immediately obtains
  Lemma \ref{Lemma:ExistenceUniquenessGeneralizedSolution} below, which ensures the well-posedness of
  equation \eqref{Eq:FisherKPPInitialBoundaryConditionVertices}.

\begin{lemma}\label{Lemma:ExistenceUniquenessGeneralizedSolution}
Let $\vec{d}$ and $\vec{\ell}$ be deterministic sequences that satisfy  \eqref{Assumption:d:Eq:UpperboundBranching} and
\eqref{Assumption:ell:Eq:PositiveBelowell} respectively and let
$\mathbb{T}:=\mathbb{T}_{\vec{d},\vec{\ell}}$ be a fixed deterministic tree.
Suppose the initial condition $u_0\in \mathcal{B}(\mathbb{T};[0,1])$.
Then there exists a unique generalized solution $u$ of
\eqref{Eq:FisherKPPInitialBoundaryConditionVertices}
with $u(t,\bullet)\in  C(\mathbb{T};[0,1])$ for all $t>0$.
\end{lemma}

Due to our symmetric construction of the initial condition and the symmetric nature of the Brownian motion $B_t$ on $\mathbb{T}_{\vec{d}, \vec{\ell}}$\ , the solution $u$ to equation \eqref{Eq:FisherKPPInitialBoundaryConditionVertices} satisfies 
$u(t, x_1)=u(t, x_2)$ whenever $d(x_1, \rho)=d(x_2, \rho)$.
Such a fact is actually a consequence of Lemma \ref{Lemma:proj} below. Thus we can give the following definition of the speed of the wavefront:

\begin{definition}\label{Def:WaveSpeedTree}
A positive real number $c^*>0$ is called the \emph{asymptotic
speed} for the wavefont of \eqref{Eq:FisherKPPInitialBoundaryConditionVertices} if for any $h_1>0$ and $c^*>h_2>0$ we have
$$\lim_{t \to \infty}\sup_{d_{\mathbb{T}}(x,\rho)>(c^*+h_1)t}u(t,x)=0 \ , \
\lim_{t \to \infty}\inf_{d_{\mathbb{T}}(x,\rho)<(c^*-h_2)t}u(t,x)=1 \ ,$$
where $u(t,x)$ is the generalized solution to \eqref{Eq:GeneralizedSolutionFeynmanKacTree}.
\end{definition}

In a nutshell, the problem studied in this work can be formally stated as follows:

\

\noindent \textbf{Statement of the Problem.} \textit{For what values of the reaction rate $\beta>0$ does the equation \eqref{Eq:FisherKPPInitialBoundaryConditionVertices} admits a wavefront, as $t\rightarrow \infty$, that satisfies Definition \ref{Def:WaveSpeedTree}? When the wavefront exists, can we analyze its asymptotic speed?}

\

This problem is answered already in the introductory section, and the rest of the paper is dedicated to solving it.

\subsection{The basic idea of projection}
\label{Sec:TreeStructure:Subsection:ProjectionIdea}

Our key observation is as follows:
when a Brownian motion on $\mathbb{T}_{\vec{d},\vec{\ell}}$ is at a vertex with degree $d_i$ that is on the right of the root (see Figure \ref{Fig:Trees}), the probability that it moves further away from the root (i.e. move to the right) in the next instance is
$p_i=\dfrac{d_i-1}{d_i}$.

Therefore, instead of analyzing the large deviation behaviors of the Brownian motion
$B_t$ on $\mathbb{T}_{\vec{d},\vec{\ell}}$, we do so for the projection of
$B_t$ onto a one-dimensional axis along the direction of the wave propagation.
The  projected process is the multi-skewed Brownian motion $Y_t\in \mathbb{R}$ introduced in \cite{Ram11}.
LDP of $Y$ then leads to the asymptotic speed of a wave $v$ travelling on $\mathbb{R}$, via the Feynman-Kac formula.
Our setting, specifically the collection of trees and the initial condition $u(0,x)=u_0(x)=\mathbf{1}_{(-\delta, \delta)}(x)$,
guarantees that the asymptotic speed of $v$ is the same as that of the
solution $u$ of the reaction-diffusion equation \eqref{Eq:FisherKPP}.
Figure \ref{Fig:Idea} illustrates this idea.

\begin{figure}
\centering
\includegraphics[height=6cm, width=14cm]{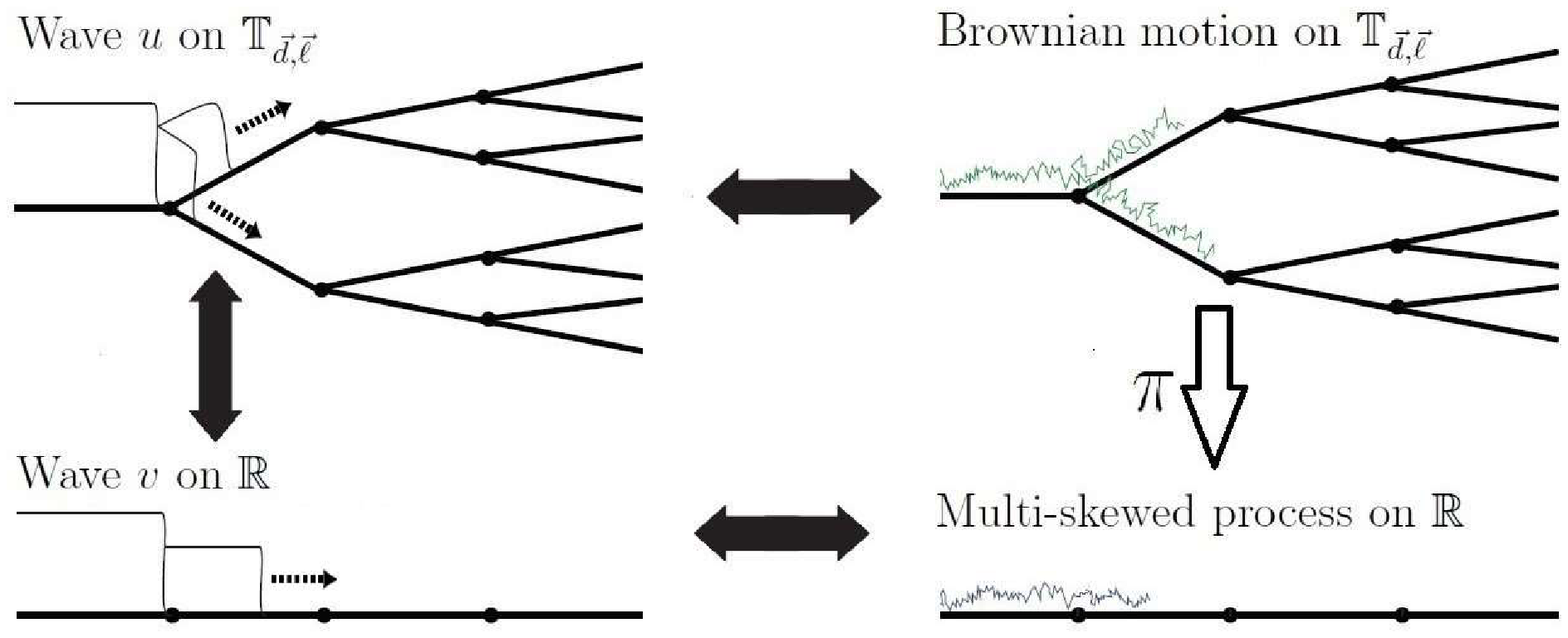}
\caption{{\bf Connection with multi-skewed Brownian motion via projection $\pi$. }
[Vertical double-arrow]
The top-left part shows the graph of a solution $u$ of \eqref{Eq:FisherKPP} (the wavefront) on half of a tree, the other half of the tree is symmetric. The wave will propagate
to the right (dashed arrows) along all edges of the tree.
The bottom left shows an arbitrary semi-infinite branch. Knowing the wave $v$ on it determines the wave $u$ on the entire tree, and vice versa, as described in Lemma \ref{Lemma:proj}. This is because the travelling waves on any two semi-infinite branches are the same.
This equivalence relation is symbolised by the vertical double-arrow.
[Horizontal double-arrows] Feynman-Kac formula allows us to write the solutions of reaction diffusion equations in terms of diffusion processes, giving  \eqref{Eq:GeneralizedSolutionFeynmanKacTree} in the upper horizontal double-arrow and \eqref{Lemma:proj:Eq:FeynmanKac-v} in the lower horizontal double-arrow.
[Vertical  one-sided arrow] On the right, $\pi$ is the projection that maps the Brownian motion on the tree $\mathbb{T}_{\vec{d},\vec{\ell}}$ to a multi-skewed process $Y$ on $\mathbb{R}$.
}
\label{Fig:Idea}
\end{figure}


The interface set $\vec{z}=(z_n)_{n\in \mathbb{Z}}$ is such that $z_0=0$, $z_{i+1}-z_i=\ell_i$ and $z_{-i}=-z_{i}$
for $i\geq 0$. Clearly  $z_{n}=\sum_{i=0}^{n-1} \ell_i=-z_{-n}$ for $n\geq 1$.
Assumption \ref{Assumption:ell} ensures that
$\vec{z}:=\{z_i\}_{i\in \mathbb{Z}}$ has no accumulation point.

The trajectories of process $Y=(Y_t)_{t\geq 0}$ behave like Brownian motion on $\mathbb{R}\setminus\vec{z}$, and
at point $z_i>0$, the probability of hitting $z_i+\varepsilon$ before hitting $z_i-\varepsilon$ is equal to
$p_i:=\dfrac{d_i-1}{d_i}$ as $\varepsilon$ is tending to zero; see \cite[Theorem 1.2]{Ram11}.
This property of $Y$ leads to the following lemma. We will formally define process $Y$ in Section \ref{Sec:MultiSkewBMAssumptions}, Definition \ref{Def:GSBM}.

\begin{lemma}\label{Lemma:piBY}
$\pi(B)=Y$ in distribution in $C(\mathbb{R}_+,\mathbb{R})$.
\end{lemma}

Based on Lemma \ref{Lemma:piBY}, Lemma \ref{Lemma:proj} below tells us that we can recover function $u$
from its restriction $v$ on a single infinite branch, and that such a
restriction $v$ also enjoys a Feynman-Kac formula involving path integrals
for the multi-skewed Brownian motion $Y$. In the below $P^{(\vec{d}, \vec{\ell})}_y$,
$E^{(\vec{d},\vec{\ell})}_y$ are the probabilities and the mathematical
expectation with respect to $(Y_t)_{t\geq 0}$ starting at $y$, under a given tree $\mathbb{T}_{\vec{d}, \vec{\ell}}$.

\begin{lemma}\label{Lemma:proj}
Let $\vec{d}$ and $\vec{\ell}$ be deterministic sequences that satisfy  \eqref{Assumption:d:Eq:UpperboundBranching} and
\eqref{Assumption:ell:Eq:PositiveBelowell} respectively and let
$\mathbb{T}:=\mathbb{T}_{\vec{d},\vec{\ell}}$ be a fixed deterministic tree.
Suppose the initial condition
$u_0\in \mathcal{B}(\mathbb{T};[0,1])$ satisfies
$u_0(x_1)=u_0(x_2)$ whenever $\pi(x_1)=\pi(x_2)$. Then
\begin{equation}\label{Lemma:proj:Eq:uEqualv}
u(t,x)=v(t,\pi(x))\quad  \text{for all } (t,x)\in\mathbb{R}_+\times\mathbb{T},
\end{equation}
where $v$ is the unique element in $\mathcal{B}([0,\infty)\times \mathbb{R};\,[0,1])$ such that
\begin{equation} \label{Lemma:proj:Eq:FeynmanKac-v}
v(t,y)=E^{(\vec{d},\vec{\ell})}_y\Big[v_0(Y_t)\exp\Big\{\beta\int^t_0 \Big(1-v(t-s,Y_s)\Big)ds\Big\} \Big],
\end{equation}
the process $Y=(Y_t)_{t\geq 0}$ is the multi-skewed Brownian Motion in Definition
\ref{Def:GSBM} and the function $v_0\in \mathcal{B}(\mathbb{R};[0,1])$ is defined by $v_0 \circ \pi=u_0$. Furthermore,
 $v(t,\cdot)\in C(\mathbb{R};[0,1])$ for all $t>0$.
\end{lemma}

\begin{proof}
Similar to the proof of the Theorem 3.1 \cite{FreidlinHu13}, from the contraction mapping
theorem on the Banach space $\mathfrak{B}_T:=\mathcal{B}([0,T]\times \mathbb{R};\,[0,1])$ with the
uniform norm, for the operator $\Phi:\mathfrak{B}_T\to \mathfrak{B}_T$ defined by
$$\Phi(f)(t,y):=E^{(\vec{d},\vec{\ell})}_y\Big[u_0(Y_t)\exp\Big\{\beta\int^t_0 \Big(1-f(t-s,Y_s)\Big)ds\Big\} \Big] \ , \ f\in \mathfrak{B}_T \ ,$$
where $T\in (0,\infty)$ is small enough, and then by extending to time
intervals of arbitrary length, it follows that there is a unique
$v\in \mathcal{B}([0,\infty)\times \mathbb{R};\,[0,1])$ satisfying \eqref{Lemma:proj:Eq:FeynmanKac-v}
on $[0,\infty)$.

To check the details, for all $0<t\leq T$, we have
\begin{align*}
\left|\Phi(f)(t,y)-\Phi(g)(t,y)\right|&=\left|E^{(\vec{d},\vec{\ell})}_y\Big[u_0(Y_t)\exp\Big\{\beta\int^t_0 \Big(1-f(t-s,Y_s)\Big)ds\Big\} \Big]\right.\\
& \qquad \qquad \left.-E^{(\vec{d},\vec{\ell})}_y\Big[u_0(Y_t)\exp\Big\{\beta\int^t_0 \Big(1-g(t-s,Y_s)\Big)ds\Big\} \Big]\right|\\
&\leq\|u_0\|_{\infty}\Big|E^{(\vec{d},\vec{\ell})}_y\Big[\exp\Big\{\beta\int^t_0 \Big(1-f(t-s,Y_s)\Big)ds\Big\}\\
& \qquad \qquad \qquad \qquad \qquad -\exp\Big\{\beta\int^t_0 \Big(1-g(t-s,Y_s)\Big)ds\Big\} \Big]\Big|\\
&\leq \|u_0\|_{\infty}\beta\exp(\beta t)t\|f-g\|_{\infty} \qquad \text{(by Mean Value Theorem)} \ ,
\end{align*}
which is strictly less than $\|f-g\|_{\infty}$ for $T$ small enough. Now we
can extend the solution to intervals $[T,2T],\dots,$ $[(n-1)T,nT]$ for $n\in \mathbb{Z}$.
The continuity of $v$ will then follow from Lemma \ref{Lemma:ExistenceUniquenessGeneralizedSolution} and
\eqref{Lemma:proj:Eq:uEqualv}.

It remains to prove \eqref{Lemma:proj:Eq:uEqualv}. By the assumption on the initial condition,
as well as the symmetry of $\mathbb{T}_{\vec{d},\vec{\ell}}$ with respect to the
horizontal direction at each bifurcation of the tree $\mathbb{T}_{\vec{d},\vec{\ell}}$
(see Figure \ref{Fig:Trees}), there exists a function $w:\mathbb{R}\to [0,1]$ such that
$u(t,x)=w(t,\pi(x))$ for all $x\in \mathbb{T}$ and $t\geq 0$. By \eqref{Eq:GeneralizedSolutionFeynmanKacTree},
we have
$$w(t,\pi(x))=E^{(\vec{d},\vec{\ell})}\Big[w(0, \pi(B_s))\exp\Big\{\beta\int^t_0 \Big(1-w\big(t-s,\pi(B_s)\big)\Big)ds\Big\}\Big] \ ,$$
where $B_t$ is a Brownian motion on the tree $\mathbb{T}$. Since $\pi(B)=Y$ in distribution
by Lemma \ref{Lemma:piBY}, we obtain that $w$ solves equation \eqref{Lemma:proj:Eq:FeynmanKac-v}
which implies that $w=v$ by uniqueness of solution to \eqref{Lemma:proj:Eq:FeynmanKac-v}.
\end{proof}

In particular, as we assumed $u_0(x)=\mathbf{1}_{(-\delta, \delta)}(x)$, we obtain from Lemma \ref{Lemma:proj} that $v_0(y)=\mathbf{1}_{(-\delta, \delta)}(y)$. From \eqref{Lemma:proj:Eq:uEqualv} in  Lemma \ref{Lemma:proj},  we see that the wave speed for
$u(t,x)$ on $\mathbb{T}_{\vec{d}, \vec{\ell}}$, defined in Definition \ref{Def:WaveSpeedTree}, is the same as that for $v(t, y)$ on $\mathbb{R}$.

\begin{lemma}\label{Lemma:WaveSpeedProjectedEqualTree}
A positive real number $c^*>0$ is the asymptotic
speed for  \eqref{Eq:FisherKPPInitialBoundaryConditionVertices} (in the sense of Definition \ref{Def:WaveSpeedTree}) if the following holds: for any $h>0$
$$\lim_{t \to \infty}\sup_{y>(c^*+h)t}v(t,y)=\lim_{t \to \infty}\sup_{y<(-c^*-h)t}v(t,y)=0 \ , \
\lim_{t \to \infty}\inf_{(-c^*+h)t<y<(c^*-h)t}v(t,y)=1 \ ,$$
where $v(t,y)$ is the generalized solution to \eqref{Lemma:proj:Eq:FeynmanKac-v}.
\end{lemma}

Our analysis of the wavefront propagation of the equation \eqref{Eq:FisherKPPInitialBoundaryConditionVertices} on $\mathbb{T}_{\vec{d}, \vec{\ell}}$
is reduced by Lemmas \ref{Lemma:proj}, \ref{Lemma:WaveSpeedProjectedEqualTree} to the analysis of the corresponding solution $v(t, y)$ of the integral equation \eqref{Lemma:proj:Eq:FeynmanKac-v}
given by the Feynman-Kac formula.
The rest of the paper is dedicated to the surprisingly delicate analysis of the LDP of $Y_t$ and the wave propagation of \eqref{Lemma:proj:Eq:FeynmanKac-v}, that leads to the solution to our problem.

\subsection{Notations and convention}
\label{Sec:TreeStructure:Subsection:Notations}

We collect some notations here for the reader's convenience. Let $\mathbb{N}=\mathbb{Z}_{>0}=\{n\in \mathbb{Z}, n>0\}$ be the set of positive integers and $\mathbb{Z}_{+}=\mathbb{Z}_{\geq0}=\{n\in \mathbb{Z}, n\geq 0\}$ be the set of non-negative integers, and similarly $\mathbb{Z}_{<0}=\{n\in \mathbb{Z}, n<0\}$ and $\mathbb{Z}_{\leq0}=\{n\in \mathbb{Z}, n\leq 0\}$. We let $a\vee b:=\max\{a,b\}$ and $a\wedge b:=\min\{a,b\}$.
We denote an open $\delta$-ball centered at $u\in \mathbb{R}$ to be
$B_\delta(u)$. A tree $\mathbb{T}=\mathbb{T}_{\vec{d}, \vec{\ell}}$ is equipped with two parameters: the branching
degree sequence $\vec{d}=(d_i)$
and the branch lengths sequence $\vec{\ell}=(\ell_i)$. If $\ell_i=1$ for all $i$, then
$\mathbb{T}_{\vec{d}, \vec{\ell}}=\mathbb{T}_{\vec{d}}$. We define $p_i=\dfrac{d_i-1}{d_i}$ and we set the interface points $z_i$ so that $z_0=0$ and
$z_{i+1}-z_i=\ell_i$ and $z_{-i}=-z_i$
for $i\geq 0$.  The pair $(\vec{d}, \vec{\ell})$ uniquely determines $(\vec{p}, \vec{z})$ and vise versa. Hence we use them interchangeably.

The Brownian motion on $\mathbb{T}_{\vec{d}, \vec{\ell}}$ is denoted by $B_t$ and the corresponding multi-skewed Brownian motion on $\mathbb{R}$
is denoted by $Y_t$. Notice that when the tree degenerates to $\mathbb{R}$, this also includes the case that $B_t$ stands for a standard Brownian motion on $\mathbb{R}$.
If these two processes are written with superscripts, like $B^x_t$ or $Y^y_t$, then it stands for the corresponding process
starting at the initial point denoted by the superscripts $x\in \mathbb{T}$ and $y\in\mathbb{R}$. The probabilities and expectations for the Brownian motion and multi-skewed Brownian Motion with a fixed environment are denoted by $P^{(\vec{d}, \vec{\ell})}$ ($P^{(\vec{p}, \vec{z})}$) and $E^{(\vec{d}, \vec{\ell})}$ ($E^{(\vec{p}, \vec{z})}$). The probabilities and expectations for the random environment are defined by $\mathbf{P}$ and $\mathbf{E}$. We set
 $p_{+1}^i\equiv p^i=\dfrac{\ell_{i-1}p_i}{\ell_i(1-p_i)+\ell_{i-1}p_i}$ and $p_{-1}^i\equiv q^i=1-p_{+1}^i$. The first hitting time for $Y_t$ from $Y_0=s$ to $r$ is defined by $T^s_r$.
We let $\tau_k$ be the $k$-th time the multi-skewed BM $Y$ hits the interface set $\vec{z}=(z_i)_{i\in \mathbb{Z}}$. We set $\eta=-\lambda$ to be two parameters of opposite sign, and  $\gamma_i=e^{\sqrt{2\lambda}\ell_i}$, $\zeta_i=2p_i-1$. A limiting random variable $\xi=\xi_{\lambda}=\xi_{-\eta}$ will be introduced to analyze the wave speed.

\section{Multi-skewed Brownian motion in random environment}\label{Sec:MultiSkewBMAssumptions}

\label{Sec:MultiSkewBMAssumptions:Subsection:Assumptions}

Let $\vec{z}:=\{z_i\}_{i\in \mathbb{Z}}$ be a set of real numbers with no accumulation point
(sometimes we refer to $z_i$'s as barriers or the interface points, and $\vec{z}$ the interface set)
and $\vec{p}:=\{p_i\}_{i\in \mathbb{Z}} \subset (0,1)$ (we refer to $\vec{p}$ as the skewness sequence).
It follows from \cite{LG84} that there is a unique
pathwise solution $Y=(Y_t)_{t\geq 0}$ to the stochastic differential equation
\begin{equation} \label{Eq:MultiSkewBM-SDE}
Y_t= Y_0+ B_t+ \int_{\mathbb{R}} L^{Y}(t,x)\,d\mu(x) \ ,
\end{equation}
where $B$ is the standard Brownian motion on $\mathbb{R}$, $L^Y$ is the local time of the
unknown process $Y$, and $\mu$ is the bounded measure
\begin{equation} \label{Eq:MultiSkewBM-BarrierMeasure}
\mu=\sum_{i\in \mathbb{Z}} (2p_i-1)\,\delta_{z_i} \ .
\end{equation}

\begin{definition}[multi-skewed Brownian motion \cite{Ram11}]\label{Def:GSBM}
The unique diffusion process  $Y=(Y_t)_{t\geq 0}$ on  $\mathbb{R}$ solving \eqref{Eq:MultiSkewBM-SDE}
is called a multi-skewed Brownian motion with skewness sequence $\vec{p}$ and interface set $\vec{z}$.
\end{definition}

Suppose, as in Section \ref{Sec:TreeStructure:Subsection:ProjectionIdea}, that $z_0=0$, $z_{i+1}-z_i=\ell_i$ and $z_{-i}=-z_i$
for $i\geq 0$, and
\begin{equation}\label{Eq:p-d-relation}
p_i:=\dfrac{d_i-1}{d_i} \ .
\end{equation}
The pair $(\vec{d}, \vec{\ell})$ uniquely determines $(\vec{p}, \vec{z})$ and vise versa.
The symmetric $\vec{d}$-regular tree with branch lengths $\vec{\ell}$ can
then be denoted either as $\mathbb{T}_{\vec{d}, \vec{\ell}}$ or $\mathbb{T}_{\vec{p}, \vec{z}}$.

We denote the two probability measures governing the environment $(\vec{p},\vec{z})\in (0,1)^{\mathbb{Z}}\times \mathbb{R}^{\mathbb{Z}}$ and the diffusion $Y$ with skewness sequence and interface set $(\vec{p},\vec{z})$ by, respectively, $\mathbf{P}$ and $P^{(\vec{p},\vec{z})}$.
Their mathematical expectations are denoted by
$\mathbf{E}$ and $E^{(\vec{p},\vec{z})}$ respectively.
Following literature on random walk in random environments (RWRE) we refer to $P^{(\vec{p},\vec{z})}$ as the \textit{quenched} law \footnote{The annealed measure $\mathbb{P}$ is defined by  $\mathbb{P}(A)=\mathbf{E}[ P^{(\vec{p},\vec{z})}(A)] =\int_{(0,1)^{\mathbb{Z}}\times \mathbb{R}^{\mathbb{Z}}} P^{(\vec{p},\vec{z})}(A)\,d\mathbf{P}$.}.

Assumption \ref{Assumption:treeP} directly implies the following
Lemma for the structure of the interface set $\vec{z}$ and skewness sequence $\vec{p}$.
\begin{lemma}\label{Assumption:ErgodicEnvironments}
The probability measure $\mathbf{P}$ on the space of ``environments"
$(\vec{p}, \vec{z})\in (0,1)^{\mathbb{Z}}\times \mathbb{R}^{\mathbb{Z}}$
that governs the structure of the multi--skewed Brownian Motion $Y_t$ satisfies the following:
\begin{enumerate}
\item[(1)] (i.i.d.\ skewness and branch lengths).
Under $\mathbf{P}$, $(p_i)_{i\geq 1}$ is an i.i.d sequence
of random variables in $(0,1)$, $(\ell_i=z_{i+1}-z_i)_{i\geq 0}$ is an i.i.d sequence in $(0,\infty)$, and the two sequences are independent.

\item[(2)] ($\vec{p}$ is symmetric). For $\mathbf{P}$-almost all $(\vec{p},\vec{z})$, there exists a sequence
of positive integers $\vec{d}:=(d_n)_{n\in \mathbb{Z}_+}$ with $d_0=2$
and $2\leq d_n\leq \overline{d}<\infty$, such that $p_i=\dfrac{d_i-1}{d_i} \text{ and } p_{-i}=1-p_{i} \text{ for } i\geq 0$.

\item[(3)] ($\vec{z}$ is symmetric). $z_0=0$, $z_n:=\sum_{i=0}^{n-1} \ell_i$ and $z_{-n}=-z_{n}$ for $n>0$.
For $\mathbf{P}$-almost all $(\vec{p},\vec{z})$ and all $n\geq 0$,
$0<\underline{\ell} \leq z_{n+1}-z_n\leq \overline{\ell}<\infty$.
\end{enumerate}
\end{lemma}


We define  $p^{i}_{+1}$ and $p^{i}_{-1}$ by
\begin{equation}\label{Eq:p^s2}
p^{i}_{+1}:= \frac{\ell_{i-1}p_i}{\ell_{i}(1-p_i)+\ell_{i-1}p_i}\quad \text{and} \quad  p^{i}_{-1}:=1- p^{i}_{+1}.
\end{equation}

\begin{remark}[embedded random walk]\rm\label{Remark:EmbeddedRandomWalk}
The embedded random walk of $Y$ on $\vec{z}:=\{z_i\}_{i\in \mathbb{Z}}$ is
a biased random walk with transition probabilities $\{p^{i}_{+1},p^{i}_{-1}\}$ given by \eqref{Eq:p^s2}.
Precisely, let $\tau_0=0$ and for $k\geq 0$ we define
\begin{equation}\label{Def:tau_i}
\tau_{k+1}:=\inf\{ t>\tau_k: Y_t \in \vec{z} \}
\end{equation}
to be the $(k+1)$-th time that $Y$ hits the set $\vec{z}=(z_i)_{i\in \mathbb{Z}}$. Then by
\cite[equations (2.9)-(2.10)]{Ram11}, the random walk $\{Y_{\tau_k}\}_{k\geq 0}$ satisfies
\begin{equation}\label{Eq:p^s}
P(Y_{\tau_{k+1}} =  z_{i+1} | Y_{\tau_k}=z_i)=p^{i}_{+1} \quad
\text{and} \quad P(Y_{\tau_{k+1}} =  z_{i-1} | Y_{\tau_k}=z_i)=p^{i}_{-1}.
\end{equation}
\end{remark}

\subsection{Hitting time estimates}
\label{Sec:MultiSkewBMAssumptions:Subsection:HittingTimeEstimates}

For the multi-skewed Brownian motion $Y_t$ on $\mathbb{R}$, and any $r,s\in \mathbb{R}$, let us introduce the first hitting time to $r$ starting at $s$
\begin{equation}\label{Eq:StoppingTimesTorMultiSkewedBM}
T^s_r=\inf\left\{t\geq 0, Y_0=s, Y_t=r\right\} \ .
\end{equation}

Let $S_t^i$ be a standard $p_i$-skewed Brownian motion (see \cite{Harrison-SheppSkewBM}). That is,
\begin{equation} \label{Eq:SSBM}
S^i_t= S^i_0+ B_t+ (2p_i-1) L^{S^i}_t \ ,
\end{equation}
where $B$ is the standard Brownian motion on $\mathbb{R}$, $L^S$ is the local time of the
unknown process $S$ at $0$.
Let $\sigma(i):=\inf\{t\geq 0:\,S^{i}_t\in \{-\ell_{i-1},\ell_{i}\}\}$ be the exit time of the standard
$p_i$-skewed Brownian motion $S^{i}_t$ on the interval $(-\ell_{i-1},\ell_{i})$, starting at 0. The probabilities
and expectations with respect to the driving Brownian motion $B_t$ in \eqref{Eq:SSBM} are denoted as $P$ and $E$, respectively.

Denote by $S^i_{+1}$ the event $\{S^{i}_{\sigma(i)}= \ell_{i}\}$
and by  $S^i_{-1}$ the event $\{S^{i}_{\sigma(i)}= -\ell_{i-1}\}$.
Then Remark \ref{Remark:EmbeddedRandomWalk} asserts that
$P(S^i_{\pm 1})=p^{i}_{\pm 1}$.
Define
\begin{align}\label{Eq:Def:J_eta}
 J^{i}_{\pm}:=J^{i}_{\eta, \,\pm 1}:=&E\left[e^{\eta \,\sigma(i)} \,{\bf 1}_{\{S^i_{\pm 1}\}}\right]=E\left[e^{\eta \,\sigma(i)} \,\big|\, S^{i}_{\pm 1}\right] p^i_{\pm 1}.
\end{align}
We write $p^i:=p^{i}_{+1}$ and $q^i:=p^{i}_{-1}=1-p^i$ to simplify notation.

For fixed $\lambda \in \mathbb{R}$, we define the auxiliary function
\begin{equation}\label{Eq:w-AuxiliaryFunction}
w(x)\equiv
w_{\lambda}(x)\equiv E^{(\vec{p},\vec{z})}\left[e^{-\lambda T_0^x}\mathbf{1}_{T_0^x<\infty}\right] \ , \ x\in \mathbb{R} \ ,
\end{equation}
which might be $+\infty$ when $\lambda<0$. Set $\eta=-\lambda$, then we have
\begin{equation}\label{Eq:Def:wforY}
w_{-\eta}(z_i)= E^{(\vec{p},\vec{z})}\left[e^{\eta T_0^{z_i}} \mathbf{1}_{T_0^{z_i}<\infty}\right].
\end{equation}

The following lemma summarizes some elementary properties of the function $w(x)$.

\begin{lemma}\label{Lemma:ElementaryPropertiesw}
Let $\vec{z}:=\{z_i\}_{i\in \mathbb{Z}}\subset \mathbb{R}^\mathbb{Z}$ and
$\vec{p}:=\{p_i\}_{i\in \mathbb{Z}} \subset (0,1)^\mathbb{Z}$ satisfy Lemma \ref{Assumption:ErgodicEnvironments}.
The function $w=w_{\lambda}:\mathbb{R} \to [0,\infty]$ defined in \eqref{Eq:w-AuxiliaryFunction}
satisfies the following  properties:
\begin{itemize}
\item[(1)] For $x>0$, $w(x)$ does not depend on the choice of $\{(z_i,p_i)\}_{i<0}$. Similarly, for $x<0$, $w(x)$ does not depend on the choice of $\{(z_i,p_i)\}_{i>0}$.

\item[(2)] $w(x)=w(-x)$ for all $x\in \mathbb{R}$.

\item[(3)] If $\lambda=0$, then $w_0(x)=P^{(\vec{p},\vec{z})}(T_0^x<\infty)$.

\item[(4)] If $\lambda>0$, then $w$ is strictly decreasing on $(0,\infty)$ and strictly increasing on $(-\infty,0)$.
\end{itemize}
\end{lemma}

\begin{proof}
\begin{itemize}
\item[(1)] This is clear from the definition: if $x>0$, then the trajectory of $Y$ stays on $(0,\infty)$ during time interval $[0,T^x_0)$.

\item[(2)] For any Borel set $A\subset (0,\infty)$, we define its reflection set
\begin{equation*}
-A\equiv \{-x: x\in A\} \ .
\end{equation*}
Then we have
\begin{equation}\label{Lemma:ElementaryPropertiesw:Eq:SymmstryProbabilityDistributionSkewBM}
P^{\vec{p}, \vec{z}}(Y_t\in A)=P^{(\vec{p}, \vec{z})}(Y_t\in -A)
\end{equation}
due to the symmetry of $(\vec{z},\vec{p})$ described in Lemma \ref{Assumption:ErgodicEnvironments}.

\item[(3)] This follows from the definition of the function $w(x)$ in \eqref{Eq:w-AuxiliaryFunction}.

\item[(4)] This can be checked by applying the strong Markov property of process $Y$ to the stopping times $\{T^x_y\}$ and the fact that, for $x>y>0$, we have
\begin{equation}\label{Lemma:ElementaryPropertiesw:Eq:w-decrease}
T^x_0=T^x_y+T^y_0\theta_{T^x_y},
\end{equation}
where for any $t\leq 0$, $\theta_t:\,\mathcal{C}_*\to \mathcal{C}_*$ is
the shift operator $\theta_t x(s)=x(s+t)$, $x\in \mathcal{C}_*$.
\end{itemize}
\end{proof}

Recall that $w_0(z_1)=P^{(\vec{p}, \vec{z})}(T^{z_1}_0<\infty)$ by Lemma \ref{Lemma:ElementaryPropertiesw} part (3).
The following auxiliary lemma about $w_0(z_1)$ will be useful in the proof of Lemma \ref{Lemma:PropertiesMu}.
It asserts that almost surely with respect to $\mathbf{P}$, this probability
is strictly positive.

\begin{lemma}\label{Lemma:EstimateHittingProbabilityBackZeroMulti-SkewedBM}
Let $\vec{z}:=\{z_i\}_{i\in \mathbb{Z}} \subset \mathbb{R}^\mathbb{Z}$
and $\vec{p}:=\{p_i\}_{i\in \mathbb{Z}} \subset (0,1)^\mathbb{Z}$ satisfy
Lemma \ref{Assumption:ErgodicEnvironments}.
Then

\begin{itemize}

\item[(a)] For $i\geq 1$,
$w_{0}(z_i)=\dfrac{\sigma_i}{1+\sigma_i} \in (0,1]$
where
\begin{equation}\label{Eq:Def:sigmaiY}
\sigma_i:=\sum_{k\geq i} \prod_{j=1}^{k}\frac{p^{j}_{-1}}{p^{j}_{+1}}=\sum_{k\geq i} \prod_{j=1}^{k}\frac{\ell_{j}(1-p_j)}{\ell_{j-1}p_j}.
\end{equation}
In particular, $w_{0}(z_1)\in (0,1)$ if and only if $\sigma_1<\infty$.

\item[(b)] There exists some positive constant $C_*=C_*(\overline{l}, \underline{l}, \overline{d})>0$
that depends only on $\overline{l}, \underline{l}, \overline{d}$ such that
$w_{0}(z_1)\geq C_*$ almost surely under $\mathbf{P}$.

\end{itemize}
\end{lemma}

\begin{proof}
(a) Note that
\begin{equation}\label{Eq:tauYX_gen}
T^x_0 = \sum_{i=1}^{\tau^X}(\tau_{i}-\tau_{i-1}),
\end{equation}
where  $\{\tau_{i}\}_{i\geq 1}$ is defined in \eqref{Def:tau_i}, $\tau_0=0$ and  $\tau^X=\inf\{k\geq 0:\,Y_{\tau_k}=0\}$.
Under the bounded Assumptions \eqref{Assumption:d:Eq:UpperboundBranching} and \eqref{Assumption:ell:Eq:PositiveBelowell}, $T^x_0<\infty$ if and only if $\tau^X<\infty$.
Part (a) then follows from standard results for random walks (see, for instance, \cite[Chapter VI section 5.1]{Taylor-KarlinStochasticModelingBook}).

(b) From part (a) and part (3) in Lemma \ref{Lemma:ElementaryPropertiesw}, the hitting probability
\[
w_{0}(z_1)=P^{(\vec{p},\vec{z})}(T^{z_1}_0<\infty)=\frac{\sigma_1}{1+\sigma_1} \ ,
\]
where by \eqref{Eq:Def:sigmaiY} and \eqref{Eq:p^s2}
\begin{equation*}
\sigma_1:=\sum_{k\geq 1} \prod_{j=1}^{k}\frac{p^{j}_{-1}}{p^{j}_{+1}}=\sum_{k\geq 1} \prod_{j=1}^{k}\frac{\ell_{j}(1-p_j)}{\ell_{j-1}p_j}=\sum_{k\geq 1} \prod_{j=1}^{k}\frac{\ell_{j}}{\ell_{j-1}(d_j-1)}.
\end{equation*}
From this we see that $\mathbf{P}$-a.s. we have
\[\sigma_1\geq  \sum_{k\geq 1}\left(\frac{\underline{\ell}}{\overline{\ell}(\overline{d}-1)}\right)^k:=C >0.
\]
So part (b) holds with $C_*=\dfrac{C}{1+C}$.
\end{proof}

\medskip

Let us define the critical exponent
\begin{equation}\label{Eq:Def:etacY-w}
\eta_c^w:=\sup\{\eta\in \mathbb{R}:\,w_{-\eta}(z_1)<\infty\} \in[0,\infty].
\end{equation}
Theorem \ref{Theorem:etac_Y-w} below gives an representation of $\eta_c^w$ and yields that $\eta_c^w<\infty$.
Similar results are obtained in \cite[Lemmas 2 and 4]{CometsGantertZeitouni2000}, but we cannot directly apply them here.

\begin{theorem}[existence of $\eta_c^w\in [0,\infty)$]\label{Theorem:etac_Y-w}
Under Lemma \ref{Assumption:ErgodicEnvironments}, the critical exponent
 $\eta_c^w$ defined in \eqref{Eq:Def:etacY-w} is the unique element in $[0,\infty)$ such that
\begin{equation}\label{Theorem:etac_Y-w:Eq:etaYc_rep1}
\lim_{k\to\infty}\left(\sum_{x\in \mathbb{X}_k}\prod_{i=0}^{2k}J^{i}_{\eta_c^w,\,x_{i+1}-x_{i}}\right)^{1/k} =1 \ ,
\end{equation}
where $J^i_{\eta_c^w, x_{i+1}-x_i}$ follows \eqref{Eq:Def:J_eta}, and $\mathbb{X}_k$ is the set of nearest neighbor paths in $\mathbb{Z}_+$ with $2k+1$ steps that
start at $1$, end at $0$ and that do not visit $0$ during the first $2k$ steps, defined by
\begin{align}
 \mathbb{X}_k:=&\Big\{x=(x_i)_{i=0}^{2k+1}\in \mathbb{Z}_+^{2k+1}:\,x_0=1,\,x_{2k+1}=0,\, \notag \\
 &\qquad\qquad x_i\geq 1 \text{ and } x_{i}-x_{i-1}\in\{-1,1\} \text{ for }1\leq i\leq 2k\Big\}. \label{Theorem:etac_Y-w:Eq:Def:pathsX_k}
\end{align}
\end{theorem}

We first give a  representation of  $w_{-\eta}(z_j)$  which will be useful in the proof of Theorem
\ref{Theorem:etac_Y-w} and other places.

Suppose $Y_0=z_j>0$. Then the embedded random walk of $Y$ takes $j+2k$ many steps to hit zero for some $k\in\mathbb{Z}_+$.
In this event, there are exactly $k$ steps to the right and $j+k$ steps to the left, in which the last step is
to the left, and during the first $2k$ steps the path does not touch $1$. The set of such left-right paths
(left $=-1$, right $=1$) is denoted by $\mathbb{X}_{j,k}$.

\begin{lemma}\label{Lemma:w(1)_gen}
For $\eta\in \mathbb{R}$ and $j\geq 1$,
\begin{equation*}
       w_{-\eta}(z_j)
    = \sum_{k\geq 0} \sum_{x\in \mathbb{X}_{j,k}} \prod_{i=1}^{2k+j}J^{x_{i-1}}_{\eta, \,x_{i}-x_{i-1}} \in (0,+\infty],
\end{equation*}
 where $\mathbb{X}_{j,k}$ is the set of nearest neighbor paths in $\mathbb{Z}_+$ with $2k+j$ steps
 that start at $j$,  end at $0$ and that do not visit $0$ during the first $2k+(j-1)$ steps.
\end{lemma}

\begin{proof}[Proof of Lemma \ref{Lemma:w(1)_gen}]
Let $j=1$. From \eqref{Eq:tauYX_gen} we have,
\begin{align}
    w_{-\eta}(z_1)=&E^{(\vec{p},\vec{z})}\left[e^{\eta T_0^{z_1}} 1_{ T_0^{z_1}<\infty}\right] = E^{(\vec{p},\vec{z})}\left[e^{\eta \sum_{i=1}^{\tau^X}(\tau_{i}-\tau_{i-1})} 1_{ \tau^X<\infty}\right]\notag \\
    =& \sum_{k\geq 0} \sum_{x\in \mathbb{X}_k}E^{(\vec{p},\vec{z})}\left[e^{\eta \sum_{i=1}^{2k+1}(\tau_{i}-\tau_{i-1})} \,1_{\{X_{i}-X_{i-1}=x_i-x_{i-1} \text{ for } 1\leq i\leq 2k+1\}} \right]. \label{Eq:w(1)_gen}
\end{align}

Recall $S^i_{+1}=\{S^{i}_{\sigma(i)}= \ell_{i}\}$ and  $S^i_{-1}=\{S^{i}_{\sigma(i)}= -\ell_{i-1}\}$,
where $\sigma(i)$ is equal in distribution to the exit time of $Y$ starting at $z_i$ from the interval $(z_{i-1},z_{i+1})$.

By conditioning at $\tau_i$ successively and the strong Markov property of $Y$, a  term on the right of \eqref{Eq:w(1)_gen} is
\begin{align*}
E^{(\vec{p},\vec{z})}\left[e^{\eta \sum_{i=1}^{2k+1}(\tau_{i}-\tau_{i-1})} \,\prod_{i=1}^{2k+1}1_{\{X_{i}-X_{i-1}=x_{i}-x_{i-1}\}} \right] =& \prod_{i=1}^{2k+1}E\left[e^{\eta \,\sigma(x_{i-1})} \,{\bf 1}_{\{S^{x_{i-1}}_{x_{i}-x_{i-1}}\}}\right]\\
=& \prod_{i=1}^{2k+1}J^{x_{i-1}}_{\eta, \,x_{i}-x_{i-1}},
\end{align*}
where $J^{i}_{\eta, \,\pm 1}=E\left[e^{\eta \,\sigma(i)} \,\big|\, S^{i}_{\pm 1}\right] p^i_{\pm 1}$
is defined in \eqref{Eq:Def:J_eta}.

Putting the last display into \eqref{Eq:w(1)_gen}, we obtain the lemma for the case $j=1$. The general case $j>1$ follows the same proof.
\end{proof}

From Lemma \ref{Lemma:w(1)_gen}, \eqref{Theorem:etac_Y-w:Eq:etaYc_rep1} follows
from the elementary root test if $\lim\limits_{k\to\infty}$ were replaced by $\limsup\limits_{k\to\infty}$.
Lemma \ref{Lemma:Subergodic_existence} below shows that the limit indeed exists.

\begin{lemma}\label{Lemma:Subergodic_existence}
The limit
\begin{equation}\label{Lemma:Subergodic_existence:Eq:Subergodic_existence}
\Theta_{\eta}:= \lim_{k\to\infty}\left(\sum_{x\in \mathbb{X}_k}\prod_{i=0}^{2k}J^{x_i}_{\eta,\,x_{i+1}-x_{i}}\right)^{1/k} \in [0,\infty]
\end{equation}
exists for all  $\eta\in\mathbb{R}$, is non-decreasing and is strictly increasing in $\eta$ when it is finite.
\end{lemma}

\begin{proof}
We shall apply Kingman's subadditive ergodic theorem in the same way it is applied
to prove existence of limiting free energy in random polymer models; see for instance
\cite[Theorems 2.2 and 2.4]{Ras-Sep-14}.

For $0\leq n\leq m$ we consider the point-to-point partition function
\begin{equation}\label{Def:patition}
Z_{n,m}:= \sum_{y\in \Pi_{n,m}}\prod_{i=n}^{m}J^{y_i}_{\eta,\,y_{i+1}-y_{i}}=
\sum_{y\in \Pi_{0,m-n}}\prod_{i=0}^{m-n}J^{y_i}_{\eta,\,y_{i+1}-y_{i}} \ ,
\end{equation}
where $\Pi_{n,m}$ is the set of nearest neighbor paths in $\mathbb{N}$ that starts and ends at 1 during the time interval $[n,m]$, that is,
\begin{align*}
\Pi_{n,m}:=&\Big\{y=(y_i)_{i=n}^{m}:\;y_n=y_m=1,\, 
 y_j\geq 1 \text{ and } y_{j}-y_{j-1}\in\{-1,1\} \text{ for }n\leq j\leq m \Big\}.
\end{align*}
For any $0\leq a\leq b\leq c$ we have $Z_{a,c} \geq Z_{a,b}\,Z_{b,c}$,
because  concatenating a path in $\Pi_{a,b}$ with a path in  $\Pi_{b,c}$ gives a path in $\Pi_{a,c}$.
This gives sub-additivity $\ln Z_{a,c} \geq \ln Z_{a,b} + \ln Z_{b,c}$,
from which existence of the limiting ``point-to-point free energy"
\begin{equation}\label{Eq:R_X finite2}
\lim_{k\to\infty}\frac{1}{k} \ln Z_{0,2k}=\lim_{k\to\infty}\frac{1}{k} \ln\left(\sum_{x\in \mathbb{X}_k}\prod_{i=0}^{2k}J^{x_i}_{\eta,\,x_{i+1}-x_{i}}\right)\in [-\infty,\infty]
\end{equation}
follows from Kingman's subadditive ergodic theorem \cite[Theorem 2.6 on page 277]{liggett2012interacting}.

Monotonicity of $\Theta_{\eta}$ follows from the fact that $J^{i}_{\eta, \,\pm 1}$
is increasing in $\eta$. The strict inequality then follows from
the fact that, under Lemma \ref{Assumption:ErgodicEnvironments},
\[
\min_{i\geq 0}\left(J^{i}_{\eta_2, \,\pm 1}-J^{i}_{\eta_1, \,\pm 1}\right)>0 \ ,
\]
for all $\eta_2>\eta_1$ such that the limits $\Theta_{\eta_1}$ and $\Theta_{\eta_2}$ in
\eqref{Lemma:Subergodic_existence:Eq:Subergodic_existence} are finite.
\end{proof}

\begin{proof}[Proof of Theorem \ref{Theorem:etac_Y-w}]

From the series representation in Lemma \ref{Lemma:w(1)_gen}, \eqref{Theorem:etac_Y-w:Eq:etaYc_rep1} follows from the
elementary root test and the existence of limit in Lemma \ref{Lemma:Subergodic_existence}.
The uniqueness of $\eta_c^w$ in Theorem \ref{Theorem:etac_Y-w} then follows from strict monotonicity of the
function $\eta\mapsto \Theta_{\eta}$ stated
in Lemma \ref{Lemma:Subergodic_existence}.

Observe that, for each path $x\in \mathbb{X}_k$, we have (i) $x_{2k}=0$ and (ii) the number of steps from $j$ to $j+1$ is the same as the number of steps from $j+1$ to $j$ for all $j\geq 1$. So for each $x\in \mathbb{X}_k$, there exists an index set $\{j_i\}$ such that
\begin{equation}\label{Eq:pairJ}
\prod_{i=1}^{2k+1}J^{x_{i-1}}_{\eta, \,x_{i}-x_{i-1}}
=\left(\prod_{i=1}^{k}J^{x_{j_i}}_{\eta, \,+1}J^{x_{j_i}+1}_{\eta, \,-1}\right)\,J^{0}_{\eta, \,-1}.
\end{equation}

Putting \eqref{Eq:pairJ} into Lemma \ref{Lemma:w(1)_gen}, we obtain the representation
\begin{align}
    w_{-\eta}(\ell_0)
    =& \sum_{k\geq 0} \sum_{x\in \mathbb{X}_k} \left(\prod_{i=1}^{k}J^{x_{j_i}}_{\eta, \,+1}J^{x_{j_i}+1}_{\eta, \,-1}\right)\,J^{0}_{\eta, \,-1} \ . \label{Eq:w(1)_gen2}
\end{align}

By Lemma \ref{Assumption:ErgodicEnvironments}, there exists $\widetilde{B}\in [0,\infty)$  such that
\begin{align}
  J^{\min}_{\eta}:=\min_{i\geq 1} J^{i}_{\eta,1}\,J^{i+1}_{\eta,-1} > \frac{1}{4} \quad\text{for all }\eta\in[\widetilde{B},\,\infty).
  \label{Theorem:etac_Y-w:Proof:Eq:minpqY}
\end{align}
Then from \eqref{Eq:w(1)_gen2} we derive that for all $\eta\geq \widetilde{B}$ we have
\begin{align}\label{Theorem:etac_Y-w:Proof:Eq:LowerBound_wY}
w_{-\eta}(z_1)\geq & J^{0}_{\eta, \,-1}\sum_{k\geq 0} C_k\,\left(J^{\min}_{\eta}\right)^k=+\infty \ ,
\end{align}
where we have used the well-known fact that the number of paths $|\mathbb{X}_k|$ is the $k$-th Catalan number
$C_k=\dfrac{1}{k+1}\dfrac{(2k)!}{k! k!}$ (see, for example, Corollary 6.2.3 and page 223 of \cite{stan-II}).
By properties of the Catalan number $C_k$, the above series is equal to $+\infty$
for $\eta\in [\widetilde{B},+\infty)$. Thus $\eta_c^w< \widetilde{B}<\infty$.
Since for any $\eta<0$ we have automatically $w_{-\eta}(x)<\infty$, we also know that $\eta_c^w\geq 0$.
Hence $\eta_c^w\in \left[0,\,\widetilde{B}\right) \subset [0,\infty)$.
\end{proof}

The following corollary provides a mild condition \eqref{Corollary:etac_Y-w-positive:Eq:maxpq}
 under which $\eta_c^w\in (0,\infty)$
(see Remark \ref{Remark:etacY-w}).

\begin{corollary}\label{Corollary:etac_Y-w-positive}
If
\begin{equation}\label{Corollary:etac_Y-w-positive:Eq:maxpq}
\max_{i\geq 1}p^i_{+1} p^{i+1}_{-1} < 1/4 \ ,
\end{equation}
then $\eta_c^w\in  (0,\infty)$.
\end{corollary}

\begin{proof}
Let $J^{\max}_{\eta}:=\max\limits_{i\geq 1} J^{i}_{\eta,1}\,J^{i+1}_{\eta,-1}$.
If \eqref{Corollary:etac_Y-w-positive:Eq:maxpq} holds, then because $J^{i}_{\eta, \,\pm 1}$ is
monotonically increasing in $\eta$, there exists a unique $\widetilde{A}\in [0,\infty)$ such that
for any $\eta\leq \widetilde{A}$ we have
\begin{align}
J_\eta^{\max}\leq  J^{\max}_{\widetilde{A}}=\max_{i\geq 1}
  J^{i}_{\widetilde{A},1}\,J^{i+1}_{\widetilde{A},-1} = \frac{1}{4}.
  \label{Corollary:etac_Y-w-positive:Proof:Eq:maxpqY}
\end{align}

From \eqref{Eq:w(1)_gen2}, we have
\begin{align}\label{Eq:UpperBound_wY}
w_{-\eta}(z_1)\leq & J^{0}_{\eta, \,-1}\sum_{k\geq 0} C_k\,\left(J^{\max}_{\eta}\right)^k \notag\\
=& J^{0}_{\eta, \,-1}\frac{1-\sqrt{1-4\,J^{\max}_{\eta}}}{2\,J^{\max}_{\eta}}  \in (0,\infty), \quad\text{for }\eta\in (-\infty, \widetilde{A}],
\end{align}
where we have used again properties of the Catalan number $C_k$
(see, for example, Corollary 6.2.3 and page 223 of \cite{stan-II}).
Thus under \eqref{Corollary:etac_Y-w-positive:Eq:maxpq} we have $\eta_c^w> \widetilde{A}>0$.
\end{proof}

\begin{remark}\rm \label{Remark:etacY-w}
The condition \eqref{Corollary:etac_Y-w-positive:Eq:maxpq} holds, for instance, if $\ell_i=\ell$ are constant
for all $i\geq 0$ and $p_i\geq 2/3$ ($d_i\geq 3$) for all $i\geq 1$.
\end{remark}

\section{Construction and analysis of the auxiliary function}\label{Sec:AuxiliaryFunctions}

In this section we provide more constructions and analysis of the auxiliary function
$$w(x)\equiv
w_{\lambda}(x)\equiv E^{(\vec{p},\vec{z})}\left[e^{-\lambda T_0^x}\mathbf{1}_{T_0^x<\infty}\right] \ , \ x\in \mathbb{R}
$$
that we introduced
in \eqref{Eq:w-AuxiliaryFunction} in the particular case when $\lambda>0$.
This will be useful in the analysis of
large deviations principle for the multi-skewed Brownian motion $Y_t$.
Our analysis will be based on some properties of the limit of an infinite
product of $2\times 2$ random matrices (see \cite{Bougerol-Lacroix} for a general reference on this topic).

Recall
the hitting time $T^s_r$ introduced in \eqref{Eq:StoppingTimesTorMultiSkewedBM} and notice that when $\lambda>0$,
\begin{equation}\label{Eq:w-AuxiliaryFunction:lambda>0EqualPDESolution}
w(x)= E^{(\vec{p},\vec{z})}\left[e^{-\lambda T_0^x}\right] \ , \ x\in \mathbb{R} \ .
\end{equation}

\subsection{Results for deterministic skewness and barriers}\label{Sec:AuxiliaryFunctions:Subsection:Deterministic-p-z}

Set $\vec{z}:=\{z_i\}_{i\in \mathbb{Z}}$ to be a set of real numbers with no accumulation point
and $\vec{p}:=\{p_i\}_{i\in \mathbb{Z}} \subset (0,1)$. Within this subsection we assume there
is no randomness in either $\vec{z}$ or $\vec{p}$.


Let $\{P_t\}_{t\geq 0}$ be the semigroup of the process $Y_t$.
Generalizing the approach of \cite{dereudre}, one can check that for any
$f\in \mathcal{C}_{b}(\mathbb{R})$, the function
 $F(t,x):=P_tf(x)$ is the solution in $\mathcal{C}^{1,2}((0,\infty)\times \mathbb{R}\setminus \vec{z},\mathbb{R})\cap \mathcal{C}([0,\infty) \times \mathbb{R},\mathbb{R})$ of
\begin{equation}\label{Eq:SemiGroupMultiSkewedBM}
\left\{\begin{array}{l}
\dfrac{\partial}{\partial t}F(t,x)=\dfrac{1}{2}\dfrac{\partial^2}{\partial x^2}F(t,x) \ ,
\text{ for }t\in (0,\infty), x\in \mathbb{R}\setminus{\{z_i\}_{i\in\mathbb{Z}}}
\\
F(0,x)=f(x), \text{ for } x\in \mathbb{R}
\\
F(t,z_i^+)=F(t,z_i^-), \text{ for }t\in (0,\infty)\text{ and } \forall i\in\mathbb{Z}
\\
p_i\partial_x F(t,z_i^+)=(1-p_i)\partial_x F(t,z_i^-), \text{ for }t\in (0,\infty)\text{ and } \forall i\in\mathbb{Z}\setminus 0
\end{array}\right.
\end{equation}
Here $f(z^-)$ and $f(z^+)$ denote respectively the left sided limit and the right sided limit of a function $f$ at $z$.

We make use of some ideas in \cite{FreidlinHu13} and \cite{Nolen2009} in the following analysis.
The following Proposition gives an explicit formula for the function $w(x)=w_{\lambda}(x)$ defined in
\eqref{Eq:w-AuxiliaryFunction} in case when $\lambda>0$. We shall use the notation
$\prod_{i=k-1}^{1}M_i=M_{k-1}M_{k-2}\cdots M_1$ and the convention that it is the
identity matrix when $k=1$.

\begin{proposition}\label{Prop:ExplicitCalculationw-AuxiliaryDeterministic}
Let $\lambda\in(0,\infty)$. Let $\vec{z}:=\{z_i\}_{i\in \mathbb{Z}}$ be a set of real numbers
and $\vec{p}:=\{p_i\}_{i\in \mathbb{Z}} \subset (0,1)$.
Let $\ell_k=z_{k+1}-z_k$ and $M_k$ be the matrix
\begin{equation}\label{Prop:ExplicitCalculationw-AuxiliaryDeterministic:Eq:StructureMatrixM}
M_k:=\frac1{2p_k}
\begin{pmatrix}
e^{\sqrt{2\lambda}\ell_k} & (2p_k-1)e^{\sqrt{2\lambda}\ell_k}\\
(2p_k-1)e^{-\sqrt{2\lambda}\ell_k} & e^{-\sqrt{2\lambda}\ell_k}
\end{pmatrix}.
\end{equation}
Consider the sum of column entries of the product matrices:
\begin{equation}\label{Prop:ExplicitCalculationw-AuxiliaryDeterministic:Eq:SumColumnPositive}
\begin{pmatrix}
L_k &R_k\end{pmatrix}:=\begin{pmatrix}
1&1
\end{pmatrix}
\prod_{i=k-1}^{1}M_i \quad \text{for }k\geq 1.
\end{equation}
Suppose the following condition hold:
\begin{equation}\label{Prop:ExplicitCalculationw-AuxiliaryDeterministic:Eq:infpositive}
(\liminf_{k\to\infty}R_k) \vee (\liminf_{k\to\infty}L_k) \in(0,\infty] \ ,
\end{equation}
and there exist (possibly $+\infty$) limit
\begin{equation}\label{Prop:ExplicitCalculationw-AuxiliaryDeterministic:Eq:xi}
\xi=\xi_{\lambda}\equiv \lim\limits_{k\to\infty} \dfrac{L_k}{R_k}\in[0,\infty].
\end{equation}
Then the function $w(x)=w_{\lambda}(x)$ defined in \eqref{Eq:w-AuxiliaryFunction} is explicitly given as follows: $w(x)=w(-x)$ for $x\in \mathbb{R}$ and on $[0,\infty)$,
\begin{equation}\label{Prop:ExplicitCalculationw-AuxiliaryDeterministic:Eq:wExplicit}
w(x)=f^+_ke^{\sqrt{2\lambda}(x-z_k)}+f^-_ke^{-\sqrt{2\lambda}(x-z_k)}\ , \quad x\in [z_{k-1},z_k]\ , \quad k\geq 1 \ ,
\end{equation}
where $\vec{f}_k=(f^+_k,\,f^-_k)$ are given by
\begin{equation}\label{Prop:ExplicitCalculationw-AuxiliaryDeterministic:Eq:VecfPositive}
\vec{f}_1=
\dfrac{1}{e^{\sqrt{2\lambda}z_1}\xi-e^{-\sqrt{2\lambda}z_1}}
\begin{pmatrix}
-1\\ \xi
\end{pmatrix}
\quad\text{and}\quad
\vec{f}_k=\frac12\left(\prod_{i=k-1}^{1}M_i\right)\vec{f}_1 \quad\text{for }k\geq 2 \ .
\end{equation}
\end{proposition}

\begin{proof}[Proof of Proposition \ref{Prop:ExplicitCalculationw-AuxiliaryDeterministic}]
The fact that $w(x)=w(-x)$ for $x\in \mathbb{R}$ follows from Lemma \ref{Lemma:ElementaryPropertiesw}.

Fix $\lambda>0$. The restriction of $w(x)=w_{\lambda}(x)$ on $x\in \mathbb{R}_+=[0,\infty)$ is the continuous solution of
the following Sturm-Liouville problem on $\mathbb{R}_+$
with skew boundary conditions on the set  $\{z_i\}_{i\geq 1}$:
\begin{equation}\label{Prop:ExplicitCalculationw-AuxiliaryDeterministic:Proof:Eq:PDE}
\left\{\begin{array}{ll}
0=\dfrac{1}{2}\partial_{xx}^2 w(x)-\lambda w(x) \ , & \text{ for }t\in (0,\infty), \,x\in (0,\infty)\setminus{\{z_i\}_{i\geq 1}} \ ,
\\
p_i \partial_x w(z_i+)=(1-p_i)\partial_x w(z_i-) \ , & \text{ for } i\geq 1 \ ,
\\
w(z_i+)=w(z_i-), & \text{ for } i\geq 1 \ ,
\\
w(0)=1  \ , &
\\
\lim\limits_{x\rightarrow +\infty} w(x)=0 \ . &
\end{array}\right.
\end{equation}

From the first equation of \eqref{Prop:ExplicitCalculationw-AuxiliaryDeterministic:Proof:Eq:PDE},
The function $w$ satisfies the eigenvalue problem
\begin{equation}\label{Prop:ExplicitCalculationw-AuxiliaryDeterministic:Proof:Eq:eigen_k}
0=\dfrac{1}{2}\partial_{xx}^2w(x)-\lambda w(x), \quad x\in (z_{k-1},\,z_k)
\end{equation}
for all $k\geq 1$. This linear second order ODE \eqref{Prop:ExplicitCalculationw-AuxiliaryDeterministic:Proof:Eq:eigen_k}
has general solution
\begin{equation}\label{Prop:ExplicitCalculationw-AuxiliaryDeterministic:Proof:Eq:GenSol_k}
w(x):=f^+_ke^{\sqrt{2\lambda}(x-z_k)}+f^-_ke^{-\sqrt{2\lambda}(x-z_k)}\quad \text{for } x \in(z_{k-1},\,z_k) \ ,
\end{equation}
where $f^\pm_k$ are constants to be determined by the boundary conditions in
\eqref{Prop:ExplicitCalculationw-AuxiliaryDeterministic:Proof:Eq:PDE}.
The functions $w_k(x)=w(x) \text{ for } x\in (z_{k-1}, z_k)$ can be extended continuously to the end points of the interval $[z_{k-1},z_k]$.
The collection $\{w(z_k)\}_{k\geq 0}$ satisfies boundary conditions of
\eqref{Prop:ExplicitCalculationw-AuxiliaryDeterministic:Proof:Eq:PDE} that can be stated as the following:
\begin{equation}\label{Prop:ExplicitCalculationw-AuxiliaryDeterministic:Proof:Eq:patching}
\left\{\begin{array}{ll}
w(z_{k}-)=w(z_k+)\equiv w(z_k), &  \text{for } k\geq 1 \ ,
\\
(1-p_{k})\partial_x w(z_{k}-)=p_{k}\partial_x w(z_{k}+), & \text{for } k\geq 1  \ ,
\\
\lim\limits_{x\downarrow 0}w(x)=w(0)=1 \ , &
\\
\lim\limits_{k \rightarrow +\infty} w(z_k)= 0   \ . &
\end{array}\right.
\end{equation}

Putting \eqref{Prop:ExplicitCalculationw-AuxiliaryDeterministic:Proof:Eq:GenSol_k} into \eqref{Prop:ExplicitCalculationw-AuxiliaryDeterministic:Proof:Eq:patching}, we obtain
the following system of equations for the unknown $\{f^{+}_k,\,f^{-}_k\}_{k\geq 1}$.
\begin{equation}\label{Prop:ExplicitCalculationw-AuxiliaryDeterministic:Proof:Eq:eq with f_k}
\left\{\begin{array}{ll}
f^+_k+f^-_k=f^+_{k+1}e^{-\sqrt{2\lambda}(z_{k+1}-z_k)}+f^-_{k+1}e^{\sqrt{2\lambda}(z_{k+1}-z_k)}, &  \text{for } k\geq 1 \ ,
\\
(1-p_k) (f^+_k-f^-_k)=p_{k}\left(f^+_{k+1}e^{-\sqrt{2\lambda}(z_{k+1}-z_k)}-f^-_{k+1}e^{\sqrt{2\lambda}(z_{k+1}-z_k)}\right), &
\text{for } k\geq 1 \ ,
\\
f^+_1e^{-\sqrt{2\lambda}z_1}+f^-_1e^{\sqrt{2\lambda}z_1}=1 \ , &
\\
f^+_k+f^-_k \to 0 \text{ as } k\to\infty \ . &
\end{array}\right.
\end{equation}
Let $s:=w'(0+)$ then we have that
$$
\left\{\begin{array}{l}
f^+_1e^{-\sqrt{2\lambda}z_1}+f^-_1e^{\sqrt{2\lambda}z_1}=1 \ ,
\\
\sqrt{2\lambda}\left(f^+_1e^{-\sqrt{2\lambda}z_1}-f^-_1e^{\sqrt{2\lambda}z_1}\right)=s \ .
\end{array}\right.
$$
Thus,
\begin{equation}\label{Prop:ExplicitCalculationw-AuxiliaryDeterministic:Proof:Eq:f_1}
\left\{\begin{array}{l}
f^+_1=e^{\sqrt{2\lambda}z_1}\left(\frac12+\frac{s}{\sqrt{8\lambda}}\right) \ ,
\\
f^-_1=e^{-\sqrt{2\lambda}z_1}\left(\frac12-\frac{s}{\sqrt{8\lambda}}\right) \ .
\end{array}\right.
\end{equation}
From the first two equations of \eqref{Prop:ExplicitCalculationw-AuxiliaryDeterministic:Proof:Eq:eq with f_k} for $k\geq 1$, we get
\begin{equation}\label{Prop:ExplicitCalculationw-AuxiliaryDeterministic:Proof:Eq:f^k_pm}
\left\{\begin{array}{l}
f^+_{k+1}=e^{\sqrt{2\lambda}(z_{k+1}-z_k)}\dfrac{1}{2p_k}\left(f^+_k+(2p_k-1)f^-_k\right) \ ,
\\
f^-_{k+1}=e^{-\sqrt{2\lambda}(z_{k+1}-z_k)}\dfrac{1}{2p_k}\left(f^+_k(2p_k-1)+f^-_k\right) \ .
\end{array}\right.
\end{equation}
To simplify notation, we let $\vec{f}_k=(f^+_k,f^-_k)^T$ to be the transpose of $(f^+_k,f^-_k)$.
Then for all $k\in \mathbb{Z} \setminus \{0\}$, we have
$$
\vec{f}_{k+1}
= M_k\,\vec{f}_{k} \ ,
$$
where
$$
M_k:=\dfrac{1}{2p_k}
\begin{pmatrix}
e^{\sqrt{2\lambda}(z_{k+1}-z_k)} & (2p_k-1)e^{\sqrt{2\lambda}(z_{k+1}-z_k)}\\
(2p_k-1)e^{-\sqrt{2\lambda}(z_{k+1}-z_k)} & e^{-\sqrt{2\lambda}(z_{k+1}-z_k)}
\end{pmatrix} \ .
$$

Iterating this equation, we have
$\vec{f}_k=\left(M_{k-1}M_{k-2}\cdots M_1\right)\vec{f}_1$ for $k\geq 2$.
i.e.
\begin{equation}\label{Prop:ExplicitCalculationw-AuxiliaryDeterministic:Proof:Eq:f_k for k>0}
\vec{f}_k=\left(\prod_{i=k-1}^{1}M_i\right)
\begin{pmatrix}
e^{\sqrt{2\lambda}z_1}\left(\frac12+\frac{s}{\sqrt{8\lambda}}\right)\\e^{-\sqrt{2\lambda}z_1}\left(\frac12-\frac{s}{\sqrt{8\lambda}}\right)
\end{pmatrix} \quad  \ , \ \text{ for } k\geq 2  \ .
\end{equation}

Now by the fifth condition of \eqref{Prop:ExplicitCalculationw-AuxiliaryDeterministic:Proof:Eq:eq with f_k}
we have $w(z_k)=f^+_k+f^-_k \to 0 \text{ as } k\to\infty$, so
\begin{equation}\label{Prop:ExplicitCalculationw-AuxiliaryDeterministic:Proof:Eq:finding s}
0=\lim_{k\to \infty}\begin{pmatrix}
1&1
\end{pmatrix}\left(\prod_{i=k-1}^{1}M_i\right)\begin{pmatrix}
e^{\sqrt{2\lambda}z_1}\left(\frac12+\frac{s}{\sqrt{8\lambda}}\right)\\e^{-\sqrt{2\lambda}z_1}\left(\frac12-\frac{s}{\sqrt{8\lambda}}\right)
\end{pmatrix} \ .
\end{equation}
Suppose $L_k$ and $R_k$ are real numbers such that
$$
\begin{pmatrix}
1&1
\end{pmatrix}
\prod_{i=k-1}^{1}M_i=\begin{pmatrix}
L_k & R_k\\
\end{pmatrix} \ .
$$
Then from \eqref{Prop:ExplicitCalculationw-AuxiliaryDeterministic:Proof:Eq:finding s}, we have
$$
\begin{array}{ll}
0&=\lim\limits_{k\to\infty} e^{\sqrt{2\lambda}z_1}\left(\frac12+\frac{s}{\sqrt{8\lambda}}\right)L_k +e^{-\sqrt{2\lambda}z_1}\left(\frac12-\frac{s}{\sqrt{8\lambda}}\right)R_k \notag
\\
&= \lim\limits_{k\to\infty}\frac{1}{2}\left(e^{\sqrt{2\lambda}z_1}\,L_k+ e^{-\sqrt{2\lambda}z_1}\,R_k\right) \,+\,
\frac{1}{\sqrt{8\lambda}}\left(e^{\sqrt{2\lambda}z_1}\,L_k- e^{-\sqrt{2\lambda}z_1}\,R_k \right)\, s \ .
\end{array}
$$
Now by  assumptions \eqref{Prop:ExplicitCalculationw-AuxiliaryDeterministic:Eq:infpositive} and
\eqref{Prop:ExplicitCalculationw-AuxiliaryDeterministic:Eq:xi}, it follows that
\begin{equation}\label{Prop:ExplicitCalculationw-AuxiliaryDeterministic:Proof:Eq:s as a limit}
s=\sqrt{2\lambda}\frac{e^{\sqrt{2\lambda}z_1}\,\xi+ e^{-\sqrt{2\lambda}z_1}}{e^{-\sqrt{2\lambda}z_1}-e^{\sqrt{2\lambda}z_1}\,\xi}
\end{equation}
for $\xi\in [0,\infty]$ (when $\xi=+\infty$, $s=-\sqrt{2\lambda}$).

Thus combining \eqref{Prop:ExplicitCalculationw-AuxiliaryDeterministic:Proof:Eq:f_1},
\eqref{Prop:ExplicitCalculationw-AuxiliaryDeterministic:Proof:Eq:f_k for k>0} and
\eqref{Prop:ExplicitCalculationw-AuxiliaryDeterministic:Proof:Eq:s as a limit} we get

$$
\vec{f}_1=
\dfrac{1}{2}
\begin{pmatrix}
e^{\sqrt{2\lambda}z_1}\left(1+\frac{e^{\sqrt{2\lambda}z_1}\,\xi+ e^{-\sqrt{2\lambda}z_1}}{e^{-\sqrt{2\lambda}z_1}-e^{\sqrt{2\lambda}z_1}\,\xi}\right)
\\
e^{-\sqrt{2\lambda}z_1}\left(1-\frac{e^{\sqrt{2\lambda}z_1}\,\xi+ e^{-\sqrt{2\lambda}z_1}}{e^{-\sqrt{2\lambda}z_1}-e^{\sqrt{2\lambda}z_1}\,\xi}\right)
\end{pmatrix}
\quad\text{and}\quad
\vec{f}_k=\frac12\left(\prod_{i=k-1}^{1}M_i\right)\vec{f}_1 \quad\text{for }k\geq 2 \ .
$$

Notice that
$$e^{\sqrt{2\lambda}z_1}\left(1+\frac{e^{\sqrt{2\lambda}z_1}\,\xi+ e^{-\sqrt{2\lambda}z_1}}{e^{-\sqrt{2\lambda}z_1}-e^{\sqrt{2\lambda}z_1}\,\xi}\right)=
e^{\sqrt{2\lambda}z_1}\left(\dfrac{2e^{-\sqrt{2\lambda}z_1}}{e^{-\sqrt{2\lambda}z_1}-e^{\sqrt{2\lambda}z_1}\,\xi}\right)
=\dfrac{2}{e^{-\sqrt{2\lambda}z_1}-e^{\sqrt{2\lambda}z_1}\,\xi}$$
and
$$e^{-\sqrt{2\lambda}z_1}\left(1-\frac{e^{\sqrt{2\lambda}z_1}\,\xi+ e^{-\sqrt{2\lambda}z_1}}{e^{-\sqrt{2\lambda}z_1}-e^{\sqrt{2\lambda}z_1}\,\xi}\right)
=e^{-\sqrt{2\lambda}z_1}\left(\dfrac{-2e^{\sqrt{2\lambda}z_1}\xi}{e^{-\sqrt{2\lambda}z_1}-e^{\sqrt{2\lambda}z_1}\,\xi}\right)
=\dfrac{-2\xi}{e^{-\sqrt{2\lambda}z_1}-e^{\sqrt{2\lambda}z_1}\,\xi} \ ,$$
we obtain \eqref{Prop:ExplicitCalculationw-AuxiliaryDeterministic:Eq:VecfPositive}, which completes the proof of this Proposition.
\end{proof}

\subsection{Results for random skewness and barriers}\label{Sec:AuxiliaryFunctions:Subsection:Random-p-z}

In Section \ref{Sec:AuxiliaryFunctions:Subsection:Deterministic-p-z} $(\vec{p},\vec{z})$ are not random. Now we suppose these are random vectors under
$\mathbf{P}$ which satisfy Lemma \ref{Assumption:ErgodicEnvironments}.
We want to verify conditions \eqref{Prop:ExplicitCalculationw-AuxiliaryDeterministic:Eq:infpositive} and \eqref{Prop:ExplicitCalculationw-AuxiliaryDeterministic:Eq:xi}
in Proposition \ref{Prop:ExplicitCalculationw-AuxiliaryDeterministic} on
$\{M_i\}$ in order to obtain an explicit formula for $w(x)$ defined in \eqref{Eq:w-AuxiliaryFunction} in the case $\lambda>0$.

Recall that $\{p_i\}_{i\geq 1}$ are i.i.d. random variables that take values
in $[\frac{1}{2},1)$ and $\{z_{i+1}-z_{i}\}=\{\ell_i\}_{i\geq 0}$ are also i.i.d. random variables taking positive values.
Moreover, 
\begin{equation}\label{Eq:M_i iid}
\begin{array}{ll}
M_i&=\dfrac{1}{2p_i}
\begin{pmatrix}
e^{\sqrt{2\lambda}\ell_i} & (2p_i-1)e^{\sqrt{2\lambda}\ell_i}\\
(2p_i-1)e^{-\sqrt{2\lambda}\ell_i} & e^{-\sqrt{2\lambda}\ell_i}
\end{pmatrix} \\
&=\dfrac{1}{2p_i \gamma_i}
\begin{pmatrix}
\gamma_i^2 & \zeta_i \gamma_i^2\\
\zeta_i & 1
\end{pmatrix} \ ,
\end{array}
\end{equation}
in which we denote $\gamma_i:=e^{\sqrt{2\lambda}\ell_i}\in (1,\infty)$ and $\zeta_i=2p_i-1 \in [0,1)$ for simplicity.
 Recall also that
$$\begin{pmatrix}L_k&R_k\end{pmatrix}=\begin{pmatrix}1&1\end{pmatrix}\prod_{i=k-1}^{1} M_i \in \mathbb{R}_+^2\quad\text{ for }k\geq 2.$$

The ``backward" process
$\left\{\begin{pmatrix}L_k&R_k\end{pmatrix}\right\}_{k\geq 2}$ is \textit{not} a Markov chain, but the corresponding  ``forward" process is.
That is
\begin{equation}\label{Eq:tild Lk}
\begin{pmatrix}
\widetilde{L}_k&\widetilde{R}_k
\end{pmatrix}:=\begin{pmatrix}1&1\end{pmatrix}\prod_{i=1}^{k-1} M_i,\quad k\geq 2
\end{equation}
is a Markov chain, with  iterative relation $\begin{pmatrix}\widetilde{L}_{k+1}&\widetilde{R}_{k+1} \end{pmatrix}
=\begin{pmatrix}\widetilde{L}_k&\widetilde{R}_k \end{pmatrix}M_k$. Furthermore, $\begin{pmatrix}L_k&R_k\end{pmatrix}\stackrel{d}{=}\begin{pmatrix}
\widetilde{L}_k&\widetilde{R}_k
\end{pmatrix}$ in $\mathbb{R}^2_+$ for each $k\geq 2$.

The following lemma ensures that condition \eqref{Prop:ExplicitCalculationw-AuxiliaryDeterministic:Eq:infpositive}
in Proposition \ref{Prop:ExplicitCalculationw-AuxiliaryDeterministic} is verified for the i.i.d. case.

\begin{lemma}\label{Lemma:LRtildeMonotonicityOrder}
With probability one,
$\widetilde{L}_{k+1}+\widetilde{R}_{k+1}> \widetilde{L}_{k}+\widetilde{R}_{k}$ and $\widetilde{L}_{k}>\widetilde{R}_{k} > 0$ for all $k\geq 2$.
\end{lemma}

\begin{proof}
Define
$\begin{pmatrix}
\widetilde{L}_1&\widetilde{R}_1
\end{pmatrix}=\begin{pmatrix}
1 & 1
\end{pmatrix}$.
Then the following iterations hold for $k\geq 1$:
$$\begin{pmatrix}\widetilde{L}_{k+1}&\widetilde{R}_{k+1} \end{pmatrix}
=\begin{pmatrix}\widetilde{L}_k&\widetilde{R}_k \end{pmatrix}M_k=\frac{1}{2 p_k \gamma_k}
\begin{pmatrix}\gamma_k^2 \widetilde{L}_k+\zeta_k \widetilde{R}_k&\zeta_k \gamma_k^2 \widetilde{L}_k +\widetilde{R}_k \end{pmatrix}.$$
Adding the two entries gives
$\widetilde{L}_{k+1}+\widetilde{R}_{k+1}=\gamma_k\widetilde{L}_k+\gamma_k^{-1}\widetilde{R}_k$ for $k\geq 1$ and so
\begin{equation}\label{Lemma:LRtildeMonotonicityOrder:Eq:Diff1}
 (\widetilde{L}_{k+1}+\widetilde{R}_{k+1}) - (\widetilde{L}_{k}+\widetilde{R}_{k})
= (\gamma_k+\gamma_k^{-1}-2) L_k +(1-\gamma_k^{-1}) (L_k-R_k) \ .
\end{equation}

Since $\gamma_k>1$, the first assertion
$\widetilde{L}_{k+1}+\widetilde{R}_{k+1}> \widetilde{L}_{k}+\widetilde{R}_{k}$
follows from the second assertion $\widetilde{L}_{k}>\widetilde{R}_{k}$.

It remains to prove the latter, $\widetilde{L}_{k}>\widetilde{R}_{k}$ for $k\geq 2$. The initial case $k=2$ holds:
$R_2<L_2$ because
$$2p_1\gamma_1(R_2-L_2)=(\zeta_1\gamma_1^2+1)-(\gamma_1^2+\zeta_1)=(\zeta_1-1)(\gamma_1^2-1) <0.$$
Similarly,
\begin{equation}\label{Lemma:LRtildeMonotonicityOrder:Eq:Diff2}
\begin{array}{ll}
2p_k\gamma_k(R_{k+1}-L_{k+1})&=(\zeta_k \gamma_k^2 \widetilde{L}_k +\widetilde{R}_k)-(\gamma_k^2 \widetilde{L}_k+\zeta_k \widetilde{R}_k)
\\
& =(\zeta_k-1)(\gamma_k^2 \widetilde{L}_k - \widetilde{R}_k) \ .
\end{array}
\end{equation}
The proof is complete by induction.
\end{proof}

\begin{lemma}\label{Lemma:LtildeTendInfty}
With probability one, $\widetilde{L}_{k} \to \infty $ as $k\to\infty$.
\end{lemma}

\begin{proof}
We will show that $\widetilde{L}_{k}+\widetilde{R}_{k} \to \infty$, which
implies $\widetilde{L}_{k} \to \infty$ because $\widetilde{L}_{k}>\widetilde{R}_{k}$ by Lemma
\ref{Lemma:LRtildeMonotonicityOrder}.
From \eqref{Lemma:LRtildeMonotonicityOrder:Eq:Diff1} and the fact that
$\widetilde{L}_{k}>\widetilde{R}_{k}$, we have ratio
\begin{equation}\label{Lemma:LtildeTendInfty:Eq:ratio}
\frac{\widetilde{L}_{k+1}+\widetilde{R}_{k+1}}{\widetilde{L}_{k}+\widetilde{R}_{k}} > 1+ \frac{\gamma_k+\gamma_k^{-1}-2}{2}=1+\frac{(\sqrt{\gamma_k}-\frac{1}{\sqrt{\gamma_{k}}})^2}{2} >1 \ .
\end{equation}
Let $\theta_k:=1+\frac{\left(\sqrt{\gamma_k}-\frac{1}{\sqrt{\gamma_{k}}}\right)^2}{2}$. Then
$$\widetilde{L}_{k+1}+\widetilde{R}_{k+1}>2\prod_{i=1}^k\theta_i \to\infty$$
by the ergodic theorem.
\end{proof}

\begin{remark}\rm\label{Remark:RtildeTendInfty}
Unless the tree $\mathbb{T}_{\vec{d}, \vec{\ell}}$ degenerates to $\mathbb{R}$,
$\widetilde{R}_{k}$ also tends to infinity $\mathbf{P}$-a.s., since later we will
show by Theorem \ref{Theorem:Existence_xi} that except for the case
that the tree $\mathbb{T}_{\vec{d}, \vec{\ell}}$ degenerates to $\mathbb{R}$, in which case $\xi\equiv
\lim\limits_{k\rightarrow\infty}\dfrac{\widetilde{L}_k}{\widetilde{R}_k}=\infty$, in general we always have
$\xi<\infty$ $\mathbf{P}$-a.s.
\end{remark}

The projective line $\textbf{PR}^1=\mathbb{R}^2/\sim\,$ is the set of the lines in $\mathbb{R}^2$
passing through the origin. Let $\pi:\mathbb{R}^2\to \textbf{PR}^1$ be the projection map.
We parameterize the projective line $\textbf{PR}^1$ by
$\pi(a,b)=\arctan\left(\dfrac{b}{a}\right) \in (-\pi/2,\pi/2]$ and equip
$\textbf{PR}^1$ with the metric $\rho(\theta_1,\theta_2):=|\theta_1-\theta_2|$.
Since ultimately we desire to study the ratio of $\dfrac{\widetilde{L}_k}{\widetilde{R}_k}$
as defined in \eqref{Eq:tild Lk} where $\widetilde{L}_k > \widetilde{R}_k>0$
by Lemma \eqref{Lemma:LRtildeMonotonicityOrder}, we are interested in $\theta \in\left[0,\dfrac{\pi}{4}\right)$.

Observe that if
\begin{align*}
\begin{pmatrix} A& B \end{pmatrix}
&=
\begin{pmatrix}a & b \end{pmatrix}M_k=\frac{1}{2 p_k \gamma_k}
\begin{pmatrix}\gamma_k^2 a+\zeta_k b&\zeta_k \gamma_k^2 a +b \end{pmatrix},
\end{align*}
then $\dfrac{B}{A}=\Phi_k\Big(\dfrac{b}{a}\Big)$,
where $\Phi_k=\Phi_{(\zeta_k,\gamma_k)}$ is the random M\"obius transform defined as
\begin{equation}\label{Eq:RandomMobiusTransform}
\Phi_{(\zeta,\gamma)}(z)=\frac{\zeta \gamma^2 +z}{\gamma^2 +\zeta z }.
\end{equation}

This leads us to consider the map $f_{(\zeta,\gamma)}:\,\textbf{PR}^1\to \textbf{PR}^1$ defined by
\begin{equation}\label{Eq:autof}
f_{(\zeta,\gamma)}(\theta)
=\arctan\left(\Phi_{(\zeta,\gamma)}(\tan\theta)\right)
=\arctan\left(\frac{\zeta\gamma^2+\tan\theta}{\gamma^2+\zeta\tan\theta}\right).
\end{equation}

The following lemma says that $f$ is a contraction if $(\zeta,\gamma)\in [0,1)\times (1,\infty)$.

\begin{lemma}\label{Lemma:Contraction-f}
Let $(\zeta,\gamma)\in [0,1)\times (1,\infty)$ and consider the deterministic function
$f=f_{(\zeta,\gamma)}:\,\textbf{PR}^1\rightarrow \textbf{PR}^1$ defined in \eqref{Eq:autof}. Then
\begin{equation}\label{Lemma:Contraction-f:Eq:Contraction}
\rho\Big(f(\theta_1),\,f(\theta_2)\Big)\leq K_{(\zeta,\gamma)}\,\rho(\theta_1,\,\theta_2)
\quad \text{for all }\theta_1,\,\theta_2 \in\left[0,\frac\pi4\right] \ ,
\end{equation}
where
\begin{equation}\label{Lemma:Contraction-f:Eq:K}
K_{(\zeta,\gamma)}:=\max\left\{\frac{\gamma^2(1-\zeta^2)}{(\zeta^2+1)\gamma^4},
\frac{2\gamma^2(1-\zeta^2)}{(\zeta^2+1)(\gamma^4+1)+4\zeta\gamma^2}\right\} \in (0,1).
\end{equation}
\end{lemma}

\begin{proof}
Notice that for $\theta\in\left[0,\dfrac{\pi}{2}\right)$,
\begin{align}\label{Lemma:Contraction-f:Eq:f'}
f'(\theta)=\frac{\gamma^2(1-\zeta^2)}{(\zeta^2+1)(\gamma^4\cos^2\theta+\sin^2\theta)+2\zeta\gamma^2\sin(2\theta)}>0.
\end{align}

To find an upper bound for $f'(\theta)$ we notice that
$$\frac{\partial}{\partial\theta}\left((\zeta^2+1)(\gamma^4\cos^2\theta+\sin^2\theta)+2\zeta\gamma^2\sin(2\theta)\right)=0$$
has solution at
$$\theta_n=\frac{1}{2}\arctan\left(\frac{4\zeta\gamma}{(\zeta^2+1)(\gamma^4-1)}\right)+\frac{\pi n}2 \text{ for } n\in\mathbb{Z} \ .$$
Moreover,
$\{\theta_n\}_{n\in\mathbb{Z}}\cap\left[0;\dfrac{\pi}{4}\right]=\{\theta_0\}$, because
$\frac{4\zeta\gamma}{(\zeta^2+1)(\gamma^4-1)}>0$ and so $\arctan\theta_1>\dfrac{\pi}{2}$. Now,
\begin{align*}
&\frac{\partial}{\partial\theta}\left((\zeta^2+1)(\gamma^4\cos^2\theta+\sin^2\theta)
+2\zeta\gamma^2\sin(2\theta)\right)(\theta_0)\notag\\
&=2(\zeta^2+1)(1-\gamma^4)(\cos(2\theta)-8\zeta\gamma^2\sin(2\theta))(\theta_0)\notag\\
&=2(\zeta^2+1)(1-\gamma^4)\cos\left(\arctan\left(\frac{4\zeta\gamma}
{(\zeta^2+1)(\gamma^4-1)}\right)\right)-8\zeta\gamma^2\sin\left(\arctan\left(\frac{4\zeta\gamma}
{(\zeta^2+1)(\gamma^4-1)}\right)\right)\notag\\
&=\frac{1}{\sqrt{\left(\frac{4\zeta\gamma}{(\zeta^2+1)(\gamma^4-1)}\right)^2+1}}
\left[2(\zeta^2+1)(1-\gamma^4)-8\zeta\gamma^2\left(\frac{4\zeta\gamma}{(\zeta^2+1)(\gamma^4-1)}\right)\right]\notag\\
&=\frac{1}{\sqrt{\left(\frac{4\zeta\gamma}{(\zeta^2+1)(\gamma^4-1)}\right)}}
\left[2(\zeta^2+1)(1-\gamma^4)-\frac{32\zeta^2\gamma^3}{(\zeta^2+1)(\gamma^4-1)}\right]<0.
\end{align*}
Thus, the minimum value of the denominator of \eqref{Lemma:Contraction-f:Eq:f'} on $\left[0,\pi/4\right]$
occurs at the endpoints of $\left[0,\pi/4\right]$ and so
$0<\sup_{\theta\in[0,\pi/4]} f'(\theta)\leq K_{(\zeta,\gamma)}$.
The proof of the lemma is complete by the midpoint theorem.
\end{proof}

The following random version of Lemma \ref{Lemma:Contraction-f} follows immediately.

\begin{corollary}\label{Corollary:Contraction-f}
Let $(\zeta,\gamma)\in [0,1)\times (1,\infty)$ be a random variable with
the same distribution as $(\zeta_1,\gamma_1)$ under $\mathbf{P}$.
The contraction \eqref{Lemma:Contraction-f:Eq:Contraction} holds with
random variable $K_{(\zeta,\gamma)}$
satisfying $K<1$ a.s.
In particular, $\mathbf{E} K<\infty$. and $\mathbf{E} \ln K <0$ for all $\lambda\in (0,\infty)$.
\end{corollary}

Corollary \ref{Corollary:Contraction-f} verifies the contraction assumptions
of \cite[Proposition 1.1]{DiaconisEtAlIteratedRandomFunctions} and thus gives the desired almost sure
limit $\xi$. This ensures that condition \eqref{Prop:ExplicitCalculationw-AuxiliaryDeterministic:Eq:xi}
is verified for the i.i.d. case.

\begin{theorem}[existence of the limit  $\xi$]\label{Theorem:Existence_xi}
The limit
$$\frac{1}{\xi_\lambda}:=\lim_{k\to\infty}\frac{R_k}{L_k} \in [0,1]$$
exists $\mathbf{P}$-a.s. for all $\lambda\in (0,\infty)$. The distribution of
$\dfrac{1}{\xi_\lambda}$ is the unique stationary distribution of the $\mathbb{R}_+$-valued
Markov chain $\{x_k\}_{k\geq 1}$ defined by
$x_{k}=\Phi_k\circ \Phi_{k-1}\circ \cdots \circ \Phi_1(1)$ for $k\geq 1$, where
$$\Phi_{k}(z)=\frac{\zeta_k \gamma^2_k +z}{\gamma^2_k +\zeta_k z } \ .$$
\end{theorem}

\begin{proof}
The process $\left\{\frac{R_k}{L_k}\right\}$ is a backward
(non-Markov) chain in the sense that
$\frac{R_{k+1}}{L_{k+1}}=\Phi_k\left(\frac{R_k}{L_k}\right)=\Phi_1\circ \Phi_{2}\circ \cdots \circ \Phi_{k}(1)$.
The corresponding forward iteration is the $\mathbb{R}_+$-valued
Markov chain $\{x_k\}_{k\geq 1}$ defined by $x_{k+1}=\Phi_k(x_k)$ for $k\geq 1$ and $x_1=\Phi_1(1)$.

Similarly,  processes
$\psi_{k}:=\arctan\left(\frac{R_{k}}{L_{k}}\right)$ and $\theta_{k}=\arctan\left(x_k\right)$ satisfy
\begin{align*}
&\psi_{k}=f_{(\zeta_1,\gamma_1)}\circ f_{(\zeta_1,\gamma_1)}\cdots\circ f_{(\zeta_k,\gamma_k)}(1) \ ,\\
&\theta_{k}=f_{(\zeta_k,\gamma_k)}\circ f_{(\zeta_2,\gamma_2)}\cdots\circ f_{(\zeta_1,\gamma_1)}(1) \ ,
\end{align*}
where $f_{(\zeta,\gamma)}$ is defined as in \eqref{Eq:autof}.
So in particular $\{\theta_k\}$ is also a Markov Chain.
By Lemma \ref{Lemma:Contraction-f},
$\{f_{(\zeta_k,\gamma_k)}\}$ is a family of Lipschitz functions
from the Corollary \ref{Corollary:Contraction-f} we know it is
contracting on average. Thus by \cite[Theorem 1.1]{DiaconisEtAlIteratedRandomFunctions}
the Markov chain $\{\theta_k\}$  converges to a unique stationary distribution.
In particular, the limit $\theta_k\to\theta_\infty$ exists in distribution. So
$$\lim_{k\to\infty}x_k=\lim_{k\to\infty}\tan(\theta_k)\stackrel{d}{=} \tan(\theta_\infty) \in [0,1] \ .$$
In the above display the limit is less or equal to 1 by Lemma \ref{Lemma:LRtildeMonotonicityOrder}.
Now by \cite[Proposition 1.1]{DiaconisEtAlIteratedRandomFunctions}
for backward iterations, there is an almost sure limit
$\psi_\infty$ and it has the same distribution as $\theta_\infty$. Thus,
$$
\frac{1}{\xi_\lambda}:=\lim_{k\to\infty}\frac{R_k}{L_k}=\lim_{k\to\infty}\tan(\psi_k)\,{\buildrel a.s. \over =}\,\tan(\psi_\infty)\stackrel{d}{=} \tan(\theta_\infty) \in [0,1] \ .
$$
\end{proof}

The following is an immediate corollary of Theorem \ref{Theorem:Existence_xi}.

\begin{corollary}\label{Corollary:InvarDistX}
Let $\lambda\in(0,\infty)$ be a fixed constant. Let $(\ell, \zeta)$ be a random variable that is
independent with $\xi_{\lambda}$ and is equal to $(\ell_k, \zeta_k)$ in marginal distribution,
where $k\geq 1$ is arbitrary. Then
\begin{equation}\label{Corollary:InvarDistX:Eq:InvariantX}
\frac{1}{\xi_{\lambda}}\stackrel{d}{=}   \Phi_{(\gamma,\zeta)}\left(\frac{1}{\xi_{\lambda}}\right).
\end{equation}
where $\gamma=e^{\sqrt{2\lambda}\ell}$.
Such random variable $\dfrac{1}{\xi_{\lambda}}$ is unique in distribution.
\end{corollary}

Based on Corollary \ref{Corollary:InvarDistX}, we can further derive the following corollaries about the properties of
$\xi=\xi_\lambda$.

\begin{corollary}\label{Corollary:xiBoundedAboveExpectation}
Unless $\zeta$ is identically $0$, i.e., $d_i$ is identically $2$,
or equivalently the tree degenerates to the real line $\mathbb{R}$, we have
\begin{equation}\label{Corollary:xiBoundedAboveExpectation:Eq:ExpectationOneOverXi}
\mathbf{E}\left[\dfrac{1}{\xi_{\lambda}}\right]>0 \ .
\end{equation}
When the tree degenerates to the real line $\mathbb{R}$, the above expectation is $0$.
\end{corollary}

\begin{proof}
Set $z=z_\lambda=\dfrac{1}{\xi_{\lambda}}\geq 0$. It suffices to show that
unless $\zeta$ is identically $0$, we have $\mathbf{P}(z>0)>0$.
Suppose this is not the case, then with $\mathbf{P}$-probability $1$ we have $z=0$.
By Corollary \ref{Corollary:InvarDistX},
we have $\dfrac{\zeta \gamma^2 +z}{\gamma^2 +\zeta z}\stackrel{d.}{=}z$ is $0$ with probability $1$.
But since $z\geq 0$ and $\zeta$ is not identically $0$, we arrive at a contradiction.

On the other hand, when $\zeta$ is identically $0$, i.e.,
$d_i$ is identically $2$, which means that the tree degenerates to
the real line $\mathbb{R}$, $0$ is a fixed point of the
transformation \eqref{Eq:autof} and thus $1/\xi=0$.
\end{proof}

\begin{remark}\rm\label{Remark:PropertyXi}
If $\zeta$ is not identically $0$, i.e., $d_i$ is not identically $2$,
then $\dfrac{1}{\xi_\lambda}\in (0, 1]$
with positive probability. Otherwise,
when $\zeta$ is identically $0$, i.e., $d_i$ is identically $2$, or in other words
the tree degenerates to the real line $\mathbb{R}$
and $\dfrac{1}{\xi_\lambda}=0$. Moreover, if all $d_i\geq 3$, then $\dfrac{1}{\zeta_\lambda}\in (0, 1]$
with $\mathbf{P}$-probability $1$ since in this case $\mathbf{P}(\zeta>0)=1$ and $0$ is not a fixed point of the
transformation \eqref{Eq:autof}.
\end{remark}

\begin{corollary}\label{Corollary:xiBoundedBelow1Plusc}
For all $\lambda\in(0,\infty)$,
\begin{equation}\label{Corollary:xiBoundedBelow1Plusc:Eq:xiBoundedBelow1Plusc}
\mathbf{P}\left(
\xi_{\lambda}\geq 1+ 2(\overline{d})^{-1}\dfrac{e^{2\underline{\ell}\sqrt{2\lambda}}-1}{e^{2\overline{\ell}\sqrt{2\lambda}}+1}\right)=1  \ .
\end{equation}
\end{corollary}

\begin{proof}
By Theorem \ref{Theorem:Existence_xi},  $\mathbf{P}$-a.s, we have $\xi=\xi_{\lambda}\geq 1$. Thus from
\eqref{Corollary:InvarDistX:Eq:InvariantX}, $\dfrac{1}{\xi} \stackrel{d}{=} \Phi_{\gamma, \zeta}(\Delta)$, where
$\gamma\geq e^{\underline{\ell}\sqrt{2\lambda}}>1$, $\zeta\in (0,1)$, $\Delta\in (0,1]$ and
$$\Phi_{\gamma,\zeta}(\Delta)=\dfrac{\zeta\gamma^2+\Delta}{\gamma^2+\zeta \Delta}=\dfrac{1}{1+\dfrac{(\gamma^2-\Delta)(1-\zeta)}{\zeta\gamma^2+\Delta}} \ .$$
This implies that $\xi$ has the same distribution with $1+\dfrac{(\gamma^2-\Delta)(1-\zeta)}{\zeta\gamma^2+\Delta}$. Since
$\zeta\leq 1-2(\overline{d})^{-1}$,
our choice of $\gamma$, $\zeta$ and $\Delta$ guarantee that
$$\dfrac{(\gamma^2-\Delta)(1-\zeta)}{\zeta\gamma^2+\Delta}\geq
2(\overline{d})^{-1}\dfrac{e^{2\underline{\ell}\sqrt{2\lambda}}-1}{e^{2\overline{\ell}\sqrt{2\lambda}}+1} \ ,$$
where $\overline{d}$ is the upper bound of the number of branches in Assumption \eqref{Assumption:d:Eq:UpperboundBranching}.
Thus  $\mathbf{P}$-a.s. we have
$$\xi\geq 1+ 2(\overline{d})^{-1}\dfrac{e^{2\underline{\ell}\sqrt{2\lambda}}-1}{e^{2\overline{\ell}\sqrt{2\lambda}}+1} \ .$$
\end{proof}

\section{Large deviations principle for multi-skewed Brownian
motion} \label{Sec:LDP}

Recall that $Y_t$ is the multi-skewed Brownian motion in an
i.i.d. environment $\{(\vec{p},\vec{z})\}$ under $\mathbf{P}$ and that Lemma \ref{Assumption:ErgodicEnvironments} holds.
Recall  the probability measure $P^{(\vec{p}, \vec{z})}$ that determines the
quenched law of $Y$ in a given tree $\mathbb{T}_{\vec{p}, \vec{z}}$,


As in the proof of part (2) of Lemma \ref{Lemma:ElementaryPropertiesw},
we know from (1) in Lemma \ref{Assumption:ErgodicEnvironments}
 that \eqref{Lemma:ElementaryPropertiesw:Eq:SymmstryProbabilityDistributionSkewBM} holds, i.e., for any Borel set $A\subset (0,\infty)$
 we have
$P^{(\vec{p}, \vec{z})}(Y_t\in A)=P^{(\vec{p}, \vec{z})}(Y_t \in -A)$.
This fact indicates that in our case,
we only have to consider wave-propagation in the positive direction, and the wave
speed in the negative direction should be the same as that in the positive direction.
To this end, we shall first establish Theorem \ref{Theorem:LyapunovExponentPositiveDirection} below,
which is parallel to Lemma 5.1 of \cite{FreidlinHu13Motor}.

Recall that the function $w(x)=w_\lambda(x)=E^{(\vec{p},\vec{z})}\left[e^{-\lambda T_0^x}\mathbf{1}_{T_0^x<\infty}\right]$ as
in \eqref{Eq:w-AuxiliaryFunction}. As in Section \ref{Sec:MultiSkewBMAssumptions:Subsection:HittingTimeEstimates}, we fix
a notational convention that $\eta=-\lambda$ in the rest of this section. Thus we can also write
 $w(x)=w_{-\eta}(x)=E^{(\vec{p},\vec{z})}\left[e^{\eta T_0^x}\mathbf{1}_{T_0^x<\infty}\right]$.

For any $\eta\in \mathbb{R}$ we define the \emph{Lyapunov Exponent}
\begin{equation}\label{Eq:Mu}
\mu(\eta)\equiv\dfrac{1}{\mathbf{E}\ell_0}\mathbf{E}\left(\ln
E^{(\vec{p},\vec{z})}[e^{\eta T_0^{\ell_0}}\mathbf{1}_{T^{\ell_0}_0<\infty}]\right) \ ,
\end{equation}
in which we allow for some choices of $\eta$ the quantity $\mu(\eta)$ to be $+\infty$.

Notice that by \eqref{Eq:w-AuxiliaryFunction} we have

\begin{equation}\label{Eq:MuEqliwWheneta<0}
\mu(\eta)=\dfrac{\mathbf{E}\ln w_{-\eta}(\ell_0)}{\mathbf{E}\ell_0} \ .
\end{equation}

\begin{theorem}[Lyapunov Exponent identity]\label{Theorem:LyapunovExponentPositiveDirection}
Let $\eta \in \mathbb{R}$ be such that
\begin{equation}\label{Theorem:LyapunovExponentPositiveDirection:Eq:BoundedNessExponentialMomentsAssumption}
\mathbf{E}\left(\left|\ln E^{(\vec{p},\vec{z})}\left[e^{\eta T_0^{\ell_0}}\mathbf{1}_{T^{\ell_0}_0<\infty}\right]\right|\right)<\infty  \ .
\end{equation}
Let $0<c<v$. Then almost surely the following limit holds
\begin{equation}\label{Theorem:LyapunovExponentPositiveDirection:Eq:LyapunovExponentIdentity}
\mu(\eta)=\lim_{t\to\infty}\frac{1}{(v-c)t}
\ln E^{(\vec{p},\vec{z})}\left[e^{\eta T_{ct}^{vt}}\mathbf{1}_{T^{vt}_{ct}<\infty}\right] \ .
\end{equation}
The convergence is uniform with respect to $v$ and $c$ as they vary on the subset of $(0,\infty)$ that is
bounded with $(v-c)>0$ bounded away from $0$. Furthermore, this limit is independent of $v$ and $c$.
\end{theorem}

\begin{proof}
Fix a pair $(\vec{p}, \vec{z})$. For $r,s\in \mathbb{R}$ we set
\begin{equation}\label{Theorem:LyapunovExponentPositiveDirection:Eq:q}
q(r,s,\eta)=E^{(\vec{p},\vec{z})}\left[e^{\eta T_r^s}\mathbf{1}_{T_r^s<\infty}\right] \ .
\end{equation}

By the strong Markov property of the multi-skewed Brownian motion $Y_t$, we have for $r<s<t$ that
\begin{equation}\label{Theorem:LyapunovExponentPositiveDirection:Eq:Subadditivity-lnq}
\ln q(r,t,\eta)=\ln q(r,s,\eta)+\ln q(s,t,\eta) \ .
\end{equation}

Fix $c>0$. Let the number $N(n)$ be such that $z_{N(n)}\leq cn$ and $z_{N(n)+1}>cn$, $n\in \mathbb{N}$. Then
$\lim\limits_{n\rightarrow\infty}\dfrac{cn}{N(n)}=\mathbf{E} \ell_0$ holds $\mathbf{P}$-almost surely. As we have in our Lemma \ref{Assumption:ErgodicEnvironments}, $(p_i)_{i\geq 1}$ and $(z_{i+1}-z_i\equiv \ell_i)_{i\geq 0}$ are two i.i.d
sequences of random variables, that are independent of each other. By the Law of Large Numbers
for ergodic sequences combined with \eqref{Theorem:LyapunovExponentPositiveDirection:Eq:Subadditivity-lnq} we see that
$$\lim\limits_{n\rightarrow\infty} \dfrac{\ln q(0,cn,\eta)}{N(n)}=
\mathbf{E}\left(\ln E^{(\vec{p}, \vec{z})}\left[e^{\eta T_{0}^{z_1}}\mathbf{1}_{T_{0}^{z_1}<\infty}\right]\right)$$
holds $\mathbf{P}$-almost surely, provided that we have \eqref{Theorem:LyapunovExponentPositiveDirection:Eq:BoundedNessExponentialMomentsAssumption}.
Therefore
$$\lim\limits_{n\rightarrow\infty}\dfrac{1}{cn}\ln q(0, cn, \eta)
=\lim\limits_{n\rightarrow\infty}\dfrac{1}{\dfrac{cn}{N(n)}}\dfrac{\ln q(0, cn,\eta)}{N(n)}
=\dfrac{1}{\mathbf{E}\ell_0}\mathbf{E}\left(\ln
E^{(\vec{p},\vec{z})}[e^{\eta T_0^{\ell_0}}\mathbf{1}_{T^{\ell_0}_0<\infty}]\right) \ ,$$
if \eqref{Theorem:LyapunovExponentPositiveDirection:Eq:BoundedNessExponentialMomentsAssumption} holds.
Now we can derive \eqref{Theorem:LyapunovExponentPositiveDirection:Eq:LyapunovExponentIdentity} as in
\cite[Section 2, Proposition 1]{Nolen2009}.
\end{proof}

Theorem \ref{Theorem:LyapunovExponentPositiveDirection} will lead to the
large deviations principle for both the hitting time $T_r^s$ and the process $Y_t$, and from there
we will analyze the wave-front propagation on a random $\mathbb{T}_{\vec{d}, \vec{\ell}}$ in Section \ref{Sec:WavePropagation}.
The proof here makes use of the arguments in the analysis presented in \cite{Nolen-XinCMP2007}, \cite{Nolen2009},
\cite{Taleb2001}, \cite{CometsGantertZeitouni2000}, \cite[Chapter 7]{FreidlinFunctionalBook} and \cite{FreidlinHu13},
yet there are many technical differences due to the presence of multi-skewness of the process $Y_t$ and
the symmetric structure of the tree.

The following lemma summarizes basic properties of the function $\mu(\eta)$.
To emphasize the dependence of the limit
random variable $\xi=\xi_{\lambda}$ on $\lambda=-\eta$ from Theorem \ref{Theorem:Existence_xi}, we will explicit
this dependence $\xi=\xi_{\lambda}=\xi_{-\eta}$ throughout.

Define
\begin{equation}\label{Def:CriticalEtaForMu}
\eta_c:=\sup\{\eta\in \mathbb{R}:\,\mu(\eta)<\infty\} \ .
\end{equation}

In the below, for a function $f(\eta)$ that depends on $\eta$, we denote $f(\eta_c-)$ to be the limit
$\lim\limits_{\eta \rightarrow \eta_c-}f(\eta)$. The function $f(\eta)$ can be $\mu(\eta)$ or $\mu'(\eta)$.

\begin{lemma}\label{Lemma:PropertiesMu}
The following properties of the function $\mu(\eta)$ hold.
\begin{itemize}
\item[(1)] $\mu(0) \leq 0$;

\item[(2)] When $\eta<0$ we have
\begin{equation}\label{Lemma:PropertiesMu:Eq:ExplicitCalculation-mu-eta}
\mu(\eta)=-\sqrt{-2\eta}+\dfrac{1}{\mathbf{E}\ell_0}\mathbf{E}\left(\ln \dfrac{\xi_{-\eta}-1}{\xi_{-\eta}-e^{-2\sqrt{-2\eta}\ell_0}}\right) \ .
\end{equation}
In particular, $\mu(\eta)<0$ for $\eta<0$;

\item[(3)] $\mu(\eta)\rightarrow -\infty$ as $\eta \rightarrow -\infty$;

\item[(4)] We have $\eta_c\in [0,\infty)$, so that $\mu(\eta)<\infty$ when $\eta< \eta_c$ and $\mu(\eta)=+\infty$ when $\eta>\eta_c$;

\item[(5)] When $\eta\leq \eta_c$, $\mu(\eta)$ is a convex function of $\eta$ and $\mu'(\eta)$ is monotonically
strictly increasing in $\eta$;

\item[(6)] For $\eta<\eta_c$, the function $\mu(\eta)$ is continuously differentiable with $\mu'(\eta)>0$. In particular,
$0<\mu'(0)\leq \mu'(\eta_c-)\in (0,+\infty]$ with the equality being satisfied when $\eta_c=0$.

\item[(7)] We have $\eta_c \mu'(0)+\mu(0)<\mu(\eta_c-)$ (if $\mu(\eta_c-)=+\infty$
this is saying that $\eta_c \mu'(0)+\mu(0)<+\infty$). In particular, if
$\mu(\eta_c-)\leq 0$, then $\dfrac{-\mu(0)}{\mu'(0)}> \eta_c$;
\end{itemize}
\end{lemma}

\begin{proof}

\begin{itemize}

\item[(1)] By \eqref{Eq:Mu} and part (3) of Lemma \ref{Lemma:ElementaryPropertiesw}, setting $\eta=0$ we get
\begin{equation}\label{Lemma:PropertiesMu:Eq:Mu-0}
\mu(0)=\dfrac{1}{\mathbf{E}\ell_0}\mathbf{E}\left[\ln P^{(\vec{p},\vec{z})}(T^{\ell_0}_0<\infty)\right] \ .
\end{equation}
Since $P^{(\vec{p},\vec{z})}(T^{\ell_0}_0<\infty)\in (0, 1]$, we get $\mu(0)\leq 0$.

\item[(2)] 
By \eqref{Prop:ExplicitCalculationw-AuxiliaryDeterministic:Eq:wExplicit} in Proposition \ref{Prop:ExplicitCalculationw-AuxiliaryDeterministic}
we have $w(\ell_0)=w(z_1)=f_1^++f_1^-$. By \eqref{Prop:ExplicitCalculationw-AuxiliaryDeterministic:Eq:VecfPositive},
\begin{equation}\label{Lemma:ExplicitCalculation-mu:Eq:w-ell1-step1}
\begin{array}{ll}
w(\ell_0) & =f_1^++f_1^-
\\
& =-\dfrac{1}{e^{\sqrt{2\lambda}\ell_0}\xi_{\lambda}-e^{-\sqrt{2\lambda}\ell_0}}+
\dfrac{\xi_{\lambda}}{e^{\sqrt{2\lambda}\ell_0}\xi_{\lambda}-e^{-\sqrt{2\lambda}\ell_0}}
\\
& =\dfrac{\xi_{\lambda}-1}{e^{\sqrt{2\lambda}\ell_0}\xi_{\lambda}-e^{-\sqrt{2\lambda}\ell_0}}
\\
& = e^{-\sqrt{2\lambda}\ell_0}\cdot \dfrac{\xi_{\lambda}-1}{\xi_{\lambda}-e^{-2\sqrt{2\lambda}\ell_0}} \ .
\end{array}
\end{equation}

Thus we have
$$\ln w(\ell_0)=-\sqrt{2\lambda}\ell_0+\ln \dfrac{\xi_{\lambda}-1}{\xi_{\lambda}-e^{-2\sqrt{2\lambda}\ell_0}} \ . $$
By \eqref{Theorem:LyapunovExponentPositiveDirection:Eq:LyapunovExponentIdentity} and the convention
that $\eta=-\lambda$,
$$\mu(\eta)=\dfrac{\mathbf{E}\ln w(\ell_0)}{\mathbf{E} \ell_0}=-\sqrt{-2\eta}+\dfrac{1}{\mathbf{E}\ell_0}
\left(\mathbf{E}\ln\dfrac{\xi_{-\eta}-1}{\xi_{-\eta}-e^{-2\sqrt{-2\eta}\ell_0}}\right) \ ,$$
which is \eqref{Lemma:PropertiesMu:Eq:ExplicitCalculation-mu-eta}.

For $\eta<0$ we have $\xi_{-\eta}-1<\xi_{-\eta}-e^{-2\sqrt{-2\eta}\ell_0}$
so that $\ln\dfrac{\xi_{-\eta}-1}{\xi_{-\eta}-e^{-2\sqrt{-2\eta}\ell_0}}<0$, which ensures that $\mu(\eta)<0$ since we have the
analytic formula \eqref{Lemma:PropertiesMu:Eq:ExplicitCalculation-mu-eta} for $\mu(\eta)$ when $\eta<0$.

\item[(3)] Since $\ln\dfrac{\xi_{-\eta}-1}{\xi_{-\eta}-e^{-2\sqrt{-2\eta}\ell_0}}<0$,
as $\eta\rightarrow -\infty$ we know that $\mu(\eta)\rightarrow -\infty$ from the
analytic formula \eqref{Lemma:PropertiesMu:Eq:ExplicitCalculation-mu-eta} for $\mu(\eta)$ when $\eta<0$.

\item[(4)] Let $\eta_c^w$ be defined by \eqref{Eq:Def:etacY-w} and
$\eta_c^*:=\text{ess}\inf\limits_{(\vec{p},\vec{z})\in\widetilde{\Omega}\subset \Omega, \mathbf{P}(\widetilde{\Omega})=1}\eta_c^w$
be the essential infimum of $\eta_c^w$ under $\mathbf{P}$.
 Then $\eta_c^* \in [0,\infty)$. Assumption \ref{Assumption:ell} ensures that $\mathbf{E}\ell_0 \in (0,\infty)$.
By the proof of Theorem \ref{Theorem:etac_Y-w}, the random variable $\eta_c^w$ given by \eqref{Eq:Def:etacY-w}
takes value in $\left[0,\,\widetilde{B}\right) \subset [0,\infty)$ almost
surely under $\mathbf{P}$, where $\widetilde{B}$ is given by \eqref{Theorem:etac_Y-w:Proof:Eq:minpqY}.
Furthermore, since both $\ell_i$ and $p_i$ are bounded above and below according to
Lemma \ref{Assumption:ErgodicEnvironments}, there exists some constant $B^*\in (0,\infty)$
such that $\mathbf{P}(\widetilde{B}\leq B^*)=1$. Thus $\eta_c^*\in [0, \infty)$.
If $\eta> \eta_c^*$, then $\mathbf{P}(\eta>\eta_c)>0$ which implies $\mu(\eta)=+\infty$. Hence
 $0\leq \eta_c\leq \eta_c^*<\infty$.

\item[(5)] By H\"{o}lder's inequality, for any $\eta_1, \eta_2\leq \eta_c$ we have
$$E^{(\vec{p}, \vec{z})}\left[e^{\frac{1}{2}(\eta_1+\eta_2)}\mathbf{1}_{T^{\ell_0}_0<\infty}\right]
\leq E^{(\vec{p}, \vec{z})}\left[e^{\frac{1}{2}\eta_1}\mathbf{1}_{T^{\ell_0}_0<\infty}\right]^{1/2}
\cdot E^{(\vec{p}, \vec{z})}\left[e^{\frac{1}{2}\eta_2}\mathbf{1}_{T^{\ell_0}_0<\infty}\right]^{1/2} \ .$$
This implies that when $\eta\leq \eta_c$, $\mu(\eta)$ is a convex function of $\eta$. Due to the condition at which
H\"{o}lder's inequality is satisfied, as long as $\eta_1\neq  \eta_2$ the above inequality is a strict inequality.
This further implies that $\mu(\eta)$ is strictly convex, i.e., $\mu'(\eta)$ is monotonically
strictly increasing in $\eta$;

\item[(6)] By the same argument in the proof of part (vi) in \cite[Lemma 2.2]{Nolen2009} we can show that
\begin{equation}\label{Lemma:PropertiesMu:Eq:MuPrimeEta}
\mu'(\eta)=\mathbf{E}\left[\dfrac{E^{(\vec{p}, \vec{z})}[T^{\ell_0}_0 e^{\eta T^{\ell_0}_0}\mathbf{1}_{T^{\ell_0}_0<\infty}]}
{E^{(\vec{p}, \vec{z})}[e^{\eta T^{\ell_0}_0}\mathbf{1}_{T^{\ell_0}_0<\infty}]}\right]>0 \ .
\end{equation}

Assume that we have a sequence $\eta_n\rightarrow \eta<\eta_c$ as $n \rightarrow\infty$.
Then there exist a constant $C=C(\eta_c, \eta)$ that may depend on $\eta_c$ and $\eta$ such that
$$E^{(\vec{p}, \vec{z})}[T^{\ell_0}_0 e^{\eta_n T^{\ell_0}_0}\mathbf{1}_{T^{\ell_0}_0<\infty}]\leq
CE^{(\vec{p}, \vec{z})}[e^{\eta T^{\ell_0}_0}\mathbf{1}_{T^{\ell_0}_0<\infty}]  \ .$$
We also have by Lemma \ref{Lemma:EstimateHittingProbabilityBackZeroMulti-SkewedBM} that there
exists another $\widehat{C}>0$ such that
$$E^{(\vec{p}, \vec{z})}[e^{\eta_n T^{\ell_0}_0}\mathbf{1}_{T^{\ell_0}_0<\infty}]\geq
P^{(\vec{p}, \vec{z})}[T^{\ell_0}_0<\infty]\geq \widehat{C} \ .$$

So we have
$$\dfrac{E^{(\vec{p}, \vec{z})}[T^{\ell_0}_0 e^{\eta_n T^{\ell_0}_0}\mathbf{1}_{T^{\ell_0}_0<\infty}]}
{E^{(\vec{p}, \vec{z})}[e^{\eta_n T^{\ell_0}_0}\mathbf{1}_{T^{\ell_0}_0<\infty}]}\leq
\dfrac{C}{\widehat{C}}E^{(\vec{p}, \vec{z})}[e^{\eta T^{\ell_0}_0}\mathbf{1}_{T^{\ell_0}_0<\infty}] \ ,$$
and we can thus apply the dominated convergence theorem to conclude that $\mu'(\eta_n)\rightarrow \mu'(\eta)$
as $\eta_n\rightarrow\eta$, i.e., $\mu'(\eta)$ is continuous in $\eta$ for $\eta<\eta_c$.

Moreover, \eqref{Lemma:PropertiesMu:Eq:MuPrimeEta} gives
\begin{equation}\label{Lemma:PropertiesMu:Eq:MuPrimeZero}
\mu'(0)=\mathbf{E}\left[E^{(\vec{p},\vec{z})}[T_0^{\ell_0}\mathbf{1}_{T_0^{\ell_0}<\infty}]\right]\in (0, +\infty] \ .
\end{equation}
Since $\mu'(\eta)$ is monotonically increasing in $\eta$ as long as $\eta\leq \eta_c$ due to part (1), we get further
\begin{equation}\label{Lemma:PropertiesMu:Eq:MuPrimeEtac}
0<\mu'(0)\leq \mu'(\eta_c-) \in (0, +\infty]
\end{equation}
with the equality being satisfied as long as $\eta_c=0$;

\item[(7)] Since $\mu(\eta)$ is a strictly convex function in $\eta$ as long as $\eta\leq \eta_c$, we obtain by the property of convexity that
$\mu(0)+\eta_c\mu'(0)<\mu(\eta_c)$. If $\mu(\eta_c)\leq 0$, we further obtain that $\mu(0)+\eta_c\mu'(0)<0$, which is $\dfrac{-\mu(0)}{\mu'(0)}>\eta_c$.
\end{itemize}
\end{proof}

\begin{remark}\rm \label{Remark:ConditionMu-0LessZero}
Consider the ratio $\rho_i:=\frac{p^i_{-1}}{p^i_{+1}}$ where
 $p^i_{+1}$ and $p^i_{-1}$ are defined in \eqref{Eq:p^s2}. 
If we assume $\mathbf{E}\left[\rho_i\right]<1$, then by part (a) of
Lemma \ref{Lemma:EstimateHittingProbabilityBackZeroMulti-SkewedBM} we have $\mu(0)<0$. Moreover, the condition
$\mathbf{E}\left[\rho_i\right]<1$ implies that
$\mathbf{E}\left[\ln \rho_i \right]<0$ by Jensen's inequality, which yields the law of strong large numbers $\lim\limits_{t\to\infty}\frac{Y_t}{t}>0$ under $\mathbf{P}$ by \cite[Theorem 1.16]{SolomonRWREAnnProb}. 

By \eqref{Eq:p^s2},  $\mathbf{E}\left[\rho_i\right]<1$ is the same as saying
$\mathbf{E}\left[\dfrac{\ell_i}{\ell_{i-1}(d_i-1)}\right]<1$. The latter is satisfied, in particular, if $\ell_i=\ell$ is a constant
and $d_i\geq 3$ for all $i$.
\end{remark}

\begin{remark}\rm \label{Remark:Etac-MuPrimeEtacRealLineCase}
When the tree $\mathbb{T}_{\vec{d}, \vec{\ell}}$ degenerates to the real line $\mathbb{R}$, it was proved in Lemma 2.2 and Proposition 2 in \cite{Nolen2009} that $\eta_c=0$, $\mu(0)=0$ and $\mu'(0)=\infty$.
\end{remark}

A key quantity in large deviations theory is the Legendre transform of the Lyapunov function $\mu(\eta)$. Due to
property (6) of Lemma \ref{Lemma:PropertiesMu}, we can define the Legendre transform of $\mu(\eta)$ as a new function $I(a)$:
\begin{equation}\label{Eq:Ia}
I(a)=\sup\limits_{\eta\leq \eta_c}(a\eta-\mu(\eta)) \ .
\end{equation}

The following lemma summarizes properties of the function $I(a)$.
\begin{lemma}\label{Lemma:PropertyIa}
The following properties of the function $I(a)$ hold.
\begin{itemize}
\item[(1)] $I(a)$ is convex in $a$ and $I(a)\geq 0$ for $a\in (0,\infty)$;

\item[(2)] $I(a)$ is decreasing in $a$ for $a\in (0,\mu'(0)]$ and is increasing in $a$ for $a\in (\mu'(0), \infty)$, with $I(\mu'(0))=-\mu(0)$
to be the minimum point of $I(a)$ as $a\in (0,\infty)$;

\item[(3)] $\lim\limits_{a\rightarrow 0+}I(a)=+\infty$;

\item[(4)] $I(a)$ is piecewisely differentiable on both intervals $a\in (0, \mu'(\eta_c-))$ and $a\in [\mu'(\eta_c-), \infty)$ and
$I'(a)\leq \eta_c$ for all $a\in (0, \infty)$;

\item[(5)] $I(a)\geq a\eta_c-\mu(\eta_c-)$ and $I(a)> a\eta_c-\mu(\eta_c-)$ when $a\in (0, \mu'(\eta_c-))$;

\item[(6)] If $\mu'(\eta_c-)<\infty$, then $I(a)=a\eta_c-\mu(\eta_c-)$ for $a\in [\mu'(\eta_c-), \infty)$;

\item[(7)] If $\mu'(\eta_c-)=\infty$, then $I(a)>a\eta_c-\mu(\eta_c-)$ for $a\in (0, \infty)$ and
$I(a)-[a\eta_c-\mu(\eta_c-)]$ decreases to $0$ as $a\rightarrow \infty$;

\item[(8)] If $\mu(\eta_c-)=\infty$, then $I'(a)<\eta_c$ for all $a\in (0, \infty)$ and $I'(a)\rightarrow \eta_c$ as
$a\rightarrow \infty$.
\end{itemize}
\end{lemma}

\begin{proof}

According to parts (1) and (6) of Lemma \ref{Lemma:PropertiesMu}, for $\eta<\eta_c$ the function
$\mu(\eta)$ is differentiable and $\mu'(\eta)$ is continuous and
monotonically strictly increasing. Thus for any $a\in (0, \mu'(\eta_c-))$ (including possibly the case $\mu'(\eta_c-)=\infty$),
there is a unique point denoted as $\eta(a)\in (-\infty, \eta_c)$
such that

\begin{equation}\label{Lemma:PropertyIa:Eq:SecantEquations-alessMuPrimeEtac-Equality}
\mu'(\eta(a))=a \ ,
\end{equation}
and for any $-\infty<\eta_1<\eta(a)<\eta_2<\eta_c$ we have
\begin{equation}\label{Lemma:PropertyIa:Eq:SecantEquations-alessMuPrimeEtac-Inequality}
\mu'(\eta_1)<a<\mu'(\eta_2) \ .
\end{equation}

Let us consider the family of functions $\iota(\eta; a)=a\eta-\mu(\eta)$ parameterized by $a\in (0, \infty)$,
so that $I(a)=\sup\limits_{\eta\leq \eta_c}\iota(\eta; a)$. The function $\iota(\eta; a)$ is differentiable in $\eta$, such that
$\dfrac{d}{d\eta}\iota(\eta; a)=a-\mu'(\eta)$. This combined with
\eqref{Lemma:PropertyIa:Eq:SecantEquations-alessMuPrimeEtac-Equality} and
\eqref{Lemma:PropertyIa:Eq:SecantEquations-alessMuPrimeEtac-Inequality} imply that
for a fixed $a\in (0, \mu'(\eta_c-))$ we have
$\dfrac{d}{d\eta}\iota(\eta; a)>0$ for $\eta\in (-\infty, \eta(a))$,
$\dfrac{d}{d\eta}\iota(\eta; a)=0$ for $\eta=\eta(a)$,
and $\dfrac{d}{d\eta}\iota(\eta; a)<0$ for $\eta\in (\eta(a), \eta_c)$.
Notice that $\iota(\eta; a)=-\infty$ when $\eta>\eta_c$, we see that
for any $a\in (0, \mu'(\eta_c-))$ we have $I(a)=a\eta(a)-\mu(\eta(a))$.

Suppose $\mu'(\eta_c-)<\infty$ and $\mu(\eta_c-)<\infty$
\footnote{Notice that if $\mu(\eta_c-)=\infty$, then $\mu'(\eta_c-)=\infty$.}, then due to \eqref{Lemma:PropertyIa:Eq:SecantEquations-alessMuPrimeEtac-Inequality}
and part (5) of Lemma \ref{Lemma:PropertiesMu}, we see that
 for a fixed $a\in [\mu'(\eta_c-), \infty)$ we have
$\dfrac{d}{d\eta}\iota(\eta; a)=a-\mu'(\eta)\geq 0$ for all $\eta\in (-\infty, \eta_c)$ and
$\iota(\eta; a)=-\infty$ when $\eta>\eta_c$. Thus in this case, when $a\in [\mu'(\eta_c-), \infty)$ we have
$I(a)=a\eta_c-\mu(\eta_c-)$.

In summary we have

\begin{equation}\label{Lemma:PropertyIa:Eq:I-a-Explicit-eta-a}
I(a)=\left\{
\begin{array}{ll}
a\eta(a)-\mu(\eta(a)) & \text{ for } a\in (0, \mu'(\eta_c-))  \ ,
\\
a\eta_c-\mu(\eta_c-) & \text{ for } a\in [\mu'(\eta_c-), \infty) \ .
\end{array}\right.
\end{equation}

From \eqref{Lemma:PropertyIa:Eq:SecantEquations-alessMuPrimeEtac-Equality},
$\eta(a)=[\mu']^{-1}(a)$ is also a continuous and increasing function of $a\in (0, \mu'(\eta_c-))$,
and thus it is almost everywhere differentiable on $a\in (0, \mu'(\eta_c-))$. This combined with
\eqref{Lemma:PropertyIa:Eq:I-a-Explicit-eta-a} and \eqref{Lemma:PropertyIa:Eq:SecantEquations-alessMuPrimeEtac-Equality}
tell us that for almost everywhere $a\in (0, \mu'(\eta_c-))$
we have $I'(a)=\eta(a)+a\eta'(a)-\mu'(\eta(a))\eta'(a)
\stackrel{\text{use } \eqref{Lemma:PropertyIa:Eq:SecantEquations-alessMuPrimeEtac-Equality}}{=}\eta(a)+a\eta'(a)-a\eta'(a)=\eta(a)$.
Since $\eta(a)$ is continuous, this further ensures that for all $a\in (0, \mu'(\eta_c-))$
we have
\begin{equation}\label{Lemma:PropertyIa:Eq:DerivativeOfIWithRespectToa}
I'(a)=\eta(a) \ .
\end{equation}

\begin{itemize}
\item[(1)] Since when $\eta>\eta_c$ we have $a\eta-\mu(\eta)=-\infty$ due to the definition of
$\eta_c$ in \eqref{Def:CriticalEtaForMu}, we see that indeed
$I(a)=\sup\limits_{\eta\in \mathbb{R}}(a\eta-\mu(\eta))$ is the Legendre transform of the convex function $\mu(\eta)$.
This concludes the convexity of $I(a)$.
Since $\mu(0)\leq 0$ due to part (2) of Lemma \ref{Lemma:PropertiesMu}, we have
$I(a)=\sup\limits_{\eta\leq\eta_c}(a\eta-\mu(\eta))\geq a\cdot 0-\mu(0)\geq 0$ for all $a\in (0, \infty)$.

\item[(2)] Since $\mu'(0)\leq \mu'(\eta_c-)$, due to \eqref{Lemma:PropertyIa:Eq:SecantEquations-alessMuPrimeEtac-Equality} we have $\eta(\mu'(0))=0$.
This and part (5) of Lemma
\ref{Lemma:PropertiesMu} the monotonically strict
increasing property of $\mu'(\eta)$ imply that when $a\leq \mu'(0)$
we have $I'(a)=\eta(a)\leq 0$ and when $a>\mu'(0)$ we have $I'(a)=\eta(a)>0$. This implies that
$I(a)$ is decreasing in $a$ for $a\in (0, \mu'(0)]$ and increasing in $a$ for $a\in (\mu'(0), \infty)$,
and the minimum of $I(a)$ is achieved at $a=\mu'(0)$ such that $I(\mu'(0))
=\mu'(0)\eta(\mu'(0))-\mu(\eta(\mu'(0)))=-\mu(0)$.

\item[(3)] Since $\eta(a)\rightarrow -\infty$ as $a\rightarrow 0$, this property follows from part (3)
of Lemma \ref{Lemma:PropertiesMu}.

\item[(4)] This follows from \eqref{Lemma:PropertyIa:Eq:I-a-Explicit-eta-a} and the fact that $\eta(a)< \eta_c$
when $a\in (0, \mu'(\eta_c-))$.

\item[(5)] Since $I'(a)-\eta_c=\eta(a)-\eta_c<0$ when $a\in (0, \mu'(\eta_c-))$, the function
$\mathcal{I}(a)\equiv I(a)-a\eta_c$ is monotonically decreasing in $a$
and we have $\mathcal{I}(0)>-\mu(\eta_c-)$. Moreover, due to \eqref{Lemma:PropertyIa:Eq:I-a-Explicit-eta-a},
as $a\rightarrow \mu'(\eta_c-)$ we have $I(a)\rightarrow a\eta_c-\mu(\eta_c-)$, so that $\mathcal{I}(0)\rightarrow -\mu(\eta_c-)$; and
when $a\geq \mu'(\eta_c-)$ we have $I(a)=a\eta_c-\mu(\eta_c-)$, so that $\mathcal{I}(0)=-\mu(\eta_c-)$. These imply the statement.

\item[(6)] This follows directly from \eqref{Lemma:PropertyIa:Eq:I-a-Explicit-eta-a}.

\item[(7)] This follows from the proof of part (6) of this Lemma and the fact that $\eta(a)\rightarrow \eta_c$ when $a\rightarrow \infty$ in the case $\mu'(\eta_c-)=\infty$.

\item[(8)] Notice that when $\mu(\eta_c-)=\infty$ we must have $\mu'(\eta_c-)=\infty$, and thus due to
\eqref{Lemma:PropertyIa:Eq:I-a-Explicit-eta-a} and \eqref{Lemma:PropertyIa:Eq:DerivativeOfIWithRespectToa}
and the fact that in this case $\eta(a)<\eta_c$ for all $a\in (0,\infty)$, we have $I'(a)<\eta_c$ for all $a\in (0, \infty)$.
Since $\eta(a)\rightarrow \eta_c$ when $a\rightarrow\infty$, we have $I'(a)\rightarrow \eta_c$ when $a\rightarrow\infty$.
\end{itemize}
\end{proof}

We have demonstrated in Figure \ref{Fig:I-a-Graph} the shape of the function $I(a)$ that exhausts $8$ different cases:
\begin{itemize}
\item[(a-1)] $\mu(\eta_c-)\geq 0, 0<\mu'(\eta_c-)<\infty, \dfrac{-\mu(0)}{\mu'(0)}\geq \eta_c$;
\item[(a-2)] $\mu(\eta_c-)\geq 0, \mu'(\eta_c-)=\infty, \dfrac{-\mu(0)}{\mu'(0)}\geq \eta_c$;
\item[(b-1)] $\mu(\eta_c-)\geq 0, 0<\mu'(\eta_c-)<\infty, \dfrac{-\mu(0)}{\mu'(0)}<\eta_c$;
\item[(b-2)] $\mu(\eta_c-)\geq 0, \mu'(\eta_c-)=\infty, \dfrac{-\mu(0)}{\mu'(0)}<\eta_c$;
\item[(c-1)] $\mu(\eta_c-)<0, 0<\mu'(\eta_c-)<\infty$;
\item[(c-2)] $\mu(\eta_c-)<0, \mu'(\eta_c-)=\infty$;
\item[(d-1)] $\mu(\eta_c-)=\mu'(\eta_c-)=\infty, \dfrac{-\mu(0)}{\mu'(0)}\geq \eta_c$;
\item[(d-2)] $\mu(\eta_c-)=\mu'(\eta_c-)=\infty, \dfrac{-\mu(0)}{\mu'(0)}<\eta_c$.
\end{itemize}
Notice that due to part (2) of Lemma \ref{Lemma:PropertiesMu},
the condition $\dfrac{-\mu(0)}{\mu'(0)}\geq \eta_c$ or $\dfrac{-\mu(0)}{\mu'(0)}< \eta_c$ determines
whether the point $(\mu'(0), -\mu(0))$ (which is the minimum point of $I(a)$) is above or below (or, to the left or right of) the line given by $a\mapsto a\eta_c$.
This issue together with the condition about the slope of $\beta$ in Figure \ref{Fig:I-a-Graph} will be discussed
in Section \ref{Sec:WavePropagation}, Remarks \ref{Remark:PurposesOfConditions-BetaGreaterEtac-BetaGreaterMuStuff}, \ref{Remark:ConditionBetaGreaterEtacImplyMuStuff}.

\begin{figure}
\centering
\includegraphics[height=21cm, width=17cm]{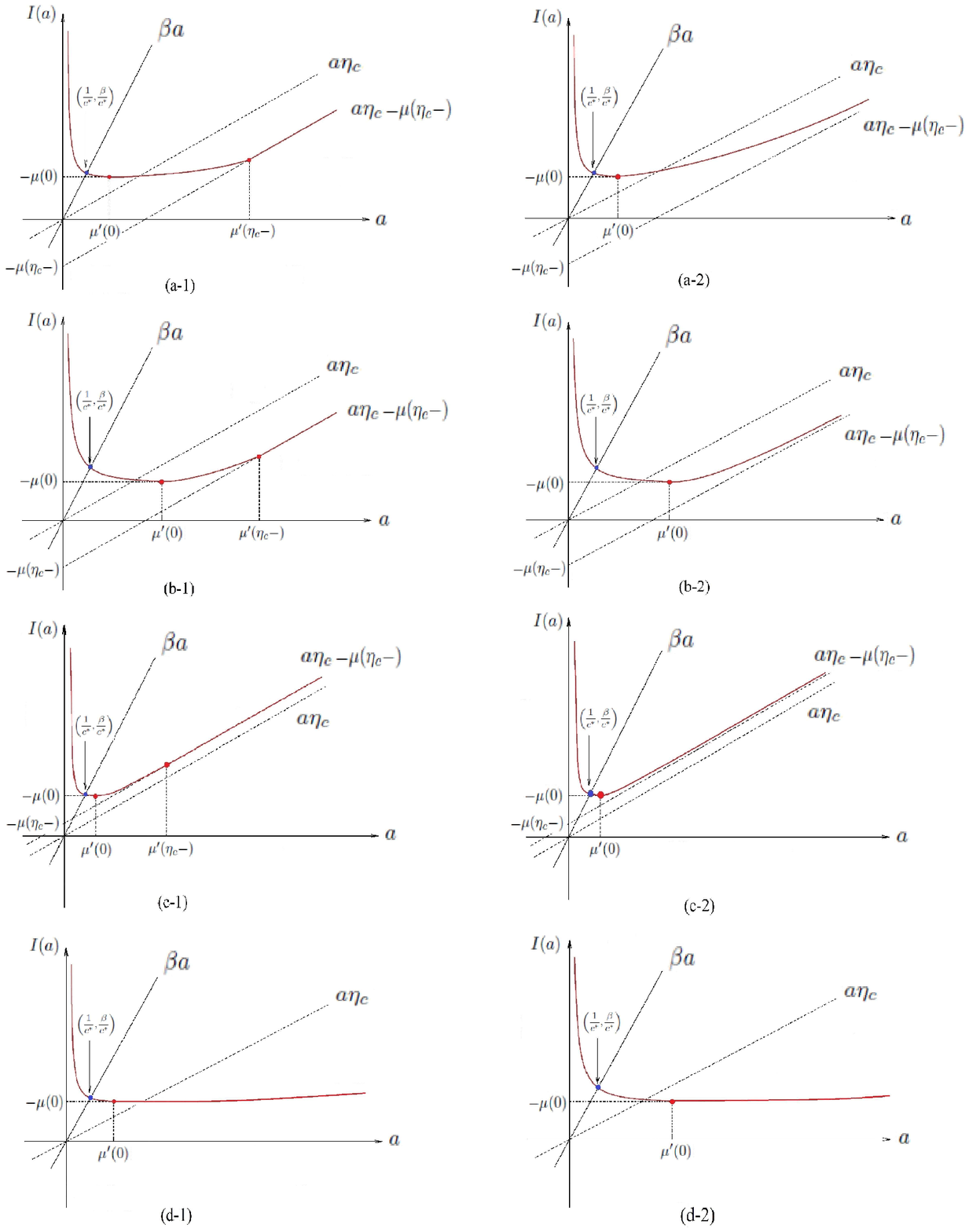}
\caption{Graph of the function $I(a)$ that exhausts all $8$ different cases. The constant-$(d,\ell)$ tree falls into Case (c-2).}
\label{Fig:I-a-Graph}
\end{figure}

The following Theorem gives the large deviations principle for the hitting time.

\begin{theorem}[Large deviations principle for the hitting time]\label{Theorem:LDPHittingTime}
Let $0<c<v$. Then $\mathbf{P}$-almost surely  the following two estimates hold. For
any closed set $G\subset (0,(v-c)\mu'(0))$ we have
\begin{equation}\label{Theorem:LDPHittingTime:Eq:UpperBound}
\lim\limits_{t\rightarrow\infty}\dfrac{1}{t}\ln P^{(\vec{p}, \vec{z})}
\left(\dfrac{T_{ct}^{vt}}{t}\in G\right)\leq -(v-c)\inf\limits_{a\in G}I\left(\dfrac{a}{v-c}\right) \ ;
\end{equation}
and for any open set $F\in (0,(v-c)\mu'(0))$ we have
\begin{equation}\label{Theorem:LDPHittingTime:Eq:LowerBound}
\lim\limits_{t\rightarrow\infty}\dfrac{1}{t}\ln P^{(\vec{p}, \vec{z})}\left(\dfrac{T_{ct}^{vt}}{t}\in F\right)
\geq -(v-c)\inf\limits_{a\in F}I\left(\dfrac{a}{v-c}\right) \ .
\end{equation}
\end{theorem}

\begin{proof}
The proof follows similar ideas as those appeared in the proof of Theorem 2.3 in
\cite{Nolen2009} (see also Theorem 5.1 in \cite{FreidlinHu13}).
We consider the upper bound first. By Chebyshev inequality, for any $\alpha>0$ and
any $\eta \leq 0$:
$$\begin{array}{ll}
\limsup\limits_{t\rightarrow\infty}\dfrac{1}{t}\ln P^{(\vec{p}, \vec{z})}\left(\dfrac{T_{ct}^{vt}}{t}<\alpha\right)
& \leq \limsup\limits_{t\rightarrow\infty}\dfrac{1}{t}\ln P^{(\vec{p}, \vec{z})}\left(e^{\eta T_{ct}^{vt}}>e^{\eta \alpha t}\right)
\\
& \leq -\eta \alpha+\limsup\limits_{t\rightarrow\infty}\dfrac{1}{t}\ln q(ct, vt, \eta)
\\
& = -\eta \alpha+(v-c)\mu(\eta) \ ,
\end{array}$$
where in the last identity we have used \eqref{Theorem:LyapunovExponentPositiveDirection:Eq:LyapunovExponentIdentity}.
The above estimate works for any $\eta\leq 0$, so

$$\begin{array}{ll}
\limsup\limits_{t\rightarrow\infty}\dfrac{1}{t}\ln P^{(\vec{p}, \vec{z})}\left(\dfrac{T^{vt}_{ct}}{t}<\alpha\right)
& \leq \inf\limits_{\eta\leq 0}(-\eta \alpha + (v-c)\mu(\eta))
\\
& =-\sup\limits_{\eta\leq 0}(\eta\alpha-(v-c)\mu(\eta))
\\
& =-(v-c)\sup\limits_{\eta\leq 0}\left(\eta \dfrac{\alpha}{v-c}-\mu(\eta)\right) \ .
\end{array}$$
If $\dfrac{\alpha}{v-c}\leq \mu'(0)$, then \eqref{Lemma:PropertyIa:Eq:I-a-Explicit-eta-a} and the fact
 $\eta(a)\leq 0$ for $a\leq \mu'(0)$ due to \eqref{Lemma:PropertyIa:Eq:SecantEquations-alessMuPrimeEtac-Equality}
 of Lemma \ref{Lemma:PropertyIa} imply that $\sup\limits_{\eta\leq \eta_c}\left(\eta \dfrac{\alpha}{v-c}-\mu(\eta)\right)$ is achieved at a point $\eta\leq 0$. So by taking into account
the definition of $I(a)$ in \eqref{Eq:Ia}, we indeed have that in this case, $I\left(\dfrac{\alpha}{v-c}\right)=\sup\limits_{\eta\leq 0}\left(\eta \dfrac{\alpha}{v-c}-\mu(\eta)\right)$. Thus the above estimate enables us to obtain

$$\limsup\limits_{t\rightarrow\infty}\dfrac{1}{t}\ln P^{(\vec{p}, \vec{z})}\left(\dfrac{T^{vt}_{ct}}{t}<\alpha\right)
\leq -(v-c)I\left(\dfrac{\alpha}{v-c}\right) \ .$$

Since by part (2) of Lemma \ref{Lemma:PropertyIa}, the function $I\left(\dfrac{\alpha}{v-c}\right)$ is monotonically
decreasing in $\alpha$ when $\dfrac{\alpha}{v-c}\leq \mu'(0)$, we see that the above estimate imply \eqref{Theorem:LDPHittingTime:Eq:UpperBound}.

We then derive the lower bound. Set $u\in (0,(v-c)\mu'(0))$ and $\delta>0$. Let $B_{\delta}(u)=(u-\delta, u+\delta)$ be the $\delta$-ball centered
at $u$. Since $\dfrac{u}{v-c}\leq \mu'(0)$, \eqref{Lemma:PropertyIa:Eq:I-a-Explicit-eta-a} and the fact
 $\eta(a)\leq 0$ for $a\leq \mu'(0)$ due to \eqref{Lemma:PropertyIa:Eq:SecantEquations-alessMuPrimeEtac-Equality}
 of Lemma \ref{Lemma:PropertyIa} imply that there is a $\eta_u\leq 0$ such that

$$I\left(\dfrac{u}{v-c}\right)
=\eta_u\dfrac{u}{v-c}-\mu(\eta_u) \ .$$

We now make use of the following Cram\'{e}r's change of measure (see \cite{CometsGantertZeitouni2000}, \cite{Taleb2001}). Let

$$\dfrac{dP^{(\vec{p}, \vec{z}), u, t}}{dP^{(\vec{p}, \vec{z})}}=\dfrac{1}{S_{u,t}}e^{\eta_uT^{vt}_{ct}}\mathbf{1}_{T^{vt}_{ct}<\infty} \ ,$$
where
$$S_{u,t}=E^{(\vec{p}, \vec{z})}[e^{\eta_u T^{vt}_{ct}}\mathbf{1}_{T^{vt}_{ct}<\infty}] \ .$$

Then we get
$$P^{(\vec{p}, \vec{z})}\left(\dfrac{T^{vt}_{ct}}{t}\in B_\delta(u)\right)
\geq
e^{-\eta_u u t-\delta t|\eta_u|}P^{(\vec{p}, \vec{z}), u, t}\left(\dfrac{T^{vt}_{ct}}{t}\in B_\delta(u)\right)
E^{(\vec{p}, \vec{z})}[e^{\eta_uT^{vt}_{ct}}\mathbf{1}_{T^{vt}_{ct}<\infty}] \ .$$

One can show in the same way as in \cite[p.77]{CometsGantertZeitouni2000} or as in \cite[pp.1195-1196]{Taleb2001}
via a truncation argument, expectation identity, fourth moment estimates and Borel-Cantelli Lemma, that we have
\begin{equation}\label{Theorem:LDPHittingTime:Eq:MeasureChangedConcentration}
\lim\limits_{t\rightarrow\infty}P^{(\vec{p}, \vec{z}), u, t}\left(\dfrac{T^{vt}_{ct}}{t}\in B_\delta(u)\right)=1 \ .
\end{equation}

With \eqref{Theorem:LDPHittingTime:Eq:MeasureChangedConcentration} at hand,
we can conclude that
$$\begin{array}{ll}
\liminf\limits_{t\rightarrow\infty}\dfrac{1}{t}\ln P^{(\vec{p}, \vec{z})}\left(\dfrac{T_{ct}^{vt}}{t}\in B_\delta(u)\right)
& \geq -\eta_u u -\delta |\eta_u|+(v-c)\mu(\eta_u)
\\
& = (v-c)\left[\eta_u\dfrac{u}{v-c}-\mu(\eta_u)\right]-\delta|\eta_u|
\\
& = (v-c)I\left(\dfrac{u}{v-c}\right)-\delta|\eta_u|
\end{array}$$
which implies the lower bound \eqref{Theorem:LDPHittingTime:Eq:LowerBound}.
\end{proof}

The following theorem gives the large deviations principle for the multi-skewed
Brownian motion $Y^{y}_t$, which starts from $Y_0^y=y\in \mathbb{R}$.

\begin{theorem}[large deviations principle for the multi-skewed Brownian motion $Y_t$]\label{Theorem:LDPSkewBM}
Almost surely with respect to $\mathbf{P}$  the following estimates hold. Let $v>0$
and $\kappa\in (0,1]$. For any closed set $G\subset [(\mu'(0))^{-1},\infty)$ we have
\begin{equation}\label{Theorem:LDPSkewBM:Eq:UpperBoundPositive}
\limsup\limits_{t\rightarrow\infty}\dfrac{1}{\kappa t}\ln P^{(\vec{p}, \vec{z})}
\left(\dfrac{vt-Y^{vt}_{\kappa t}}{\kappa t}\in G\right)\leq -\inf\limits_{c\in G}cI\left(\dfrac{1}{c}\right) \ ;
\end{equation}
and for any open set $F\subset [(\mu'(0))^{-1},\infty)$ we have
\begin{equation}\label{Theorem:LDPSkewBM:Eq:LowerBoundPositive}
\liminf\limits_{t\rightarrow\infty}\dfrac{1}{\kappa t}\ln P^{(\vec{p}, \vec{z})}\left(\dfrac{vt-Y^{vt}_{\kappa t}}{\kappa t}\in F\right)
\geq -\inf\limits_{c\in F}cI\left(\dfrac{1}{c}\right) \ .
\end{equation}
For any closed set $G\subset (-\infty, -(\mu'(0))^{-1}]$ we have
\begin{equation}\label{Theorem:LDPSkewBM:Eq:UpperBoundNegative}
\limsup\limits_{t\rightarrow\infty}\dfrac{1}{\kappa t}\ln P^{(\vec{p}, \vec{z})}
\left(\dfrac{-vt-Y^{-vt}_{\kappa t}}{\kappa t}\in G\right)\leq -\inf\limits_{c\in G}|c|I\left(\dfrac{1}{|c|}\right) \ ;
\end{equation}
and for any open set $F\subset (-\infty, -(\mu'(0))^{-1}]$ we have
\begin{equation}\label{Theorem:LDPSkewBM:Eq:LowerBoundNegative}
\liminf\limits_{t\rightarrow\infty}\dfrac{1}{\kappa t}\ln P^{(\vec{p}, \vec{z})}
\left(\dfrac{-vt-Y^{-vt}_{\kappa t}}{\kappa t}\in F\right)
\geq -\inf\limits_{c\in F}|c|I\left(\dfrac{1}{|c|}\right) \ .
\end{equation}
\end{theorem}

\begin{proof}
Due to the symmetry identity \eqref{Lemma:ElementaryPropertiesw:Eq:SymmstryProbabilityDistributionSkewBM}, we see that
\eqref{Theorem:LDPSkewBM:Eq:UpperBoundNegative} and \eqref{Theorem:LDPSkewBM:Eq:LowerBoundNegative} follow from
\eqref{Theorem:LDPSkewBM:Eq:UpperBoundPositive} and \eqref{Theorem:LDPSkewBM:Eq:LowerBoundPositive}, respectively. So we can focus on the proof of
\eqref{Theorem:LDPSkewBM:Eq:UpperBoundPositive} and \eqref{Theorem:LDPSkewBM:Eq:LowerBoundPositive}.
These two identities are parallel to Theorem 5.2 in \cite{FreidlinHu13},
Theorem 2.4 of \cite{Nolen2009} and \cite[Section 5]{Taleb2001}, and we will employ very similar methods in the proof.
However, the symmetric branching structure in our case will bring in new technical differences, as the reader will notice
in our use of Lemma \ref{Lemma:ExponentialTightnessHittingTimePositiveNegative} during the proof.

First, for $c\geq 0$ we have
\begin{equation}\label{Theorem:LDPSkewBM:Eq:DualityUpperBound}
P^{(\vec{p}, \vec{z})}\left(\dfrac{vt-Y^{vt}_{\kappa t}}{\kappa t}>c\right)\leq P^{(\vec{p}, \vec{z})}\left(\dfrac{T^{vt}_{(v-c\kappa)t}}{t}<\kappa\right) \ .
\end{equation}
Thus due to Theorem \ref{Theorem:LDPHittingTime} we conclude that
$$\limsup\limits_{t\rightarrow\infty}\dfrac{1}{\kappa t}
\ln P^{(\vec{p}, \vec{z})}\left(\dfrac{vt-Y^{vt}_{\kappa t}}{\kappa t}>c\right)
\leq -c\inf\limits_{a\in (0,\kappa)}I\left(\dfrac{a}{c\kappa}\right)=-cI\left(\dfrac{1}{c}\right)$$
whenever $c\geq (\mu'(0))^{-1}$. This proves the upper bound \eqref{Theorem:LDPSkewBM:Eq:UpperBoundPositive}.

Let $u\geq (\mu'(0))^{-1}$. Let $\varepsilon>0$ and $\delta>0$ be given. We have the identity
\begin{equation}\label{Theorem:LDPSkewBM:Eq:LowerBoundDualityIdentity:Step1}
P^{(\vec{p}, \vec{z})}\left(\dfrac{vt-Y^{vt}_{\kappa t}}{\kappa t}\in B_{\delta}(u)\right)
=P^{(\vec{p}, \vec{z})}\left(\left|Y^{vt}_{\kappa t}-(v-\kappa u)t\right|< \kappa t \delta\right) \ .
\end{equation}

We apply the method in \cite[Section 5]{Taleb2001}. By splitting the event
$\{(1-\varepsilon)\kappa t< T^{vt}_{(v-\kappa u)t}<\kappa t\}$ into two parts depending on
whether or not $|Y^{vt}_{\kappa t}-(v-\kappa u)t|< \kappa t \delta$,
\begin{equation}\label{Theorem:LDPSkewBM:Eq:LowerBoundDualityIdentity:Step2}
\begin{array}{ll}
& P^{(\vec{p}, \vec{z})}\left((1-\varepsilon)\kappa t< T^{vt}_{(v-\kappa u)t}<\kappa t\right)
\\
\leq
& P^{(\vec{p}, \vec{z})}\left(|Y^{vt}_{\kappa t}-(v-\kappa u)t|< \kappa t \delta\right)
\\
& + P^{(\vec{p}, \vec{z})}\left(\left|Y^{vt}_{\kappa t}-(v-\kappa u)t\right|\geq  \kappa t \delta; (1-\varepsilon)\kappa t< T^{vt}_{(v-\kappa u)t}<\kappa t\right) \ .
\end{array}
\end{equation}

Combining \eqref{Theorem:LDPSkewBM:Eq:LowerBoundDualityIdentity:Step1} and
\eqref{Theorem:LDPSkewBM:Eq:LowerBoundDualityIdentity:Step2} we see that
\begin{equation}\label{Theorem:LDPSkewBM:Eq:LowerBoundDualityIdentity:Step3}
\begin{array}{ll}
& P^{(\vec{p}, \vec{z})}\left(\dfrac{vt-Y^{vt}_{\kappa t}}{\kappa t}\in B_{\delta}(u)\right)
\\
\geq & P^{(\vec{p}, \vec{z})}\left((1-\varepsilon)\kappa t< T^{vt}_{(v-\kappa u)t}<\kappa t\right)
\\
& -
P^{(\vec{p}, \vec{z})}\left(\left|Y^{vt}_{\kappa t}-(v-\kappa u)t\right|\geq  \kappa t \delta; (1-\varepsilon)\kappa t< T^{vt}_{(v-\kappa u)t}<\kappa t\right) \ .
\end{array}
\end{equation}

The last term
\begin{equation}\label{Theorem:LDPSkewBM:Eq:LowerBoundDualityIdentity:Step4}
\begin{array}{ll}
& P^{(\vec{p}, \vec{z})}\left(\left|Y^{vt}_{\kappa t}-(v-\kappa u)t\right|\geq \kappa t \delta; (1-\varepsilon)\kappa t< T^{vt}_{(v-\kappa u)t}<\kappa t\right)
\\
\leq & P^{(\vec{p}, \vec{z})} \left(\sup\limits_{0<s-T^{vt}_{(v-\kappa u)t}<\varepsilon\kappa t}
\left|Y^{vt}_s-(v-\kappa u)t\right|\geq \kappa t \delta\right)
\\
= & P^{(\vec{p}, \vec{z})} \left(\sup\limits_{0<s<\varepsilon\kappa t}
\left|Y^{(v-\kappa u)t}_s-(v-\kappa u)t \right|\geq \kappa t \delta\right) \ ,
\end{array}
\end{equation}
where the first inequality is due to the fact that $0<\kappa t- T^{vt}_{(v-\kappa u)t}<\varepsilon\kappa t$ and the second equality is due to the strong Markov property
of $Y_t$. We can then turn the above estimate to the hitting time by duality to obtain that
\begin{equation}\label{Theorem:LDPSkewBM:Eq:LowerBoundDualityIdentity:Step5}
\begin{array}{ll}
& P^{(\vec{p}, \vec{z})} \left(\sup\limits_{0<s<\varepsilon\kappa t}
\left|Y^{(v-\kappa u)t}_s-(v-\kappa u)t \right|\geq \kappa t \delta\right)
\\
= & P^{(\vec{p}, \vec{z})} \left( T^{(v-\kappa u)t}_{(v-\kappa u)t-\kappa t \delta}\wedge  T^{(v-\kappa u)t}_{(v-\kappa u)t+\kappa t \delta}<\varepsilon \kappa t\right)
\\
\leq & P^{(\vec{p}, \vec{z})} \left( T^{(v-\kappa u)t}_{(v-\kappa u)t-\kappa t \delta}<\varepsilon \kappa t\right)+
 P^{(\vec{p}, \vec{z})} \left( T^{(v-\kappa u)t}_{(v-\kappa u)t+\kappa t \delta}<\varepsilon \kappa t\right) \ .
\end{array}
\end{equation}

By Lemma \ref{Lemma:ExponentialTightnessHittingTimePositiveNegative},
\begin{equation}\label{Theorem:LDPSkewBM:Eq:LowerBoundDualityIdentity:Step6}
\lim\limits_{\varepsilon\rightarrow 0}\limsup\limits_{t\rightarrow\infty}\dfrac{1}{t}\ln
\left[P^{(\vec{p}, \vec{z})} \left( T^{(v-\kappa u)t}_{(v-\kappa u)t-\kappa t \delta}<\varepsilon \kappa t\right)
+P^{(\vec{p}, \vec{z})} \left( T^{(v-\kappa u)t}_{(v-\kappa u)t+\kappa t \delta}<\varepsilon \kappa t\right)\right]=-\infty \ .
\end{equation}

From \eqref{Theorem:LDPSkewBM:Eq:LowerBoundDualityIdentity:Step4}, \eqref{Theorem:LDPSkewBM:Eq:LowerBoundDualityIdentity:Step5},
\eqref{Theorem:LDPSkewBM:Eq:LowerBoundDualityIdentity:Step6} we obtain that
\begin{equation}\label{Theorem:LDPSkewBM:Eq:LowerBoundDualityIdentity:Step7}
\lim\limits_{\varepsilon\rightarrow 0}\limsup\limits_{t\rightarrow\infty}\dfrac{1}{t}\ln
P^{(\vec{p}, \vec{z})}\left(\left|Y^{vt}_{\kappa t}-(v-\kappa u)t\right|\geq  \kappa t \delta; (1-\varepsilon)\kappa t< T^{vt}_{(v-\kappa u)t}<\kappa t\right)=-\infty \ .
\end{equation}

Therefore by \eqref{Theorem:LDPSkewBM:Eq:LowerBoundDualityIdentity:Step3}
and \eqref{Theorem:LDPSkewBM:Eq:LowerBoundDualityIdentity:Step7}, combined with \eqref{Theorem:LDPHittingTime:Eq:LowerBound} in Theorem \ref{Theorem:LDPHittingTime},
we obtain
$$\begin{array}{l}
\liminf\limits_{t\rightarrow\infty}\dfrac{1}{t}\ln P^{(\vec{p},\vec{z})}\left(\dfrac{vt-Y^{vt}_{\kappa t}}{\kappa t}\in B_\delta(u)\right)
\\
\geq \liminf\limits_{\varepsilon\rightarrow 0}\liminf\limits_{t\rightarrow\infty}\dfrac{1}{t}\ln P^{(\vec{p}, \vec{z})}\left(T^{vt}_{(v-\kappa u)t}\in \left((1-\varepsilon)\kappa t, \kappa t\right)\right)
\\
=-\kappa uI\left(\dfrac{1}{u}\right) \ ,
\end{array}$$
which proves the lower bound \eqref{Theorem:LDPSkewBM:Eq:LowerBoundPositive}.
\end{proof}

Finally we provide the technical Lemma \ref{Lemma:ExponentialTightnessHittingTimePositiveNegative}
that we have used in the proof of Theorem \ref{Theorem:LDPSkewBM}, and we will be using it again during the proof of some Lemmas that leads to the proof of Theorem \ref{Theorem:WaveSpeed}, in particular in Lemmas \ref{Lemma:WaveSpeedInsideWaveFrontLowerBound}, \ref{Lemma:WaveSpeedInsideWaveFront}.

\begin{lemma}\label{Lemma:ExponentialTightnessHittingTimePositiveNegative}
For any $a,b\in \mathbb{R}$ such that $a\neq b$, there exist some $\varepsilon_0=\varepsilon_0(a,b,\overline{\ell},\underline{\ell},\overline{d}, \mathbf{E}\ell_0)>0$
depending on $a, b$ and the constants $\overline{\ell}, \underline{\ell}, \overline{d}$, $\mathbf{E}\ell_0$ that are related to the tree structure,
such that for any $0<\varepsilon<\varepsilon_0$
and any $M>0$, we have
\begin{equation}\label{Lemma:ExponentialTightnessHittingTimePositiveNegative:Eq:Tightness}
\limsup\limits_{t\rightarrow\infty}\dfrac{1}{t}\ln
P^{(\vec{p}, \vec{z})} \left( T^{at}_{bt}<\varepsilon t\right)\leq -M \ ,
\end{equation}
almost surely with respect to $\mathbf{P}$.
\end{lemma}

\begin{proof}

By Chebyshev's inequality, for any $\lambda>0$,
$$P^{(\vec{p}, \vec{z})} \left( T^{at}_{bt}<\varepsilon t\right)
= P^{(\vec{p}, \vec{z})} \left( e^{-\lambda T^{at}_{bt}}>e^{-\lambda \varepsilon t}\right)
\leq  e^{\lambda \varepsilon t}E^{(\vec{p}, \vec{z})} e^{-\lambda T^{at}_{bt}} \ ,$$
and therefore
\begin{equation}\label{Lemma:ExponentialTightnessHittingTimePositiveNegative:Eq:ExpConcentration-1}
\dfrac{1}{t}\ln P^{(\vec{p}, \vec{z})} \left( T^{at}_{bt}<\varepsilon t\right)
\leq \lambda \varepsilon + \dfrac{1}{t}\ln E^{(\vec{p}, \vec{z})} e^{-\lambda T^{at}_{bt}} \ .
\end{equation}

It now suffices to prove
\begin{equation}\label{Lemma:ExponentialTightnessHittingTimePositiveNegative:Eq:LinearLowerBoundHittingTime}
\limsup\limits_{t\rightarrow \infty}\dfrac{1}{t}\ln E^{(\vec{p}, \vec{z})} e^{-\lambda T^{at}_{bt}}\leq - C\lambda|b-a| \ ,
\end{equation}
where $C=C(\overline{\ell}, \underline{\ell}, \overline{d}, \mathbf{E}\ell_0)>0$
depends on the structure of the tree $\mathbb{T}_{\vec{d}, \vec{\ell}}$. This is because with
\eqref{Lemma:ExponentialTightnessHittingTimePositiveNegative:Eq:LinearLowerBoundHittingTime} we can bound
from \eqref{Lemma:ExponentialTightnessHittingTimePositiveNegative:Eq:ExpConcentration-1} that
\begin{equation}\label{Lemma:ExponentialTightnessHittingTimePositiveNegative:Eq:ExpConcentration-2}
\limsup\limits_{t\rightarrow\infty}\dfrac{1}{t}\ln P^{(\vec{p}, \vec{z})} \left(T^{at}_{bt}<\varepsilon t\right)
\leq \lambda \varepsilon -\lambda C|b-a|=-\lambda[C|b-a|-\varepsilon] \ .
\end{equation}
We can then pick $0<\varepsilon_0<\dfrac{1}{2}C|b-a|$ and choose $\lambda>\dfrac{2M}{C|b-a|}$
to conclude \eqref{Lemma:ExponentialTightnessHittingTimePositiveNegative:Eq:Tightness}.

It remains to prove \eqref{Lemma:ExponentialTightnessHittingTimePositiveNegative:Eq:LinearLowerBoundHittingTime}.
If $a$ and $b$ have different signs, then by strong Markov property of $Y_t$ we have $T^{at}_{bt}=T^{at}_0+T^{0}_{bt}$.
Thus without loss of generality we only need to consider the cases $0\leq a< b$ or $0\leq b< a$.
We label the interface points $z\in (z_i)_{i\in \mathbb{Z}}$ between $at$ and $bt$ as
$at\leq z(1)<z(2)<...<z(n)\leq bt$ (if $0\leq a< b$) or
$at\geq z(1)>z(2)>...>z(n)\geq bt$ (if $0\leq b<a$). Here $n=n(a,b,t)$ is the number of interface points
between $at$ and $bt$. By Lemma \ref{Assumption:ErgodicEnvironments},
$$\lim\limits_{t\rightarrow \infty}\dfrac{|z(n)-z(1)|}{n}=\mathbf{E}\ell_0 \ .$$
Since $|b-a|t-2\overline{\ell}\leq |z(n)-z(1)|\leq |b-a|t$, from the above we have
\begin{equation}\label{Lemma:ExponentialTightnessHittingTimePositiveNegative:Eq:RatioLimitnOvert}
\lim\limits_{t\rightarrow\infty} \dfrac{n}{t}=C_1|b-a| \ ,
\end{equation}
for constant $C_1=\dfrac{1}{\mathbf{E}\ell_0}>0$.

We can then write
\begin{equation}\label{Lemma:ExponentialTightnessHittingTimePositiveNegative:Eq:HittingTimeDecopositionInterfacePoints}
T^{at}_{bt}=T^{at}_{z(1)}+T^{z(1)}_{z(2)}+...+T^{z(n-1)}_{z(n)}+T^{z(n)}_{bt} \ .
\end{equation}
By the strong Markov property of $Y_t$, the sequence $T^{z(k)}_{z(k+1)}$ is a $P^{(\vec{p}, \vec{z})}$-independent sequence so that
from \eqref{Lemma:ExponentialTightnessHittingTimePositiveNegative:Eq:HittingTimeDecopositionInterfacePoints} we obtain
\begin{equation}\label{Lemma:ExponentialTightnessHittingTimePositiveNegative:Eq:LaplaceTransformHittingTimeDecopositionInterfacePoints}
\ln E^{(\vec{p}, \vec{z})}e^{-\lambda T^{at}_{bt}}=
\ln E^{(\vec{p}, \vec{z})}e^{-\lambda T^{at}_{z(1)}}+
\ln E^{(\vec{p}, \vec{z})}e^{-\lambda T^{z(1)}_{z(2)}}+
...+
\ln E^{(\vec{p}, \vec{z})}e^{-\lambda T^{z(n-1)}_{z(n)}}+
\ln E^{(\vec{p}, \vec{z})}e^{-\lambda T^{z(n)}_{bt}} \ .
\end{equation}

If $0\leq b<a$, then $z(k)> z(k+1)>0$ for all $1\leq k\leq n-1$. By the same reason as we prove
part (1) in Lemma \ref{Lemma:ElementaryPropertiesw}, as well as
Lemma \ref{Assumption:ErgodicEnvironments}, the sequence
$\ln E^{(\vec{p}, \vec{z})}e^{-\lambda T^{z(k)}_{z(k+1)}}$ is a stationary ergodic sequence with respect to $\mathbf{P}$, so that by the
Law of Large Numbers for ergodic sequences and \eqref{Lemma:ExponentialTightnessHittingTimePositiveNegative:Eq:LaplaceTransformHittingTimeDecopositionInterfacePoints} we have, with $\mathbf{P}$-probability $1$, that

\begin{equation}\label{Lemma:ExponentialTightnessHittingTimePositiveNegative:Eq:LLN-Ergodic-lnLaplaceTransform}
\lim\limits_{n\rightarrow\infty} \dfrac{1}{n}\ln E^{(\vec{p}, \vec{z})}e^{-\lambda T^{at}_{bt}}
=\mathbf{E} \left(\ln E^{(\vec{p}, \vec{z})}e^{-\lambda T^{z(1)}_{z(2)}}\right) \ .
\end{equation}

Combining \eqref{Lemma:ExponentialTightnessHittingTimePositiveNegative:Eq:LLN-Ergodic-lnLaplaceTransform} and
\eqref{Lemma:ExponentialTightnessHittingTimePositiveNegative:Eq:RatioLimitnOvert} we see that

\begin{equation}\label{Lemma:ExponentialTightnessHittingTimePositiveNegative:Eq:LinearLowerBoundHittingTime-Prestep1}
\lim\limits_{t\rightarrow \infty}\dfrac{1}{t}\ln E^{(\vec{p}, \vec{z})}e^{-\lambda T^{at}_{bt}}
=C_1|b-a|\mathbf{E} \left(\ln E^{(\vec{p}, \vec{z})}e^{-\lambda T^{z(1)}_{z(2)}}\right) \ .
\end{equation}
By Lemma \ref{Assumption:ErgodicEnvironments} we have that there exist some
$C_2=C_2(\overline{\ell}, \underline{\ell}, \overline{d})>0$ that $T^{z(1)}_{z(2)}\geq C_2$
with $P^{(\vec{p}, \vec{z})}$-probability $1$. Thus

\begin{equation}\label{Lemma:ExponentialTightnessHittingTimePositiveNegative:Eq:UpperBoundEnvironmentExpectedLaplaceTransform}
\mathbf{E} \left(\ln E^{(\vec{p}, \vec{z})}e^{-\lambda T^{z(1)}_{z(2)}}\right)\leq -\lambda C_2 \ .
\end{equation}
Finally \eqref{Lemma:ExponentialTightnessHittingTimePositiveNegative:Eq:LinearLowerBoundHittingTime-Prestep1} and
\eqref{Lemma:ExponentialTightnessHittingTimePositiveNegative:Eq:UpperBoundEnvironmentExpectedLaplaceTransform} conclude
\eqref{Lemma:ExponentialTightnessHittingTimePositiveNegative:Eq:LinearLowerBoundHittingTime} with $C=C_1C_2$.

If $0\leq a< b$, then $0<z(k)<z(k+1)$ for all $1\leq k\leq n-1$. In this case the sequence
$\ln E^{(\vec{p}, \vec{z})}e^{-\lambda T^{z(k)}_{z(k+1)}}$ is not a stationary ergodic sequence with respect to $\mathbf{P}$,
but as $t,n\rightarrow \infty$ and thus $k\rightarrow \infty$, the sequence
$\ln E^{(\vec{p}, \vec{z})}e^{-\lambda T^{z(k)}_{z(k+1)}}$ becomes asymptotically stationary ergodic and its distribution tends to
$\lim\limits_{k\rightarrow\infty}\ln E^{(\vec{p}, \vec{z})}e^{-\lambda T^{z(k)}_{z(k+1)}}$. Hence
\eqref{Lemma:ExponentialTightnessHittingTimePositiveNegative:Eq:LLN-Ergodic-lnLaplaceTransform} will be replaced by
$$\lim\limits_{n\rightarrow\infty} \dfrac{1}{n}\ln E^{(\vec{p}, \vec{z})}e^{-\lambda T^{at}_{bt}}
=\mathbf{E} \lim\limits_{k\rightarrow\infty}\left(\ln E^{(\vec{p}, \vec{z})}e^{-\lambda T^{z(k)}_{z(k+1)}}\right) \ .
$$
Therefore, in this case we can still obtain
\eqref{Lemma:ExponentialTightnessHittingTimePositiveNegative:Eq:LinearLowerBoundHittingTime-Prestep1} and
\eqref{Lemma:ExponentialTightnessHittingTimePositiveNegative:Eq:UpperBoundEnvironmentExpectedLaplaceTransform} if we replace
$\ln E^{(\vec{p}, \vec{z})}e^{-\lambda T^{z(1)}_{z(2)}}$ by
$\lim\limits_{k\rightarrow\infty}\left(\ln E^{(\vec{p}, \vec{z})}e^{-\lambda T^{z(k)}_{z(k+1)}}\right)$. So we still conclude
\eqref{Lemma:ExponentialTightnessHittingTimePositiveNegative:Eq:LinearLowerBoundHittingTime}.
\end{proof}

\section{From LDP to wave propagation  on infinite random trees} \label{Sec:WavePropagation}

Based on the large deviations principle established for the multi-skewed Brownian motion $Y_t$ as in Theorems \ref{Theorem:LDPHittingTime} and \ref{Theorem:LDPSkewBM} in Section \ref{Sec:LDP}, we establish in this section the wavefront propagation for
FKPP  equation \eqref{Eq:FisherKPPInitialBoundaryConditionVertices}
on infinite random tree $\mathbb{T}_{\vec{d}, \vec{\ell}}$.

Let us define a non-random constant $c^*>0$ as the solution to the equation
\begin{equation}\label{Eq:WaveSpeedFormula}
c^*I\left(\dfrac{1}{c^*}\right)=\beta \ ,
\end{equation}
where $\beta=f'(0)$ is the constant in \eqref{Eq:FisherKPPInitialBoundaryConditionVertices}.
The next lemma characterizes $c^*$.
\begin{lemma}\label{Lemma:SolutionPropertyWaveSpeedFormula}
When $\beta>\max\left(\eta_c, \dfrac{-\mu(0)}{\mu'(0)}\right)$, the equation \eqref{Eq:WaveSpeedFormula}
admits a unique solution $c^*>0$ with the following properties:
\begin{itemize}
\item[(1)] $c^*>(\mu'(0))^{-1}$;
\item[(2)] For any $\delta>0$ such that $(c^*-\delta, c^*+\delta)\subset (0, \infty)$, there exist some $\widetilde{\varepsilon}=\widetilde{\varepsilon}(\delta)>0$ such that we have
\begin{equation}\label{Lemma:SolutionPropertyWaveSpeedFormula:Eq:c-small}
cI\left(\dfrac{1}{c}\right)-\beta<-\varepsilon  \text{ whenever } 0<c<c^*-\delta \ ,
\end{equation}
and
\begin{equation}\label{Lemma:SolutionPropertyWaveSpeedFormula:Eq:c-large}
cI\left(\dfrac{1}{c}\right)-\beta>\varepsilon  \text{ whenever } c>c^*+\delta \ ,
\end{equation}
where the positive constant $\varepsilon=c\widetilde{\varepsilon}$ depends on $\delta$ and the choice of $c\in (0,c^*-\delta)\cup (c^*+\delta, \infty)$.
\item[(3)] When $c>c^*$ the function $cI\left(\dfrac{1}{c}\right)$ is monotonically increasing as $c$ is increasing;
\end{itemize}
\end{lemma}

\begin{proof}
The validity of the statements in this Lemma can be seen from Figure \ref{Fig:I-a-Graph}.
To be precise, by property (3) of Lemma \ref{Lemma:PropertyIa}, the function $I(a)-\beta a$ approaches $+\infty$
when $a\rightarrow 0$. Since $I'(a)\leq \eta_c$ for all $a\in (0, \infty)$, and $\beta>\eta_c$,
the function $I(a)-\beta a$ is monotonically decreasing in $a$ and it approaches $-\infty$ as $a\rightarrow \infty$. Thus there exists
a unique $a^*\in (0, \infty)$ such that $I(a^*)-\beta a^*=0$. We can then set $c^*=\dfrac{1}{a^*}$.

\begin{itemize}
\item[(1)] Since $\beta>\dfrac{-\mu(0)}{\mu'(0)}$, and both the points
$\left(\dfrac{1}{c^*}, \dfrac{\beta}{c^*}\right)$ and $(\mu'(0), -\mu(0))$ lie on the graph of $I(a)$, we see that
the intersection of the line $\beta a$ with $I(a)$ must happen at a point with the $a$-coordinate
less that $\mu'(0)$. That is, $a^*=\dfrac{1}{c^*}< \mu'(0)$, i.e., $c^*> (\mu'(0))^{-1}$.

\item[(2)] As we have seen, the function $I(a)-\beta a$ is strictly monotonically decreasing from $+\infty$
to $-\infty$ as $a$ goes from $0$ to $\infty$. This implies that for any $\widetilde{\delta}>0$ and
any $a>a^*+\widetilde{\delta}$ we have $I(a)-\beta a<-\widetilde{\varepsilon}$, any $a<a^*-\widetilde{\delta}$ we have
$I(a)-\beta a>\widetilde{\varepsilon}$, where $\widetilde{\varepsilon}$ is a positive constant that may depend on $\widetilde{\delta}$
and $a$. Set $c=\dfrac{1}{a}$, we get from here that for any $\delta>0$, for any $c<c^*-\delta$ we have
$I\left(\dfrac{1}{c}\right)-\dfrac{\beta}{c}<-\widetilde{\varepsilon}$, and for any $c>c^*+\delta$ we have
$I\left(\dfrac{1}{c}\right)-\dfrac{\beta}{c}>\widetilde{\varepsilon}$. Set $\varepsilon=c\widetilde{\varepsilon}$, we get the statement.

\item[(3)] By part (2) of Lemma \ref{Lemma:PropertyIa}, the function $I\left(\dfrac{1}{c}\right)$ is a monotonically increasing
function of $c$ when $\dfrac{1}{c}<\mu'(0)$, i.e., $c>[\mu'(0)]^{-1}=c^*$. This implies the statement.
\end{itemize}
\end{proof}

\begin{remark}\rm\label{Remark:PurposesOfConditions-BetaGreaterEtac-BetaGreaterMuStuff}
As we will see in the arguments below that to prove Theorem \ref{Theorem:WaveSpeed}, the condition
$\beta> \dfrac{-\mu(0)}{\mu'(0)}$ is to ensure that we can use the LDP Theorems \ref{Theorem:LDPHittingTime},
\ref{Theorem:LDPSkewBM} in our analysis of the wavefront propagation, and the condition
$\beta> \eta_c$ is to ensure that property (2) in Lemma \ref{Lemma:SolutionPropertyWaveSpeedFormula} holds,
which ensures the existence of a unique wavefront.
\end{remark}

\begin{remark}\rm\label{Remark:ConditionBetaGreaterEtacImplyMuStuff}
According to part (7) of Lemma \ref{Lemma:PropertiesMu}, once we have $\mu(\eta_c-)\leq 0$,
then $\dfrac{-\mu(0)}{\mu'(0)}> \eta_c$, so that the condition
$\beta> \max\left(\eta_c, \dfrac{-\mu(0)}{\mu'(0)}\right)$ becomes the only condition that
$\beta> \dfrac{-\mu(0)}{\mu'(0)}$. However, when $\mu(\eta_c-)>0$, it might happen that
$\dfrac{-\mu(0)}{\mu'(0)}< \eta_c$. To this end, Figure \ref{Fig:I-a-Graph} parts (a-1),
(a-2), (d-1) demonstrate the case when $\dfrac{-\mu(0)}{\mu'(0)}\geq \eta_c$, and parts
(b-1), (b-2), (d-2) demonstrate the case when $\dfrac{-\mu(0)}{\mu'(0)}< \eta_c$.
\end{remark}

Due to Lemma \ref{Lemma:SolutionPropertyWaveSpeedFormula}, in the following we will obtain our result about the existence of a travelling wavefront based on the assumption that $\beta$ is large enough. Thus we have
\begin{assumption}\label{Assumption:ReactionRateBetaLarge}
We assume that the reaction rate \begin{equation}\label{Assumption:ReactionRateBetaLarge:Eq:BetaLarge}
\beta>\max\left(\eta_c, \dfrac{-\mu(0)}{\mu'(0)}\right) \equiv \beta_c \ .
\end{equation}
\end{assumption}

Our main result that characterizes the wave-speed is given by the following
\begin{theorem}[wavefront propagation for FKPP equation on infinite random tree $\mathbb{T}_{\vec{d}, \vec{\ell}}$]\label{Theorem:WaveSpeed}
Assume Assumption \ref{Assumption:ReactionRateBetaLarge} holds. For any closed set $F\subset (-\infty, -c^*)\cup (c^*, \infty)$ we have
\begin{equation}\label{Theorem:WaveSpeed:UpperBound}
\lim\limits_{t\rightarrow\infty}\sup\limits_{c\in F}v(t, ct)=0
\end{equation}
almost surely with respect to $\mathbf{P}$. For any compact set $K\subset (-c^*, c^*)$ we have
\begin{equation}\label{Theorem:WaveSpeed:LowerBound}
\lim\limits_{t\rightarrow\infty}\inf\limits_{c\in K}v(t, ct)=1
\end{equation}
almost surely with respect to $\mathbf{P}$.
\end{theorem}

\begin{proof}
The proof makes use of the arguments from the classical variational analysis as in \cite[Chapter 7, Theorem 3.1]{FreidlinFunctionalBook},
\cite[Theorem 4.1, Lemma 4.1, Lemma 4.2]{Nolen2009} and \cite[Theorem 1 and Theorem 2]{Nolen-XinCMP2007}.
Lemma \ref{Lemma:WaveSpeedOutOfWaveFront} provides the upper bound \eqref{Theorem:WaveSpeed:UpperBound}
for the behavior of the wave outside $(-c^*, c^*)$. Lemma \ref{Lemma:WaveSpeedInsideWaveFrontLowerBound} provides the lower bound \eqref{Theorem:WaveSpeed:LowerBound} for the behavior of the wave inside $(-c^*, c^*)$. Lemmas \ref{Lemma:WaveSpeedInsideWaveFront} and
\ref{Lemma:Finiteness-q-WaveSpeedInsideWaveFront} are of auxiliary nature, but they are important in proving Lemma
\ref{Lemma:WaveSpeedInsideWaveFrontLowerBound}. Thus the Theorem is proved.
\end{proof}

The following lemma provides the upper bound \eqref{Theorem:WaveSpeed:UpperBound}
for the behavior of the wave outside $(-c^*, c^*)$.
\begin{lemma}\label{Lemma:WaveSpeedOutOfWaveFront}
Suppose Assumption \ref{Assumption:ReactionRateBetaLarge} holds. For any closed set $F\subset (-\infty, -c^*)\cup (c^*, \infty)$,
\begin{equation}\label{Lemma:WaveSpeedOutOfWaveFront:Eq:UpperBound}
\lim\limits_{t\rightarrow\infty}\sup\limits_{c\in F}v(t, ct)=0
\end{equation}
almost surely with respect to $\mathbf{P}$.
\end{lemma}

\begin{proof}
We first consider the case when $c>c^*$. We can apply Lemma \ref{Lemma:proj} and in particular equation \eqref{Lemma:proj:Eq:FeynmanKac-v} and we obtain that,
for the function $v(t,y)$ defined in \eqref{Lemma:proj:Eq:uEqualv} we have
$$
v(t,y) =\displaystyle{E^{(\vec{p}, \vec{z})}_y\left[v_0(Y_t)\exp\left\{\beta \int_0^t \left(1-v(t-s, Y_s)\right)ds\right\}\right]}
\leq \exp(\beta t)E^{(\vec{p}, \vec{z})}_y v_0(Y_t) \ .
$$

Let the support of the function $v_0(y)$ be a compact set $U\subset (-\infty, \infty)$, and further assume that
$U=B_\delta=(-\delta, \delta)$ for some $\delta>0$. Thus we have
$$\begin{array}{ll}
v(t, ct)& \leq \|v_0\|\exp(\beta t)P^{(\vec{p}, \vec{z})}\left(-\delta\leq Y^{ct}(t)\leq \delta\right)
\\
& = \|v_0\|\exp(\beta t)P^{(\vec{p}, \vec{z})}\left(c+\dfrac{\delta}{t}\geq \dfrac{ct-Y^{ct}(t)}{t}\geq c-\dfrac{\delta}{t}\right) \ .
\end{array}$$
By Lemma \ref{Lemma:SolutionPropertyWaveSpeedFormula},
when $c>c^*$, we see from \eqref{Lemma:SolutionPropertyWaveSpeedFormula:Eq:c-large} that $\beta-cI\left(\dfrac{1}{c}\right)<-\varepsilon<0$ for some
$\varepsilon>0$ that may depend on $c$. Notice that since $c^*>(\mu'(0))^{-1}$ due to part (1) of Lemma \ref{Lemma:SolutionPropertyWaveSpeedFormula},
as $t$ is large and $c>c^*$, we have $\left(c-\dfrac{\delta}{t}, c+\dfrac{\delta}{t}\right)\subset [(\mu'(0))^{-1}, \infty)$.
We can then apply Theorem \ref{Theorem:LDPSkewBM} estimate \eqref{Theorem:LDPSkewBM:Eq:UpperBoundPositive} with $\kappa=1$ and $v=c$, and we obtain that
$\limsup\limits_{t\rightarrow\infty}\dfrac{1}{t}\ln v(t, ct)\leq -\dfrac{\varepsilon}{2}$ almost surely with respect to $\mathbf{P}$, which implies that
$\lim\limits_{t\rightarrow\infty}\sup\limits_{c\in F\cap (c^*, \infty)}v(t, ct)=0$ almost surely with respect to $\mathbf{P}$. The case when $c<-c^*$ can be argued similarly using
estimate \eqref{Theorem:LDPSkewBM:Eq:UpperBoundNegative} in Theorem \ref{Theorem:LDPSkewBM}.
\end{proof}

The following lemma provides the lower bound \eqref{Theorem:WaveSpeed:LowerBound} for the behavior of the wave inside $(-c^*, c^*)$.

\begin{lemma}\label{Lemma:WaveSpeedInsideWaveFrontLowerBound}
Suppose Assumption \ref{Assumption:ReactionRateBetaLarge} holds. For any compact set $K\subset (-c^*, c^*)$,
\begin{equation}\label{Lemma:WaveSpeedInsideWaveFrontLowerBound:Eq:LowerBound}
\lim\limits_{t\rightarrow\infty}\inf\limits_{c\in K}v(t, ct)=1
\end{equation}
almost surely with respect to $\mathbf{P}$.
\end{lemma}

\begin{proof}
The proof of makes use of ideas from \cite[Section 5]{Nolen-XinCMP2007} and
\cite[Chapter 7, Theorem 3.1]{FreidlinFunctionalBook}, but is adapted to the case when $\beta$ is large (see Lemma \ref{Lemma:SolutionPropertyWaveSpeedFormula}).
From \eqref{Lemma:proj:Eq:FeynmanKac-v} we have
\begin{equation}\label{Lemma:WaveSpeedInsideWaveFrontLowerBound:Eq:vFeynmanKac}
v(t,y) =\displaystyle{E^{(\vec{p}, \vec{z})}_y\left[v_0(Y_t)\exp\left\{\beta \int_0^t \left(1-v(t-s, Y_s)\right)ds\right\}\right]} \ .
\end{equation}
If $\tau$ is any stopping time, we also have
\begin{equation}\label{Lemma:WaveSpeedInsideWaveFrontLowerBound:Eq:vFeynmanKacStoppingTime}
v(t,y) =\displaystyle{E^{(\vec{p}, \vec{z})}_y\left[v(t-t\wedge\tau, Y_{t\wedge\tau})\exp\left\{\beta \int_0^{t\wedge\tau} \left(1-v(t-s, Y_s)\right)ds\right\}\right]} \ .
\end{equation}
Indeed, since $Y_t$ is a strong Markov process, given $\widetilde{Y}_0=Y_{t\wedge\tau}$, the process $\widetilde{Y}_r=Y_{(t\wedge\tau)+r}$, $0\leq r \leq t-t\wedge\tau$ has the same distribution as $Y$ and hence satisfies \eqref{Lemma:WaveSpeedInsideWaveFrontLowerBound:Eq:vFeynmanKac}, so that
$$v(t-t\wedge\tau,\widetilde{Y}_0) =\displaystyle{E^{(\vec{p}, \vec{z})}_{\widetilde{Y}_0}{\left[v_0(\widetilde{Y}_{t-t\wedge\tau})\exp\left\{\beta \int_0^{t-t\wedge\tau} \left(1-v(t-t\wedge\tau-r, \widetilde{Y}_r)\right)dr\right\}\right]}} \ ,$$
which translates, by setting $s=t\wedge\tau + r$, to
$$v(t-t\wedge\tau,Y_{t\wedge \tau}) =\displaystyle{E^{(\vec{p}, \vec{z})}_{Y_{t\wedge \tau}}{\left[v_0(Y_t)\exp\left\{\beta \int_{t\wedge\tau}^{t} \left(1-v(t-s, Y_s)\right)ds\right\}\right]}} \ .$$
The above equation, when plugged in, justifies \eqref{Lemma:WaveSpeedInsideWaveFrontLowerBound:Eq:vFeynmanKacStoppingTime}.

Therefore, we can obtain estimates on $v$ by choosing stopping times and restricting the expectation to
certain sets of paths. The exponential term inside the expectation will be large when the path $Y_t^y$ passes through
regions where $v$ is small; on the other hand, if $v(t-t\wedge\tau, Y_{t\wedge\tau})$ is too small, then the expectation
as a whole may be small.

For $s\in \mathbb{R}$ we define the set
$$\Psi(s)=\left\{c\in \mathbb{R}; |c|I\left(\dfrac{1}{|c|}\right)-\beta=s\right\}  \text{ and }
\underline{\Psi}(s)=\left\{c\in \mathbb{R}; |c|I\left(\dfrac{1}{|c|}\right)-\beta\leq s\right\} \ .$$

For any $\delta>0$ and $T>1$ we define
$$\Gamma_T=\left([\{1\} \times \underline{\Psi}(\delta)]\cup\left[\bigcup\limits_{1\leq t\leq T}(\{t\} \times t\Psi(\delta))\right]\right) \ .$$
Notice that for $1\leq t_1<t_2$ we have $\Gamma_{t_1}\subset \Gamma_{t_2}$ and
the set $\Gamma\equiv \bigcup\limits_{1<t<\infty}\Gamma_t$ defines the boundary of an unbounded region that spreads outward in $z$ and is linearly in $t$.
Due to part (3) of Lemma \ref{Lemma:SolutionPropertyWaveSpeedFormula}, as $|c|>c^*$ is monotonically increasing we have $|c|I\left(\dfrac{1}{|c|}\right)-\beta>0$
is monotonically increasing. By the argument of Lemma \ref{Lemma:WaveSpeedOutOfWaveFront} this indicates that outside the region $\Gamma$,
as $t$ is sufficiently large,  $v(t, y)$ may be close to zero. But on the boundary of this region $\Gamma$, we have the crucial lower bound from
Lemma \ref{Lemma:WaveSpeedInsideWaveFront}, which gives
\begin{equation}\label{Lemma:WaveSpeedInsideWaveFrontLowerBound:Eq:LowerBoundvBoundaryGamma-t}
v(s, y)\geq e^{-2\delta t} \text{ for all } (s, y)\in \Gamma_t
\end{equation}
when $t$ is sufficiently large.

Let $K$ be a compact set such that $K\subset (-c^*, c^*)$. By \eqref{Lemma:SolutionPropertyWaveSpeedFormula:Eq:c-small}
in Lemma \ref{Lemma:SolutionPropertyWaveSpeedFormula},
for any $c\in K$ we have $|c|I\left(\dfrac{1}{|c|}\right)-\beta<0$.
Set $y=ct$ for some $c\in K$. Set $h\in (0,1)$ and $t>0$. We define the stopping times
$$\begin{array}{lll}
\sigma_h(t)& =& \min\{s\in [0,t]; v(t-s, Y^y_s)\geq h\} \ ,
\\
\sigma_\Gamma(t)& = &\min\{s\in [0,t]; (t-s, Y^y_s)\in \Gamma_t\}  \ ,
\\
\hat{\sigma}(t)& =& \sigma_h(t)\wedge \sigma_\Gamma(t) \ .
\end{array}$$

We then apply \eqref{Lemma:WaveSpeedInsideWaveFrontLowerBound:Eq:vFeynmanKacStoppingTime} with the stopping time
$\hat{\sigma}(t)$ we express $v(t,y)$ as
\begin{equation}\label{Lemma:WaveSpeedInsideWaveFrontLowerBound:Eq:vFeynmanKacStoppingTimeA1A2A3}
v(t,y) =\displaystyle{E^{(\vec{p}, \vec{z})}_y\left[v(t-t\wedge{\hat{\sigma}}, Y_{t\wedge{\hat{\sigma}}})
\exp\left\{\beta \int_0^{t\wedge{\hat{\sigma}}} \left(1-v(t-s, Y_s)\right)ds\right\}(\mathbf{1}_{A_1}+\mathbf{1}_{A_2}+\mathbf{1}_{A_3})\right]} \ ,
\end{equation}
where $A_1, A_2, A_3$ are disjoint sets separating the whole sample space
$$\begin{array}{lll}
A_1 & = & \{\omega; \sigma_h(t)\leq t\} \ ,
\\
A_2 & = & \{\omega; \sigma_h(t)> t \ , \ \sigma_\Gamma(t)\geq rt\} \ ,
\\
A_3 & = & \{\omega; \sigma_h(t)> t \ , \ \sigma_\Gamma(t)< rt\}
\end{array}$$
for some $r\in (0,1)$ to be chosen.

Because $A_1, A_2, A_3$ are disjoint, the expectation \eqref{Lemma:WaveSpeedInsideWaveFrontLowerBound:Eq:vFeynmanKacStoppingTimeA1A2A3}
splits into three integrals. We can bound the first integral over $A_1$ from below by
\begin{equation}\label{Lemma:WaveSpeedInsideWaveFrontLowerBound:Eq:vFeynmanKacStoppingTimeA1}
\displaystyle{E^{(\vec{p}, \vec{z})}_y\left[v(t-t\wedge{\hat{\sigma}}, Y_{t\wedge{\hat{\sigma}}})
\exp\left\{\beta \int_0^{t\wedge{\hat{\sigma}}} \left(1-v(t-s, Y_s)\right)ds\right\}\mathbf{1}_{A_1}\right]}\geq hP^{(\vec{p},\vec{z})}(A_1) \ .
\end{equation}

The second integral over $A_2$ can be bounded from below by
\begin{equation}\label{Lemma:WaveSpeedInsideWaveFrontLowerBound:Eq:vFeynmanKacStoppingTimeA2}
\displaystyle{E^{(\vec{p}, \vec{z})}_y\left[v(t-t\wedge{\hat{\sigma}}, Y_{t\wedge{\hat{\sigma}}})
\exp\left\{\beta \int_0^{t\wedge{\hat{\sigma}}} \left(1-v(t-s, Y_s)\right)ds\right\}\mathbf{1}_{A_2}\right]}\geq e^{-2\delta t}
e^{\beta(1-h) r t}P^{(\vec{p}, \vec{z})}(A_2) \ ,
\end{equation}
where we have used \eqref{Lemma:WaveSpeedInsideWaveFrontLowerBound:Eq:LowerBoundvBoundaryGamma-t}.

Combining \eqref{Lemma:WaveSpeedInsideWaveFrontLowerBound:Eq:vFeynmanKacStoppingTimeA1} and
\eqref{Lemma:WaveSpeedInsideWaveFrontLowerBound:Eq:vFeynmanKacStoppingTimeA2} we obtain that

\begin{equation}\label{Lemma:WaveSpeedInsideWaveFrontLowerBound:Eq:vFeynmanKacStoppingTimeA1A2}
v(t,y)\geq h P^{(\vec{p}, \vec{z})}(A_1)+e^{-2\delta t}e^{\beta(1-h) r t}P^{(\vec{p}, \vec{z})}(A_2) \ .
\end{equation}

We will choose $\delta=\delta(h,r)>0$ to be small so that $-2\delta t+\beta(1-h)rt>0$. Then since $v(t,y)\in (0,1)$
for all $(t,y)$, \eqref{Lemma:WaveSpeedInsideWaveFrontLowerBound:Eq:vFeynmanKacStoppingTimeA1A2} implies that
$P^{(\vec{p}, \vec{z})}(A_2)\rightarrow 0$ exponentially fast as $t\rightarrow\infty$ for small $\delta>0$. Thus if we can show that $P^{(\vec{p}, \vec{z})}(A_3)\rightarrow 0$
as $t\rightarrow\infty$, then we conclude that $P^{(\vec{p}, \vec{z})}(A_1)\rightarrow 1$ as $t\rightarrow\infty$, which then implies that
$v(t,y)>h$ as $t\rightarrow\infty$ for any $h\in (0,1)$, that is \eqref{Lemma:WaveSpeedInsideWaveFrontLowerBound:Eq:LowerBound}.

It remains to show $P^{(\vec{p}, \vec{z})}(A_3)\rightarrow 0$ as $t\rightarrow\infty$.
By Lemma \ref{Lemma:SolutionPropertyWaveSpeedFormula} parts (2) and (3), we see that $\Psi(0)=\{\pm c^*\}$ and $\Psi(\delta)=\{\pm c_\Psi(\delta)\}$ for some
$c_\Psi(\delta)>c^*$. Notice that the initial point $y=ct$ for $c\in K\subset (-c^*, c^*)$, and thus
we have $|c|<c^*<c_\Psi(\delta)$. Thus $\sigma_\Gamma(t)\geq \min\{s\in [0,t]: Y^{ct}_s=\pm c^*(t-s)\}$.
Therefore we have the inclusion of the events
$$\{\sigma_\Gamma(t)\leq rt\}\subseteq \left\{\min\{s\in [0,t]: Y^{ct}_s=\pm c^*(t-s)\}\leq rt\right\} \ .$$
Notice that on the event $\left\{\min\{s\in [0,t]: Y^{ct}_s=\pm c^*(t-s)\}\leq rt\right\}$ we have
$$\min\{s\in [0,t]: Y^{ct}_s=\pm c^*(t-s)\}\geq T^{ct}_{c^*(1-r)t}\wedge T^{ct}_{-c^*(1-r)t} \ ,$$
so that we can simply bound
$$P^{(\vec{p}, \vec{z})}(A_3)\leq P^{(\vec{p}, \vec{z})}(\sigma_\Gamma(t)<rt)
\leq P^{(\vec{p}, \vec{z})}(T^{ct}_{c^*(1-r)t}\wedge T^{ct}_{-c^*(1-r)t}<rt)\rightarrow 0 \ ,$$
when $r>0$ is picked to be sufficiently small, due to Lemma \ref{Lemma:ExponentialTightnessHittingTimePositiveNegative}.
\end{proof}

The following lemma helps to prove Lemma \ref{Lemma:WaveSpeedInsideWaveFrontLowerBound}.
\begin{lemma}\label{Lemma:WaveSpeedInsideWaveFront}
Suppose Assumption \ref{Assumption:ReactionRateBetaLarge} holds. For any compact set $K\subset (-\infty, -c^*) \cup (c^*, \infty)$,
\begin{equation}\label{Lemma:WaveSpeedInsideWaveFront:Eq:LowerBound}
\liminf\limits_{t\rightarrow\infty}\dfrac{1}{t}\ln \inf\limits_{c\in K}v(t, ct)\geq -\max\limits_{c\in K}\left[|c|I\left(\dfrac{1}{|c|}\right)-\beta\right] \ .
\end{equation}
\end{lemma}

\begin{proof}
We use the argument in \cite[Lemma 4.1]{Nolen2009}, \cite[Lemma 7]{Nolen-XinCMP2007},
\cite[Lemma 7.3.2]{FreidlinFunctionalBook}, with various technical differences that come from
Lemma \ref{Lemma:SolutionPropertyWaveSpeedFormula}. The compactness of $K$ implies that
it suffices to show that given $\varepsilon>0$ and any $c$ for which
$|c|I\left(\dfrac{1}{|c|}\right)-\beta>0$, we have
\begin{equation}\label{Lemma:WaveSpeedInsideWaveFront:Eq:LowerBoundEquivalentBallConcentrated}
\liminf\limits_{t\rightarrow\infty}\left(\dfrac{1}{t}\ln \inf\limits_{\widetilde{c}\in B_\delta(c)}v(t, \widetilde{c}t)\right)
\geq \beta-|c|I\left(\dfrac{1}{|c|}\right)-\varepsilon \ ,
\end{equation}
for $\delta>0$ sufficiently small. Due to part (2) of Lemma \ref{Lemma:SolutionPropertyWaveSpeedFormula},
 we see that such a $c$ satisfies $|c|>c^*$. Without loss of generality we can assume that the initial data
$v_0(y)\geq \mathbf{1}_{B_{\delta}(0)}(y)$ for some $\delta>0$, and we can assume that $c>c^*$ with $B_{6\delta}(c)\subset (c^*, \infty)$.
Let us define the limit on the left-hand side of
\eqref{Lemma:WaveSpeedInsideWaveFront:Eq:LowerBoundEquivalentBallConcentrated} as
\begin{equation}\label{Lemma:WaveSpeedInsideWaveFront:Eq:q}
q=\liminf\limits_{t\rightarrow\infty}\left(\dfrac{1}{t}\ln \inf\limits_{\widetilde{c}\in B_{\delta}(c)}v(t, \widetilde{c}t)\right) \ .
\end{equation}

The estimate \eqref{Lemma:Finiteness-q-WaveSpeedInsideWaveFront:Eq:Bound-hatq} in
Lemma \ref{Lemma:Finiteness-q-WaveSpeedInsideWaveFront} immediately implies that $q>-\infty$.
As above we set $c\in K$ and $c>c^*$ so that $cI\left(\dfrac{1}{c}\right)-\beta>0$. Suppose for the moment that $q$ is finite.
By the representation \eqref{Lemma:WaveSpeedInsideWaveFrontLowerBound:Eq:vFeynmanKacStoppingTime} we have for any $\kappa\in (0,1]$
that
\begin{equation}\label{Lemma:WaveSpeedInsideWaveFront:Eq:LowerBoundFeynmanKac}
\inf\limits_{\widetilde{c}\in B_{\delta}(c)}v(t,\widetilde{c}t)
\geq
\displaystyle{\inf\limits_{\widetilde{c}\in B_{\delta}(c)}
E^{(\vec{p}, \vec{z})}_{\widetilde{c}t}\left[v(t-\kappa t, Y_{\kappa t})
\exp\left\{\beta \int_0^{\kappa t} \left(1-v(t-s, Y_s)\right)ds\right\}\cdot \mathbf{1}_A\right]}
\end{equation}
for some $P^{(\vec{p}, \vec{z})}$-adapted set $A$. We pick some small $h>0$ and choose $A$ to be the set of paths
satisfying that for all $\widetilde{c}\in B_\delta(c)$ we have both
\begin{equation}\label{Lemma:WaveSpeedInsideWaveFront:Eq:SetOfPathsACondition-Y}
Y^{\widetilde{c}t}_{\kappa t}\in B_{(1-\kappa)\delta t}\left((1-\kappa)tc\right)
\end{equation}
and
\begin{equation}\label{Lemma:WaveSpeedInsideWaveFront:Eq:SetOfPathsACondition-u}
v(t-s,Y_s^{\widetilde{c}t})\leq h \text{ for all } s\in [0,\kappa t] \ .
\end{equation}
Then
$$\begin{array}{ll}
&\displaystyle{\inf\limits_{\widetilde{c}\in B_{\delta}(c)}
E^{(\vec{p}, \vec{z})}_{\widetilde{c}t}\left[v(t-\kappa t, Y_{\kappa t})
\exp\left\{\beta \int_0^{\kappa t} \left(1-v(t-s, Y_s)\right)ds\right\}\cdot \mathbf{1}_A\right]}
\\
\geq & \inf\limits_{\widetilde{c}\in B_\delta(c)}v((1-\kappa)t, \widetilde{c}(1-\kappa)t)\cdot e^{\beta(1-h)\kappa t}
\cdot \inf\limits_{\widetilde{c}\in B_\delta(c)}P^{(\vec{p}, \vec{z})}(A) \ ,
\end{array}
$$
which gives
$$\begin{array}{ll}
&\dfrac{1}{t}\ln \inf\limits_{\widetilde{c}\in B_{\delta}(c)}v(t, \widetilde{c}t)
\\
\geq & (1-\kappa)\dfrac{1}{(1-\kappa)t}\ln\inf\limits_{\widetilde{c}\in B_\delta(c)}v\left((1-\kappa)t, \widetilde{c}(1-\kappa)t\right)
+\kappa \beta(1-h)+\dfrac{1}{t}\ln \inf\limits_{\widetilde{c}\in B_{\delta}(c)} P^{(\vec{p}, \vec{z})}(A) \ .
\end{array}$$
Thus taking $t\rightarrow\infty$ this gives

\begin{equation}\label{Lemma:WaveSpeedInsideWaveFront:Eq:LowerBound-q}
q\geq \beta(1-h)+\liminf\limits_{t\rightarrow\infty}\dfrac{1}{\kappa t}\ln \inf\limits_{\widetilde{c}\in B_{\delta}(c)} P^{(\vec{p}, \vec{z})}(A) \ .
\end{equation}

Since $c\in K$ and $c>c^*$ is chosen such that $cI\left(\dfrac{1}{c}\right)-\beta>0$, by Lemma \ref{Lemma:WaveSpeedOutOfWaveFront}
we see that there is a $\delta>0$ sufficiently small so that for any $h\in (0,1)$ there is a constant $t_0>0$ depending on $h$
such that
$$v(t, c't)\leq h \text{ for all } c'\in B_{6\delta}(c) \text{ and all } t\geq t_0 \ .$$

Now if $0<\kappa<\dfrac{1}{2}$ and for any $\widetilde{c}\in B_\delta(c)$ we have
\begin{equation}\label{Lemma:WaveSpeedInsideWaveFront:Eq:RestrictionYCenter(t-s)c}
\sup\limits_{s\in [0, \kappa t]}|Y^{\widetilde{c}t}_s-(t-s)c|\leq 3\delta t \ ,
\end{equation}
then \eqref{Lemma:WaveSpeedInsideWaveFront:Eq:SetOfPathsACondition-u} is achieved along such paths when $t>2t_0$.

Next, if $\widetilde{c}\in B_{\delta}(c)$ is written as $\widetilde{c}=c+\Delta_1$ for $|\Delta_1|<\delta$, then define $\hat{c}=c+2\Delta_1$, and for any
$\Delta_2$ with $|\Delta_2|<\delta$ we have
\begin{equation}\label{Lemma:WaveSpeedInsideWaveFront:Eq:KeyConditionTechnial}
\widetilde{c}t-\kappa t\hat{c}+\kappa t \Delta_2\in B_{(1-\kappa)\delta t}((1-\kappa) ct)
\end{equation}
when $\kappa\in \left(0, \dfrac{1}{3}-\dfrac{|\widetilde{c}-c|}{3\delta}\right)$ is sufficiently small. Indeed
$$
\left(\widetilde{c}t-\kappa t \hat{c}+\kappa t \Delta_2\right)-(1-\kappa)ct=
t\left[\Delta_1-\kappa (2\Delta_1-\Delta_2)\right] \ .
$$
We see from here that $-(1-\kappa)\delta<\Delta_1-\kappa(2\Delta_1-\Delta_2)<(1-\kappa)\delta$ ensures \eqref{Lemma:WaveSpeedInsideWaveFront:Eq:KeyConditionTechnial}.
This reduces to $\kappa\left(\delta-(2\Delta_1-\Delta_2)\right)<\delta-\Delta_1$ and $\kappa\left(\delta+(2\Delta_1-\Delta_2)\right)<\delta+\Delta_1$.
Since $-3\delta<2\Delta_1-\Delta_2<3\delta$, we see \eqref{Lemma:WaveSpeedInsideWaveFront:Eq:KeyConditionTechnial} is guaranteed if
$0<\kappa<\dfrac{\delta-|\Delta_1|}{3\delta}=\dfrac{1}{3}-\dfrac{|\widetilde{c}-c|}{3\delta}$.

This ensures that for each $\widetilde{c}\in B_\delta(c)$
there is a $\hat{c}\in B_{2\delta}(c)$ such that \eqref{Lemma:WaveSpeedInsideWaveFront:Eq:SetOfPathsACondition-Y} is achieved
whenever
\begin{equation}\label{Lemma:WaveSpeedInsideWaveFront:Eq:RestrictionRescaledYDifferenceBallHatc}
\dfrac{\widetilde{c}t-Y^{\widetilde{c}t}_{\kappa t}}{\kappa t}\in B_\delta(\hat{c}) \ .
\end{equation}

Therefore by \eqref{Lemma:WaveSpeedInsideWaveFront:Eq:RestrictionYCenter(t-s)c} and \eqref{Lemma:WaveSpeedInsideWaveFront:Eq:RestrictionRescaledYDifferenceBallHatc} we can estimate
\begin{equation}\label{Lemma:WaveSpeedInsideWaveFront:Eq:LowerBoundP(A)}
\begin{array}{ll}
& \inf\limits_{\widetilde{c}\in B_\delta(c)}P^{(\vec{p}, \vec{z})}(A)
\\
\geq & \inf\limits_{\hat{c}\in B_{2\delta}(c), \widetilde{c}\in B_{\delta}(c)} P^{(\vec{p}, \vec{z})}\left(\sup\limits_{s\in [0, \kappa t]}|Y^{\hat{c}t}_s-(t-s)c|\leq 3\delta t \text{ and } \dfrac{\widetilde{c}t-Y^{\widetilde{c}t}_{\kappa t}}{\kappa t}\in B_\delta(\hat{c})\right) \ .
\end{array}
\end{equation}

For $\kappa\in \left(0, \dfrac{2\delta}{3\max(1,c)}\right)$ we see that
$$\begin{array}{ll}
 \sup\limits_{\hat{c}\in B_{2\delta}(c)} P^{(\vec{p}, \vec{z})}\left(\sup\limits_{s\in [0, \kappa t]}|Y^{\hat{c}t}_s-(t-s)c|> 3\delta t\right)
& \leq  \sup\limits_{\hat{c}\in B_{2\delta}(c)} P^{(\vec{p}, \vec{z})}\left(\sup\limits_{s\in [0, \kappa t]}|Y^{\hat{c}t}_s-\hat{c}t|> \dfrac{\delta t}{3}\right)
\\
& \leq  \sup\limits_{\hat{c}\in B_{2\delta}(c)} P^{(\vec{p}, \vec{z})}\left(T^{\hat{c}t}_{(\hat{c}-\delta/3)t}\wedge T^{\hat{c}t}_{(\hat{c}+\delta/3)t} <\kappa t\right) \ .
\end{array}$$

By Lemma \ref{Lemma:ExponentialTightnessHittingTimePositiveNegative}, for any $M>0$ we
can pick $\kappa$ sufficiently small so that 
$$
\limsup\limits_{t\rightarrow\infty} \dfrac{1}{\kappa t}\ln \sup\limits_{\hat{c}\in B_{2\delta}(c)} P^{(\vec{p}, \vec{z})}\left(T^{\hat{c}t}_{(\hat{c}-\delta/3)t}\wedge T^{\hat{c}t}_{(\hat{c}+\delta/3)t} <\kappa t\right)\leq -M \ , $$
That is,
\begin{equation}\label{Lemma:WaveSpeedInsideWaveFront:Eq:LowerBoundP(A)YPart}
\limsup\limits_{t\rightarrow\infty} \dfrac{1}{\kappa t}\ln \sup\limits_{\hat{c}\in B_{2\delta}(c)} P^{(\vec{p}, \vec{z})}\left(\sup\limits_{s\in [0, \kappa t]}|Y^{\hat{c}t}_s-(t-s)c|> 3\delta t\right)\leq -M \ .
\end{equation}

Combining \eqref{Lemma:WaveSpeedInsideWaveFront:Eq:LowerBound-q}, \eqref{Lemma:WaveSpeedInsideWaveFront:Eq:LowerBoundP(A)} and
\eqref{Lemma:WaveSpeedInsideWaveFront:Eq:LowerBoundP(A)YPart} we see that
\begin{equation}\label{Lemma:WaveSpeedInsideWaveFront:Eq:LowerBound-q-LDP}
q\geq \beta(1-h)+\liminf\limits_{t\rightarrow\infty}\dfrac{1}{\kappa t}\inf\limits_{\hat{c}\in B_{2\delta}(c), \widetilde{c}\in B_{\delta}(c)}
P^{(\vec{p}, \vec{z})}\left(\dfrac{\widetilde{c}t-Y^{\widetilde{c}t}_{\kappa t}}{\kappa t}\in B_\delta(\hat{c})\right) \ .
\end{equation}

Set $h>0$ and $\delta>0$ sufficiently small. By part (1) of Lemma \ref{Lemma:SolutionPropertyWaveSpeedFormula} we see that
$B_\delta(\hat{c})\subset B_{3\delta}(c)\subset [(\mu'(0))^{-1}, \infty)$. Thus we can apply
the estimate \eqref{Theorem:LDPSkewBM:Eq:UpperBoundPositive} in Theorem \ref{Theorem:LDPSkewBM} and we see that \eqref{Lemma:WaveSpeedInsideWaveFront:Eq:LowerBound-q-LDP}
gives \eqref{Lemma:WaveSpeedInsideWaveFront:Eq:LowerBoundEquivalentBallConcentrated}.
\end{proof}

The following Lemma helps to prove Lemma \ref{Lemma:WaveSpeedInsideWaveFront}.
\begin{lemma}\label{Lemma:Finiteness-q-WaveSpeedInsideWaveFront}
Suppose Assumption \ref{Assumption:ReactionRateBetaLarge} holds. For any bounded set $\Lambda\subset (c^*,\infty)$ and any small $\delta>0$,
there is a finite constant $K_1>0$ such that
\begin{equation}\label{Lemma:Finiteness-q-WaveSpeedInsideWaveFront:Eq:Bound-hatq}
\liminf\limits_{t\rightarrow\infty}\dfrac{1}{t}\ln \left(\inf\limits_{y\in B_\delta(tc)}P^{(\vec{p}, \vec{z})}
\left(Y^y_t\in B_\delta(0)\right)\right)>-K_1
\end{equation}
uniformly over all $c\in \Lambda$ such that $B_\delta(c)\subset(c^*,\infty)$.
\end{lemma}

\begin{proof}
Due to continuity of $Y_t$, for any trajectory of $Y_t^y$ starting from $y=ct+\delta$ and hitting $-\delta$ before time $t$, there must exist a piece of this trajectory that starts from some $y\in B_\delta(ct)$ and ends in
$B_\delta(0)$ before time $t$. This gives us the event inclusion
$\{T^{ct+\delta}_{-\delta}\leq t\}\subseteq \{Y^y_t\in B_\delta(0), y\in B_\delta(ct)\}$, which implies that
$$\inf\limits_{y\in B_\delta(tc)}P^{(\vec{p}, \vec{z})}
\left(Y^y_t\in B_\delta(0)\right)\geq P^{(\vec{p}, \vec{z})}\left(\dfrac{T^{ct+\delta}_{-\delta}}{t}\leq 1\right) \ .$$
This further implies that
$$\begin{array}{ll}
\liminf\limits_{t\rightarrow\infty}\dfrac{1}{t}\ln \left(\inf\limits_{y\in B_\delta(tc)}P^{(\vec{p}, \vec{z})}
\left(Y^y_t\in B_\delta(0)\right)\right)
& \geq \liminf\limits_{t\rightarrow\infty}\dfrac{1}{t}\ln P^{(\vec{p}, \vec{z})}\left(\dfrac{T^{ct+\delta}_{-\delta}}{t}\leq 1\right)
\\
& \stackrel{(*)}{=} \liminf\limits_{t\rightarrow\infty}\dfrac{1}{t}\ln P^{(\vec{p}, \vec{z})}\left(\dfrac{T^{ct}_0}{t}\leq 1\right) \ .
\end{array}$$
Here $(*)$ is due to the fact that $T^{ct+\delta}_{-\delta}=T^{ct+\delta}_{ct}+T^{ct}_0+T^0_{-\delta}$, and that
$\lim\limits_{t\rightarrow\infty}\dfrac{T^{ct+\delta}_{ct}+T^0_{-\delta}}{t}=0$ holds $P^{(\vec{p}, \vec{z})}$-almost surely.
Setting $v=c$ and $c=0$ in
\eqref{Theorem:LDPHittingTime:Eq:LowerBound} of Theorem \ref{Theorem:LDPHittingTime}, we obtain
$$\begin{array}{ll}
\liminf\limits_{t\rightarrow\infty}\dfrac{1}{t}\ln P^{(\vec{p}, \vec{z})}\left(\dfrac{T^{ct}_0}{t}\leq 1\right)
&> \liminf\limits_{t\rightarrow\infty}\dfrac{1}{t}\ln P^{(\vec{p}, \vec{z})}\left(\dfrac{T^{ct}_0}{t}\in (0, c\mu'(0))\right)
\\
&\geq -c\inf\limits_{a\in (0, c\mu'(0))}I\left(\dfrac{a}{c}\right) \equiv -K_1 \ ,
\end{array}$$
so that \eqref{Lemma:Finiteness-q-WaveSpeedInsideWaveFront:Eq:Bound-hatq} follows.
\end{proof}

\section{Variational formula for the speed 
} \label{Sec:Algorithm-Examples}

Theorem \ref{Theorem:WaveSpeed} indicates that to compute the speed $c^*$ in terms of the degrees $(d_i)$ and the branch lengths $(\ell_i)$, we need to solve the equation
\eqref{Eq:WaveSpeedFormula}, i.e.
$c^* I\Big(\dfrac{1}{c^*}\Big)=\beta$
for $c^*>0$ (assuming Assumption \ref{Assumption:ReactionRateBetaLarge}).

By part (1) of Lemma \ref{Lemma:SolutionPropertyWaveSpeedFormula}, we see that $c^*>(\mu'(0))^{-1}$, i.e.,
$0<\dfrac{1}{c^*}<\mu'(0)$. Thus $\sup\limits_{\eta\leq \eta_c}\left(\dfrac{1}{c^*}\eta-\mu(\eta)\right)$
is achieved at a point $\eta\leq 0$ due to part (1) of Lemma \ref{Lemma:PropertiesMu}, saying that $\mu'(\eta)$ is strictly monotonically increasing in $\eta$. This implies that $I\left(\dfrac{1}{c^*}\right)=
\sup\limits_{\eta\leq 0}\left(\dfrac{1}{c^*}\eta-\mu(\eta)\right)$. Thus we have

$$c^*I\left(\dfrac{1}{c^*}\right)=c^*\sup\limits_{\eta \leq 0}\left(\dfrac{1}{c^*}\eta-\mu(\eta)\right)
= \sup\limits_{\eta\leq 0}(\eta-c^*\mu(\eta))=\beta \ .$$

This gives us

\begin{equation}\label{Eq:VariationalFormulaWaveSpeedInTermsMu}
c^*=\inf\limits_{\eta\leq 0}\dfrac{\eta-\beta}{\mu(\eta)}=\inf\limits_{\lambda\geq 0}\dfrac{\lambda+\beta}{|\mu(-\lambda)|} \ .
\end{equation}

Here we have used the fact that $\mu(\eta)\leq 0$ for $\eta \leq 0$ (part (3) of Lemma \ref{Lemma:PropertiesMu}). Equation \eqref{Eq:VariationalFormulaWaveSpeedInTermsMu} provides
a \textit{variational formula} for the wave speed in terms of the Lyapunov function $\mu(\eta)$ that we introduced in
\eqref{Theorem:LyapunovExponentPositiveDirection:Eq:LyapunovExponentIdentity}.
Using \eqref{Eq:VariationalFormulaWaveSpeedInTermsMu},we obtain in the following theorem that gives
the variational formula for the wave speed $c^*$ in terms of $\vec{d}$ and $\vec{\ell}$.

\begin{theorem}[variational formula for the wave speed on $\mathbb{T}_{\vec{d}, \vec{\ell}}$]\label{Theorem:VariationalFormulaWaveSpeed}
Assuming Assumption \ref{Assumption:ReactionRateBetaLarge}. The wave speed $c^*$ for the system \eqref{Eq:FisherKPPInitialBoundaryConditionVertices} on $\mathbb{T}_{\vec{d}, \vec{\ell}}$
in the sense of Definition \ref{Def:WaveSpeedTree} is given by
\begin{equation}\label{Theorem:VariationalFormulaWaveSpeed:Eq:VariationalFormulaWaveSpeed}
c^*=\inf\limits_{\lambda\geq 0}\left\{\dfrac{\lambda+\beta}{\sqrt{2\lambda}+\dfrac{1}{\mathbf{E} \ell_0}\mathbf{E}\left[\ln\left(1+\dfrac{1-e^{-2\sqrt{2\lambda}\ell_0}}{\xi_\lambda-1}\right)\right]}\right\} \ ,
\end{equation}
where $\xi_{\lambda}\in [1,\infty)$ is given by Theorem \ref{Theorem:Existence_xi}.
In particular, $c^*\leq \sqrt{2\beta}$ with the equality achieved if and only if the tree
$\mathbb{T}_{\vec{d}, \vec{\ell}}$ degenerates to $\mathbb{R}$.
\end{theorem}

\begin{proof}
By \eqref{Lemma:PropertiesMu:Eq:ExplicitCalculation-mu-eta} in Lemma \ref{Lemma:PropertiesMu}
we can calculate $\mu(-\lambda)$ in terms of $\vec{d}$ and $\vec{\ell}$:
\begin{equation}\label{Eq:muOfMinusLambda}
\mu(-\lambda)=-\sqrt{2\lambda}+\dfrac{1}{\mathbf{E} \ell_0}\mathbf{E}\left(\ln\dfrac{\xi_\lambda-1}{\xi_\lambda-e^{-2\sqrt{2\lambda}\ell_0}}\right) \ .
\end{equation}

Since $\xi=\xi_\lambda\geq 1$, we further see that

\begin{equation}\label{Eq:AbsloteValuemuOfMinusLambda}
|\mu(-\lambda)|=\sqrt{2\lambda}+\dfrac{1}{\mathbf{E} \ell_0}\mathbf{E}\left[\ln\left(1+\dfrac{1-e^{-2\sqrt{2\lambda}\ell_0}}{\xi_\lambda-1}\right)\right] \ .
\end{equation}

Formula \eqref{Theorem:VariationalFormulaWaveSpeed:Eq:VariationalFormulaWaveSpeed} is an easy consequence of
\eqref{Eq:VariationalFormulaWaveSpeedInTermsMu} and \eqref{Eq:AbsloteValuemuOfMinusLambda}.
We first demonstrate how  \eqref{Theorem:VariationalFormulaWaveSpeed:Eq:VariationalFormulaWaveSpeed} gives the asymptotic speed $c^*$
for the FKPP equation on $\mathbb{R}$,
\begin{equation}\label{Eq:ExampleFKPPWaveSpeed-FKPP}
\dfrac{\partial u}{\partial t} =\dfrac{1}{2}\dfrac{\partial^2 u}{\partial x^2} +\beta u(1-u) \ .
\end{equation}
In this case $\mathbb{T}_{\vec{d}, \vec{\ell}}$ degenerates to $\mathbb{R}$
and all $p_i=\dfrac{1}{2}$. Thus $\xi=+\infty$ by Theorem \ref{Theorem:Existence_xi}, Corollary \ref{Corollary:xiBoundedAboveExpectation}
and Remark \ref{Remark:PropertyXi}.
By \eqref{Theorem:VariationalFormulaWaveSpeed:Eq:VariationalFormulaWaveSpeed},
\begin{equation}\label{Eq:ExampleFKPPWaveSpeed-FKPPWaveSpeed}
c^*_{\mathbb{R}}=\inf\limits_{\lambda\geq 0}\dfrac{\lambda+\beta}{\sqrt{2\lambda}}=\sqrt{2\beta} \ .
\end{equation}

Consider the general non-degenerate tree $\mathbb{T}_{\vec{d}, \vec{\ell}}$ case.
Using the elementary inequality $\ln(1+x)\geq \dfrac{x}{1+x}$ for all $x>0$, we can estimate
$$\begin{array}{ll}
 \dfrac{1}{\mathbf{E} \ell_0}
\mathbf{E}\left[\ln\left(1+\dfrac{1-e^{-2\sqrt{2\lambda}\ell_0}}{\xi_\lambda-1}\right)\right]
& \geq  \dfrac{1}{\overline{\ell}}\mathbf{E}\left[\dfrac{\dfrac{1-e^{-2\sqrt{2\lambda}\ell_0}}{\xi_\lambda-1}}
{1+\dfrac{1-e^{-2\sqrt{2\lambda}\ell_0}}{\xi_\lambda-1}}\right]
 = \dfrac{1}{\overline{\ell}}\mathbf{E}\left[\dfrac{1-e^{-2\sqrt{2\lambda}\ell_0}}
{\xi_\lambda-e^{-2\sqrt{2\lambda}\ell_0}}\right]
\\
& \geq \dfrac{1-e^{-2\sqrt{2\lambda}\underline{\ell}}}{\overline{\ell}}\mathbf{E}\left[\dfrac{1}
{\xi_\lambda-e^{-2\sqrt{2\lambda}\ell_0}}\right]
\\
& \geq
\dfrac{1-e^{-2\sqrt{2\lambda}\underline{\ell}}}{\overline{\ell}}\mathbf{E}
\left[\dfrac{1}{\xi_\lambda}\right]>0
\end{array}$$
by Corollary \ref{Corollary:xiBoundedAboveExpectation}. Since the infinimum in \eqref{Theorem:VariationalFormulaWaveSpeed:Eq:VariationalFormulaWaveSpeed} is taken at some $\lambda\leq \sqrt{2\beta}$, it ensures that we have
\begin{equation}\label{Eq:ExampleFKPPWaveSpeedTree-SlowDown}
c^*_{\mathbb{T}_{\vec{d}, \vec{\ell}}}\,=\,\inf\limits_{\lambda\geq 0}\dfrac{\lambda+\beta}{\sqrt{2\lambda}+\dfrac{1}{\mathbf{E} \ell_0}\mathbf{E}\left[\ln\left(1+\dfrac{1-e^{-2\sqrt{2\lambda}\ell_0}}{\xi_\lambda-1}\right)\right]}
\,<\, \inf\limits_{\lambda\geq 0}\dfrac{\lambda+\beta}{\sqrt{2\lambda}}=\sqrt{2\beta}=c^*_{\mathbb{R}} \ ,
\end{equation}
i.e., the wave speed is strictly slower than the one on $\mathbb{R}$.
\end{proof}

\begin{remark}[heuristic reason for the slow down of wave speed]\rm\label{Remark:HeuristicReasonSlowDownWave}
The slow down of the wave speed on $\mathbb{T}_{\vec{d}, \vec{\ell}}$ can be heuristically explained.
The seemingly very complicated arguments that we employed in Section \ref{Sec:WavePropagation}
which lead to the existence of the wavefront is essentially based on an analysis of \eqref{Lemma:proj:Eq:FeynmanKac-v}:
$$v(t,y)=E^{(\vec{d},\vec{\ell})}_y\Big[v_0(Y_t)\exp\Big\{\beta\int^t_0 \Big(1-v(t-s,Y_s)\Big)ds\Big\} \Big] \ .$$

From this equation we see that for those regions of $y=ct$
that the value of $v(t, y)$ is small (indeed not close $1$), the reaction term $f(u)=\beta u(1-u)$ will be creating an exponential
birth of the particles at a rate of $\beta$, i.e., an $e^{\beta t}$ factor
in the solution $v(t, y)$ in \eqref{Lemma:proj:Eq:FeynmanKac-v}.
However, this exponential term $\exp\Big\{\beta\displaystyle{\int^t_0} \Big(1-v(t-s,Y_s)\Big)ds\Big\}$ is
multiplied by $v_0(Y_t)=\mathbf{1}_{(-\delta, \delta)}(Y_t)$,
the expectation of which is given by the large deviations principle of $Y_t$ at a rate of
$-|c|I(\frac{1}{|c|})t$, i.e. an $e^{-|c|I(\frac{1}{|c|})t}$ factor in the solution $v(t, y)$ in \eqref{Lemma:proj:Eq:FeynmanKac-v}.
The competition  between these two effects,
namely the exponential growth due to reaction and the large deviation effect due to diffusion,
results in the fact that the wavefront speed $c^*$ is formed by the equation
$c^*I\left(\dfrac{1}{c^*}\right)=\beta$. This is to say that the traveling speed $c^*$ (or $-c^*$) to the direction of the wave propagation
should be a speed so that, when travelling at this speed, the rate of coming back to $(-\delta, \delta)$ (the large deviations rate)
equals the birth rate $\beta$. In our case, the local time term in the multi-skewed process $Y_t$ from the stochastic differential equation \eqref{Eq:MultiSkewBM-SDE} can be viewed as providing
a drift that directs towards the direction of the wave propagation, which results in more difficulty for $Y_t$ to reach back
$(-\delta, \delta)$, i.e., larger large deviations rate $|c|I\left(\dfrac{1}{|c|}\right)$ for fixed speed $c$.
Noticing that $cI\left(\dfrac{1}{c}\right)$ is monotonically increasing when $c>c^*$ and increases,
for fixed $\beta> \max\left(\dfrac{-\mu(0)}{\mu'(0)}, \eta_c\right)$,
to satisfy $c^*I\left(\dfrac{1}{c^*}\right)=\beta$, the speed $c^*$ in our case should be \emph{slower} than the
bare line $\mathbb{R}$ case $\sqrt{2\beta}$ as we see in \eqref{Eq:ExampleFKPPWaveSpeedTree-SlowDown}.
\end{remark}

The slow down of the wave speed can be quantitatively estimated from $(d_i)$ and $(\ell_i)$ using our calculations in
Corollary \ref{Corollary:xiBoundedBelow1Plusc}, formula \eqref{Lemma:PropertiesMu:Eq:ExplicitCalculation-mu-eta} in
Lemma \ref{Lemma:PropertiesMu} as well as the variational formula for the wave speed
\eqref{Theorem:VariationalFormulaWaveSpeed:Eq:VariationalFormulaWaveSpeed}. We have

\begin{corollary}\label{Corollary:SlowDown}
Under the same assumption of Theorem \ref{Theorem:VariationalFormulaWaveSpeed},
\begin{equation}\label{Corollary:SlowDown:Eq:EstimateOfSlowDown}
0\leq c^*_{\mathbb{R}}-c^*_{\mathbb{T}_{\vec{d}, \vec{\ell}}}< \sqrt{2\beta}-\inf\limits_{\lambda\geq 0}\dfrac{\lambda+\beta}{\sqrt{2\lambda}+\dfrac{1}{\underline{\ell}}\ln\left(1+\dfrac{\overline{d}e^{2\sqrt{2\sqrt{2\beta}}\cdot\overline{\ell}}}{2}
\cdot\dfrac{e^{4\sqrt{2\lambda}\cdot\overline{\ell}}-1}
{e^{2\sqrt{2\lambda}\underline{\ell}}-1}\right)} \ .
\end{equation}
\end{corollary}

\begin{proof}
By Corollary \ref{Corollary:xiBoundedBelow1Plusc}
we can estimate
$$\begin{array}{ll}
\mathbf{E}\left[\ln\left(1+\dfrac{1-e^{-2\sqrt{2\lambda}\ell_0}}{\xi_\lambda-1}\right)\right]
& \leq 
\ln\left(1+\dfrac{1-e^{-2\sqrt{2\lambda}\cdot\overline{\ell}}}
{2(\overline{d})^{-1}\dfrac{e^{2\sqrt{2\lambda}\underline{\ell}}-1}{e^{2\sqrt{2\lambda}\cdot\overline{\ell}}+1}}\right)
\\
\\
& =
\ln\left(1+\dfrac{\overline{d}}{2}\dfrac{e^{2\sqrt{2\lambda}\cdot\overline{\ell}}[e^{4\sqrt{2\lambda}\cdot\overline{\ell}}-1]}
{e^{2\sqrt{2\lambda}\underline{\ell}}-1}\right)  \ .
\end{array}$$

We see that in \eqref{Theorem:VariationalFormulaWaveSpeed:Eq:VariationalFormulaWaveSpeed}, the inf is taken at the point $\lambda=c^*$, and further by
\eqref{Eq:ExampleFKPPWaveSpeedTree-SlowDown}, we have $\lambda=c^*=c^*_{\mathbb{T}_{\vec{d}, \vec{\ell}}}<\sqrt{2\beta}$.
So we  further have
$$\begin{array}{ll}
\mathbf{E}\left[\ln\left(1+\dfrac{1-e^{-2\sqrt{2\lambda}\ell_0}}{\xi_\lambda-1}\right)\right]
 < 
\ln\left(1+\dfrac{\overline{d}e^{2\sqrt{2\sqrt{2\beta}}\cdot\overline{\ell}}}{2}
\cdot\dfrac{e^{4\sqrt{2\lambda}\cdot\overline{\ell}}-1}
{e^{2\sqrt{2\lambda}\underline{\ell}}-1}\right) \ .
\end{array}$$

Therefore by \eqref{Theorem:VariationalFormulaWaveSpeed:Eq:VariationalFormulaWaveSpeed} we see that
$$\begin{array}{ll}
c^*_{\mathbb{T}_{\vec{d}, \vec{\ell}}} & =\inf\limits_{\lambda\geq 0}\dfrac{\lambda+\beta}{\sqrt{2\lambda}+\dfrac{1}{\mathbf{E} \ell_0}\mathbf{E}\left[\ln\left(1+\dfrac{1-e^{-2\sqrt{2\lambda}\ell_0}}{\xi_\lambda-1}\right)\right]}
\\
\\
& >   \inf\limits_{\lambda\geq 0}\dfrac{\lambda+\beta}{\sqrt{2\lambda}+\dfrac{1}{\underline{\ell}}\ln\left(1+\dfrac{\overline{d}e^{2\sqrt{2\sqrt{2\beta}}\cdot\overline{\ell}}}{2}
\cdot\dfrac{e^{4\sqrt{2\lambda}\cdot\overline{\ell}}-1}
{e^{2\sqrt{2\lambda}\underline{\ell}}-1}\right)}
 \ ,
\end{array}$$
which gives the upper bound \eqref{Corollary:SlowDown:Eq:EstimateOfSlowDown} on the magnitude of the slow down of the wave speed on $\mathbb{T}_{\vec{d}, \vec{\ell}}$ compared to $\mathbb{R}$.
\end{proof}

When $\mathbb{T}_{\vec{d}, \vec{\ell}}$ is
deterministic with two identical $d$-regular trees attaching to the root, the asymptotic wave speed is more explicit. See Corollary \ref{Cor:ConstantCase} and Figure \ref{Fig:ConstantCase} below.
\begin{corollary}[Constant-$(d,\ell)$ tree]\label{Cor:ConstantCase}
Suppose there exist deterministic constants $d>2$ and $\ell\in(0,\infty)$ such that
$d_i=d$ and $\ell_i=\ell_0=\ell$ for all $i\geq 1$. Let $p=\frac{d-1}{d} \in (0,1)$. Then
\begin{equation}\label{betacConstantCase}
\beta_c=\frac{-\mu(0)}{\mu'(0)}=\frac{2p-1}{\ell}\ln \left(\frac{p}{1-p}\right)=\frac{d-2}{\ell\,d}\ln \left(d-1\right)  
\end{equation}
and for $\beta\in (\beta_c,\infty)$, the   asymptotic speed is given by \eqref{SpeedConstantCase}.
Furthermore,  $\lim\limits_{\beta\to\infty}\frac{c^*}{\sqrt{2\beta}}=1$, $\lim\limits_{d\to\infty}c^*=0$ and $\lim\limits_{d\to\infty}\lim\limits_{\beta\downarrow \beta_c}c^*=1$ for $\ell>0$.
\end{corollary}

\begin{remark}\rm
The last assertion raises a curious point: if the reaction rate is maintained at the critical reaction rate $\beta_c$, the speed is bounded even if  the degree $d\to\infty$. On other hand, the LDP rate function  falls into case (c-2) in Figure \ref{Fig:I-a-Graph}.
\end{remark}

\begin{proof}
Recall $J^i_{\eta,+1}$ defined in \eqref{Eq:Def:J_eta}.
Basic stochastic calculus gives
$J^i_{\eta,+1} = \frac{p}{\cos(\sqrt{2\eta}\ell)}:=J_{\eta,+}$ for all $i\geq 1$ and $\eta\in\left(0,\frac{\pi^2}{8\ell^2}\right)$.
From this we get $\eta_c=\frac{1}{2\ell^2}\arccos^2\left(2\sqrt{p(1-p)}\right)$ and $w_{-\eta_c}(1)=\frac{1}{2 J_{\eta_c,+}}\in (0, 1)$.
Besides,
\begin{align}
\lim_{\lambda\downarrow 0}w_{\lambda}(\ell)=\frac{1-p}{p}\in(0,1) \quad\text{and}\quad
\lim_{\lambda\downarrow 0}\frac{d w_{\lambda}(\ell)}{d\lambda} = -\ell^2\frac{1-p}{p(2p-1)}\in (-\infty,0). \label{w0_v2}
\end{align}
So from $\mu(\eta)=\dfrac{\ln w_{-\eta}(\ell)}{\ell}$
and \eqref{w0_v2}  we have
$$\frac{-\mu(0)}{\mu'(0)}=\frac{-1}{\ell}\ln \left(\frac{1-p}{p}\right)(2p-1) \,>\,\eta_c,$$
giving \eqref{betacConstantCase}.

By solving \eqref{Corollary:InvarDistX:Eq:InvariantX} we obtain
    \begin{equation*}
        \xi_{\lambda}=\frac{2\zeta}{\sqrt{(\gamma^2-1)^2+4\zeta^2\gamma^2}+1-\gamma^2} =\frac{\sqrt{(\gamma^2-1)^2+4\zeta^2\gamma^2}+\gamma^2-1}{2\zeta \gamma^2},
    \end{equation*}
where $\zeta=2p-1=\frac{d-2}{d}$ and $\gamma:=e^{\ell\sqrt{2\lambda}}$. The formula of $c^*$ now follows from \eqref{Theorem:VariationalFormulaWaveSpeed:Eq:VariationalFormulaWaveSpeed}.

\medskip

Formula \eqref{SpeedConstantCase} allows further explicit calculations using calculus. 
View $\ell\in (0,\infty)$ as fixed always and
write $\Phi(\beta,p,\lambda)$ as the function after the infinimum.
For $(\beta,p)\in(0,\infty)\times (1/2,1)$, there is a unique positive number $\lambda_{\beta,p}$ at which infinmum on the right of \eqref{SpeedConstantCase} is obtained. That is,
\begin{equation}\label{SpeedConstantCase2}
c^*=\inf\limits_{\lambda\geq 0}\,\Phi(\beta,p,\lambda)=\Phi(\beta,p,\lambda_{\beta,p}).
\end{equation}

The function of two variables $c^*=c^*(\beta,p)$ is continuous on $(0,\infty)\times(1/2,1)$.
It can be checked that
for fixed $p\in(1/2,1)$, i.e.  fixed degree $d$, the mapping $\beta\mapsto c^*$ is increasing and $\lim\limits_{\beta\to\infty}\frac{\lambda_{\beta,p}}{\beta}=1$. From the latter we obtain  $\lim\limits_{\beta\to\infty}\frac{c^*}{\sqrt{2\beta}}=1$ from \eqref{SpeedConstantCase}.

We further choose $\beta$ to be the critical $\beta_c$ in \eqref{betacConstantCase} and consider the ``speed at critical", $c^*(\beta_c,p)$. 
As $p\to 1$, we have $\beta_c\to\infty$ and $\frac{\lambda_{\beta_c,p}}{\beta_c}\to 0$. 
 From the latter we obtain   $\lim_{p\to 1}c^*(\beta_c,p)=1$ from \eqref{SpeedConstantCase}. To see this, from \eqref{SpeedConstantCase} we have
\begin{align*}
  \lim_{p\to 1}c^*(\beta_c,p)
  =& \lim_{p\to 1} \dfrac{1}{\dfrac{1}{ \ell\,\beta_c}\ln\left(\dfrac{4p}{1+\gamma^2_c-\sqrt{(\gamma^2_c-1)^2+4(2p-1)^2\gamma^2_c}}\right)}
\end{align*}
where $\gamma_c=e^{\ell\sqrt{2\,\lambda_{\beta_c,p}}}$. Note that $1-p=\frac{1}{d}$ decays linearly in $d$ and $\beta_c$ given by \eqref{betacConstantCase} grows like $\frac{\ln d}{\ell}$. From these and the fact that $\lim\limits_{d\to\infty}\frac{1}{d}e^{C\sqrt{\ln d}}=0$ for all $C\in(0,\infty)$, we obtain $\lim\limits_{p\to 1}(1-p)\gamma_c^2=0$ and
 \begin{align*}
   \lim_{p\to 1}\frac{1}{c^*(\beta_c,p)}=& \,\frac{1}{\ell}\,\lim_{p\to 1}\frac{-\ln\left(1+\gamma^2_c-\sqrt{(\gamma^2_c-1)^2+4(2p-1)^2\gamma^2_c}\right)}{\beta_c}\\
   =& \,\lim_{d\to \infty}\frac{\ln\left(1+\gamma^2_c-\sqrt{(\gamma^2_c+1)^2-16p(1-p)\gamma^2_c}\right)}{-\ln d}\\
    =& \,\lim_{d\to \infty}\frac{\ln\left(\frac{16p(1-p)\gamma^2_c}{1+\gamma^2_c+\sqrt{(\gamma^2_c+1)^2-16p(1-p)\gamma^2_c}}\right)}{-\ln d}\\
    =& \,\lim_{d\to \infty}\frac{\ln\left(\frac{8\gamma^2_c}{(1+\gamma^2_c)d}\right)}{-\ln d}\\
   =& \,1.
\end{align*}
\end{proof}

\begin{figure}
\centering
\includegraphics[height=4cm, width=15cm]{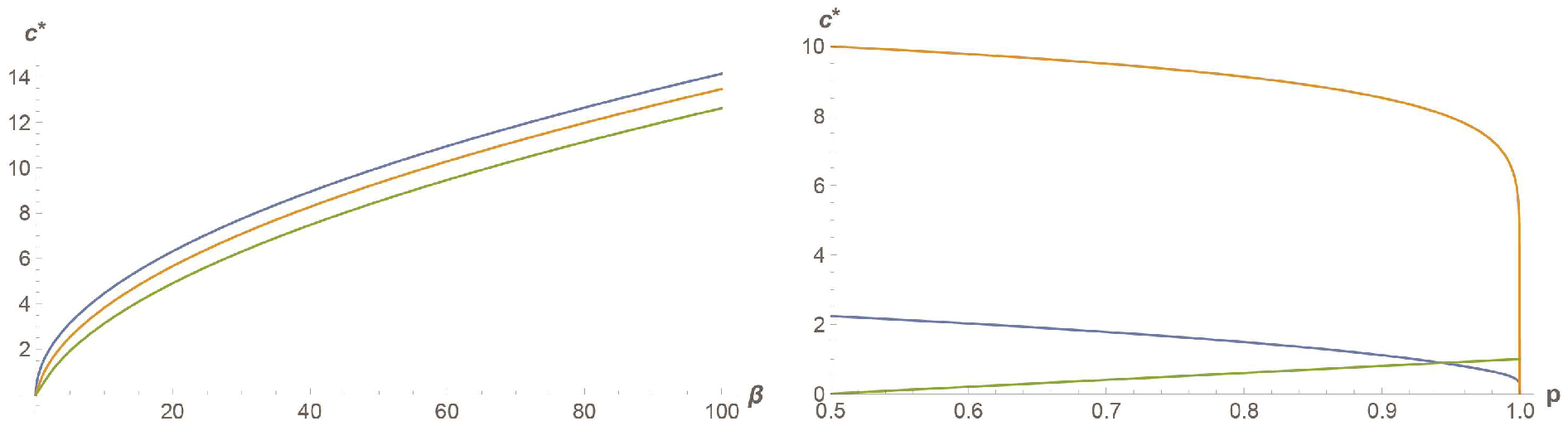}
\caption{[Left panel] Asymptotic speed $c^*$ versus reaction rate $\beta$ for the constant-$(d,\ell)$ tree. Top curve (blue) is for $\mathbb{R}$, i.e. $d=2$, so it is exactly $\sqrt{2\beta}$; Middle curve (orange) is for $(d,\ell)=(4,1)$; Bottom curve (green) is for $(d,\ell)=(10,1)$. Hence the speed decreases as $d$ increases.\;
[Right panel] $c^*$ versus $p=\frac{d-1}{d}$. Top curve (orange) is for $\beta=50$; Middle curve (blue) is for $\beta=2.5$; Bottom curve (green) is the ``critical curve" $p\mapsto c^*(\beta_c(p),p)$. 
}
\label{Fig:ConstantCase}
\end{figure}

\newpage

\bibliographystyle{plain}
\bibliography{wave-tree_bibliography}

\end{document}